\documentclass[reqno]{amsart}

\usepackage{amsfonts}
\usepackage{amsmath,amssymb,amsxtra}
\usepackage{float}
\usepackage[colorlinks,linkcolor=RoyalBlue,anchorcolor=Periwinkle, citecolor=Orange,urlcolor=Green]{hyperref}
\usepackage[usenames,dvipsnames]{xcolor}
\usepackage{enumitem}
\usepackage{geometry} 
\usepackage{graphicx}
\usepackage{subfigure}
\usepackage{tikz}\usetikzlibrary{matrix}
\usepackage{url}
\usepackage[all]{xy}

\newtheorem{theorem}{Theorem}[section]
\newtheorem{lemma}[theorem]{Lemma}
\newtheorem{corollary}[theorem]{Corollary}
\newtheorem{proposition}[theorem]{Proposition}
\newtheorem{definition}[theorem]{Definition}
\newtheorem{construction}[theorem]{Construction}
\newtheorem{remark}[theorem]{Remark}
\newtheorem{example}[theorem]{Example}
\newtheorem{notation}[theorem]{Notation}

\usetikzlibrary{arrows}

\numberwithin{equation}{section}

\def\BLUE{blue} 
\def\RED{red} 

\def\hua{\mathcal}

\def\<{\langle}
\def\>{\rangle}

\newcommand{\add}{\mathsf{add}\hspace{.01in}}%
\newcommand{\rep}{\mathsf{rep}\hspace{.01in}}%
\renewcommand{\mod}{\mathsf{mod}\hspace{.01in}}%
\newcommand{\Int}{\operatorname{Int}\nolimits}%
\newcommand{\Hom}{\operatorname{Hom}\nolimits}%
\newcommand{\End}{\operatorname{End}\nolimits}%
\newcommand{\Ext}{\operatorname{Ext}\nolimits}%
\newcommand{\im}{\operatorname{Im}\nolimits}%

\def\numbers{\begin{enumerate}[label=\arabic*{$^\circ$}.]}
\def\ends{\end{enumerate}}

\newcommand{\R}{\operatorname{\mathbf{R}}}
\newcommand{\CEG}{\operatorname{CEG}}             
\newcommand{\D}{\operatorname{\hua{D}}}
\newcommand{\per}{\operatorname{per}}

\def\arrow{red}
\usepackage{tikz,ifthen}

\newcommand{\EGT}{\operatorname{EG}^{\bowtie}}

\renewcommand\bowtie{\times}

\def\T{\mathbf{T}}

\def\CS{\mathbf{C}(\surf)}
\def\TA{{\mathbf{A}^\bowtie}(\surf)}
\def\TC{{\mathbf{C}^\bowtie}(\surf)}
\def\obj{\hua{I}^\bowtie(\surf)}

\def\QP{\Gamma(Q,W)}
\def\M{\mathbf{M}}
\def\P{\mathbf{P}}

\renewcommand{\k}{\mathbf{k}}
\def\m{\mathfrak{m}}

\def\n{\mathfrak{n}}
\def\r{\mathfrak{r}}

\def\C{\mathcal{C}}

\def\surf{\mathbf{S}} 
\newcommand\Sp{\operatorname{Sp}}
\newcommand\so{\operatorname{\mathfrak{S}}}

\newcommand{\yd}{\operatorname{[-]}}
\newcommand{\bt}{\operatorname{[\text{\tiny{+}}]}}
\usetikzlibrary{decorations.markings}
\tikzset{->-/.style={decoration={  markings,  mark=at position #1 with
    {\arrow{>}}},postaction={decorate}}}
\tikzset{-<-/.style={decoration={  markings,  mark=at position #1 with
    {\arrow{<}}},postaction={decorate}}}

\def\X{\mathfrak{X}}
\def\bX{\overline{\X}}
\def\tX{\widetilde{\X}}

\begin{document}

\title{Cluster categories for marked surfaces: punctured case}
\author{Yu Qiu}
\address{YQ:
Department of Mathematics,
Chinese University of Hong Kong,
Shatin,
N.T.,
Hong Kong}
\email{yu.qiu@bath.edu}
\author{Yu Zhou}
\address{YZ:
Department of Mathematical Sciences,
Norwegian University of Science and Technology,
N-7491,
Trondheim,
Norway}
\email{yuzh@math.ntnu.no}
\dedicatory{Dedicated to Idun Reiten on the occasion of her seventy-fifth birthday}
\thanks{This work is supported by Research Council of Norway, FRINAT grant number 231000.}


\begin{abstract}
    We study cluster categories arising from marked surfaces (with punctures and non-empty boundaries).
    By constructing skewed-gentle algebras,
    we show that there is a bijection between tagged curves and string objects. Applications include interpreting dimensions of $\Ext^1$ as intersection numbers of tagged curves and
    Auslander-Reiten translation as tagged rotation.
    An important consequence is that the cluster(-tilting) exchange graphs of such cluster categories are connected.
\end{abstract}

\maketitle

\section{Introduction}
\subsection{Overall}
Cluster algebras were introduced by Fomin-Zelevinsky \cite{FZ} around 2000,
with quiver mutation as the combinatorial aspect.
Derksen-Weyman-Zelevinsky(=DWZ) \cite{DWZ} further developed quiver mutation to mutation of quivers with potential. During the last decade,
the cluster phenomenon was spotted in various areas in mathematics, as well as in physics,
including geometric topology and representation theory.
On one hand, the geometric aspect of cluster theory was explored by
Fomin-Shapiro-Thurston (=FST) \cite{FST} after the pioneering work of Fock-Goncharov \cite{FG1,FG2}.
They constructed a quiver $Q_\T$ (and later Labardini-Fragoso \cite{LF,ILF} gave a corresponding potential $W_\T$)
from any (tagged) triangulation $\T$ of a marked surface $\surf$.
Moreover, they showed that mutation of quivers (with potential)
is compatible with flip of triangulations. There are a lot of known results about cluster algebras in the surface case:
\begin{itemize}
  \item Felikson-Shapiro-Tumarkin \cite{FST2} classified cluster algebras of finite mutation type, that they are all from marked surfaces expect for few cases.
  \item Musiker-Schiffler-Williams \cite{MSW} constructed two canonical bases by two types of collections of curves.
  \item Musiker-Schiffler-Williams \cite{MSW2}, Musiker-Williams \cite{MW} and Canakci-Schiffler \cite{CS1,CS2,CS3} gave combinatorial formulas for cluster variables and relations.
  \item Mills \cite{M} showed that there exist maximal green sequences for quivers with potential associated to triangulated marked surfaces except for once-punctured closed surfaces (cf. \cite{ACCERV}).
\end{itemize}

On the other hand, the categorification of cluster algebras leads to representations of quivers,
due to Buan-Marsh-Reineke-Reiten-Todorov \cite{BMRRT}.
Later, Amiot \cite{A} introduced generalized cluster categories via Ginzburg dg algebras associated to quivers with potential.
Then there is an associated cluster category $\C(\T)$ for each triangulation $\T$ of $\surf$.

Several works have been done concerning the cluster categories associated to
triangulations of surfaces.
Namely, for some special cases,
\begin{itemize}
  \item Caldero-Chapoton-Schiffler \cite{CCS} realized the cluster category of type $A_n$ by a regular polygon with $n+3$ vertices (i.e. a disk with $n+3$ vertices on its boundary).
  \item Schiffler \cite{Sch} realized the cluster category of type $D_n$ by a regular polygon with $n$ vertices and one puncture in the center (i. e. a disk with $n$ vertices on its boundary and one puncture in its interior).
\end{itemize}
In the unpunctured case,
\begin{itemize}
  \item Assem-Br\"{u}stle-Charbonneau Jodoin-Plamondon \cite{ABCP} proved that the Jacobian algebra of such a quiver with potential is a gentle algebra and gave a bijection between arcs that are not in the triangulation and string modules of the associated gentle algebras.
  \item Br\"{u}stle-Zhang(=BZ) \cite{BZ} generalised the bijection of \cite{ABCP} to a bijection between the set of curves and valued closed curves and the set of indecomposable objects in the associated cluster category. Under this bijection, they described irreducible morphisms, the Auslander-Reiten(=AR) translation and AR-triangles in the cluster category by geometric terms in the surface. They also gave a bijection between triangulations of the surface and cluster tilting objects in the cluster category such that flip of an arc is compatible with mutation.
  \item Based on Br\"{u}stle-Zhang's work, Zhang-Zhu-Zhou \cite{ZhZZ} proved that the intersection number of two curves is equal to the dimension of $\Ext^1$ of the corresponding objects and gave a geometric model of torsion pairs and their mutations.
  \item Canakci-Schroll \cite{CS} described a basis for $\Ext^1$ in the cluster category and also computed a basis for $\Ext^1$ in the module category of the corresponding Jacobian algebra by distinguishing different types of crossings between curves.
  \item Marsh-Palu \cite{MP} showed that Calabi-Yau reduction (introduced in \cite{IY}) can be interpreted as cutting along curves without self-intersections in the surface.
  \item The authors \cite{QQ, QZ2} also investigated other categories, which are used to define cluster categories, see the formula \eqref{eq:Amiot}, and obtained similar structures/formulae.
\end{itemize}
For general cases,
\begin{itemize}
\item For each ideal triangulation without self-folded triangles,
    Labardini-Fragoso \cite{LF2} associated a representation of the quiver with potential $(Q_\T,W_\T)$ to each curve without self-intersections and proved that mutation of representations is compatible with flip of triangulations.
\item Br\"{u}stle-Qiu \cite{BQ} made an effort to understand a basic functor in the cluster category,
i.e. the shift (or the AR-translation in this case),
in terms of an element, the tagged rotation, in the tagged mapping class group
of the marked surface.
Their motivation lies on the study of the Seidel-Thomas braid group.
\end{itemize}
Notice that most works above only deal with the unpunctured case.
This is because: i) the usual flip does not work for self-folded triangles
(cf. Figure~\ref{fig:self-folded}) and ii) the associated quivers with potential are much more complicated in the punctured case and their Jacobian algebras are not gentle (cf. \cite{ILF}). FST solved i) by introducing the notion of tagging.

In this paper, we aim to study the cluster categories associated to marked surfaces with punctures and with non-empty boundaries.
The main tool is skewed-gentle algebras (a special kind of clannish algebras),
which were developed in \cite{B,CB,De,G,GP}.
The essential results are summarized as follows.

\begin{theorem}[(Theorem~\ref{thm:bi}, Theorem~\ref{thm:ind}, Theorem~\ref{thm:T-rotation}, Theorem~\ref{thm:Int} and Theorem~\ref{thm:conn})]\label{thm:0}
Let $\surf$ be a marked surface with punctures and with non-empty boundary.
Given an admissible triangulation $\T$ of $\surf$ (see Definition~\ref{def:admissible}),
let $\C(\T)$ be the associated cluster category.
Then there is a bijection
\[\begin{array}{rccc}
X^\T\colon&\TC&\to&\so(\T)\\
&(\gamma,\kappa)&\mapsto&X^\T_{(\gamma,\kappa)}
\end{array}\]
from the set $\TC$ of tagged curves in the surface $\surf$
to the set $\so(\T)$ of string objects in the category $\C(\T)$ (see Definition~\ref{def:curve} and Definition~\ref{def:stringob}),
satisfying the following.
\begin{enumerate}
\item For every admissible triangulation $\T'$ of $\surf$,
there is an equivalence $\Theta\colon\C(\T)\simeq\C(\T')$, such that
$\Theta\circ X^\T=X^{\T'}.$
\item For any tagged curve $(\gamma,\kappa)\in\TC$, we have
$X^\T_{\rho(\gamma,\kappa)}\cong X^\T_{(\gamma,\kappa)}[1],$
where $\rho$ is the tagged rotation (see Definition~\ref{def:rotation}).
\item For any two tagged curves $(\gamma_1,\kappa_1)$, $(\gamma_2,\kappa_2)$
(not necessarily distinct), we have
\[\Int\left((\gamma_1,\kappa_1),(\gamma_2,\kappa_2)\right)=\dim_\k\Ext_{\C(\T)}^1(X^\T_{(\gamma_1,\kappa_1)},X^\T_{(\gamma_2,\kappa_2)})\]
where $\Int$ denotes
the intersection number (see Definition~\ref{def:Int});
\item The exchange graph $\CEG(\C(\T))$ of cluster tilting objects in $\C(\T)$ is isomorphic to the exchange graph
$\EGT(\surf)$ of tagged triangulations of $\surf$ and hence it is connected.
\end{enumerate}
\end{theorem}

\subsection{Context}
In Section~\ref{app:clan} we recall notions and notations about skewed-gentle algebras
that we will use throughout the paper.
In Section~\ref{sec:2}, we recall the background of cluster categories associated to triangulated marked surfaces.
In Section~\ref{sec:3}, we study the skewed-gentle algebra associated to an admissible triangulation and give a correspondence between tagged curves and string objects. The relation between such correspondences from different admissible triangulations is also studied.
In Section~\ref{sec:5}, we give homological interpretations of geometric objects from marked surfaces,
namely, tagged rotation, intersection numbers and exchange graph of tagged triangulations.
An example is presented in Section~\ref{sec:ex}, to demonstrate some of the notions/results in the paper.
The technical proof of the main theorem, Theorem~\ref{thm:Int}, is given in Section~\ref{sec:proof}.
In Appendix~~\ref{app:ex} we discuss some properties of admissible triangulations and in Appendix~~\ref{app:DWZ} we recall DWZ mutation of decorated representations.

\subsection{Conventions}
Throughout this article, $\k$ denotes an algebraically closed field.
For any $\k$-algebra $A$, an $A$-module means a finitely generated left $A$-module and we denote by $\mod A$ the category of all $A$-modules.
For a finite set $I$, we denote by $|I|$ the number of elements in $I$.
For an object $X$ in a triangulated category $\C$, we denote
\begin{itemize}
\item by $\add X$ the full subcategory of $\C$ consisting of direct summands of direct sums of copies of $X$;
\item by $X^\bot$ the full subcategory of $\C$ consisting of objects $Y$ with $\Hom_\C(X,Y)=0$;
\item by $\C/(X)$ the additive quotient category of $\C$ by $\add X$.
\end{itemize}

\section{Preliminaries on skewed-gentle algebras}\label{app:clan}
We recall from \cite{B,CB,De,G,GP} some notions, notations and results about skewed-gentle algebras used in this paper.

\subsection{Skewed-gentle algebras}\label{subsec:clannish}

A \emph{biquiver} is a tuple $(Q_0,Q_1,Q_2)$,
where $Q_0$ is the set of vertices, $Q_1$ is the set of solid arrows and
$Q_2$ is the set of dashed arrows. Let $s,t:Q_1\cup Q_2\rightarrow Q_0$ be
the start/terminal functions of arrows. We call an arrow $\alpha$ in $Q_1\cup Q_2$ a loop if $s(\alpha)=t(\alpha)$.

In this paper, we always assume that a biquiver $Q=(Q_0,Q_1,Q_2)$ satisfies
\begin{itemize}
\item each arrow in $Q_2$ is a dashed loop;
\item there is at most one loop in $Q_2$ at each vertex;
\item there is no loop in $Q_1$.
\end{itemize}
Let $Q_0^{Sp}$ be the subset of $Q_0$ consisting of vertices where there is a dashed loop in $Q_2$.

Skewed-gentle algebras, modeled on gentle algebras, were introduced in \cite{GP} as a certain class of clannish algebras defined in \cite{CB}.

\begin{definition}\label{def:gentle}
A pair $(Q,Z)$ of a biquiver $Q$ and a set $Z$ of compositions $ab$ of arrows $a$, $b$ in $Q_1$ is called skewed-gentle if the following conditions hold.
\begin{itemize}
\item For each vertex $p\in Q_0^{Sp}$, there is at most one arrow $\alpha\in Q_1$ ending at $p$ and at most one arrow $\beta\in Q_1$ starting at $p$, with $\beta\alpha\in Z$ (if both exist and similar below).
\item For each vertex $p\notin Q_0^{Sp}$, there are at most two arrows $\alpha_1,\alpha_2\in Q_1$ ending at $p$ and at most two arrows $\beta_1,\beta_2\in Q_1$ starting at $p$, and they can be labeled in a way that $\beta_1\alpha_1\in Z$, $\beta_2\alpha_2\in Z$, $\beta_1\alpha_2\notin Z$ and $\beta_2\alpha_1\notin Z$.
\end{itemize}
An algebra $\Lambda$ is called a skewed-gentle algebra if $\Lambda$ is Morita equivalent to $\k Q/(R)$ for a skewed-gentle pair $(Q,Z)$, where $R=Z\cup\{\varepsilon^2-\varepsilon\mid\varepsilon\in Q_2\}$.
\end{definition}

\begin{example}\label{exm:clan}
Let $Q$ be the following biquiver
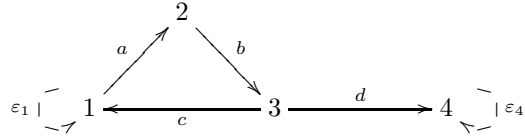
\[
\xymatrix{
&2\ar[dr]^b\\
1\ar@{-->}@(lu,ld)[]_{\varepsilon_1}\ar[ur]^a&&3\ar[ll]^c\ar[rr]^d&&4\ar@{-->}@(ru,rd)[]^{\varepsilon_4}
}\]
with $Z=\{ba,cb,ac\}$. Then $(Q,Z)$ is a skewed-gentle pair and hence $\k Q/(R)$ is a skewed-gentle algebra, where $R=Z\cup\{\varepsilon_1^2-\varepsilon_1,\varepsilon_4^2-\varepsilon_4\}$.
\end{example}

\subsection{Letters}\label{subsec:letters}

Let $(Q,Z)$ be a skewed-gentle pair. Following \cite{CB,G}, we associate a new biquiver $\widehat{Q}=(\widehat{Q}_0,\widehat{Q}_1,\widehat{Q}_2)$ to $Q=(Q_0,Q_1,Q_2)$ by adding two new vertices $i_+$ and $i_-$ and two new solid arrows
$a_{i_\pm}:i\rightarrow i_\pm$ for each vertex $i\in Q_0$. That is,
\begin{itemize}
\item $\widehat{Q}_0=Q_0\cup\{i_\pm\mid i\in Q_0\}$;
\item $\widehat{Q}_1=Q_1\cup\{a_{i_\pm}:i\rightarrow i_\pm\mid i\in Q_0\}$;
\item $\widehat{Q}_2=Q_2$.
\end{itemize}
For example, the biquiver $\widehat{Q}$ associated to the biquiver $Q$ in Example~\ref{exm:clan} is the following.
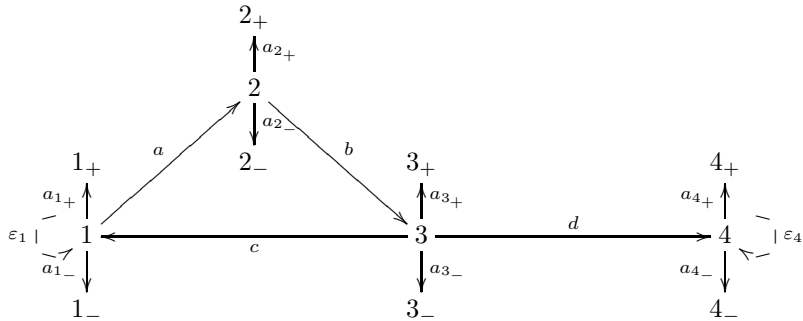
\[
\xymatrixrowsep{.5cm}
\xymatrixcolsep{1.6cm}
\xymatrix{
&2_+ \ar@{<-}[d]^{a_{2_+}}&\\
&2\ar[ddr]^b\\
1_+&2_- \ar@{<-}[u]_{a_{2_-}}&3_+ \ar@{<-}[d]^{a_{3_+}}&&4_+\\
1\ar@{-->}@(lu,ld)[]_{\varepsilon_1}\ar[uur]^a\ar[u]^{a_{1_+}}&&3\ar[ll]^c\ar[rr]^d&&4\ar@{-->}@(ru,rd)[]^{\varepsilon_4}\ar[u]^{a_{4_+}}\\
1_- \ar@{<-}[u]^{a_{1_-}}&&3_- \ar@{<-}[u]_{a_{3_-}}&&4_- \ar@{<-}[u]^{a_{4_-}}\\
}\]

For any arrow $\alpha$ in $\widehat{Q}$, we define a \emph{direct letter} $\alpha$ and an \emph{inverse letter} $\alpha^{-1}$, which are mutually inverse. Let $L$ be the set of all letters. The functions $s,t$ can be extended to $L$ by setting
$s(\alpha^{-1})=t(\alpha)\quad\text{and}\quad t(\alpha^{-1})=s(\alpha)$.
For each $i\in Q_0\subset \widehat{Q}_0$, let $L(i):=\{l\in L\mid s(l)=i\}$ . We divide $L(i)$ into two disjoint subsets $L_+(i)$ and $L_-(i)$ with linear orders such that the subset $L_\theta(i)$
has one of the following forms:
\begin{itemize}
  \item $\{a_{i_\theta}\}$,
  \item $\{a_{i_\theta}>\alpha\}$,
  \item $\{\beta^{-1}>a_{i_\theta}\}$,
  \item $\{\beta^{-1}>a_{i_\theta}>\alpha\}$,
  \item $\{\varepsilon^{-1}>a_{i_\theta}>\varepsilon\}$,
\end{itemize}
for $\theta\in\{\pm\}$, some solid arrows $\alpha$ and $\beta$ in $Q_1\subset \widehat{Q}_1$ and some dashed arrow $\varepsilon$ in $Q_2=\widehat{Q}_2$, satisfying that
\begin{itemize}
\item for any two solid arrows $\gamma$ and $\delta$ in $Q_1\subset \widehat{Q}_1$ with $t(\delta)=s(\gamma)=k$,
    we have that $\gamma\delta\in Z$ if and only if $\gamma$ and $\delta^{-1}$ are both in $L_+(k)$ or $L_-(k)$.
\end{itemize}
Observe that in the set $L_\theta(i)$, the inverse of an arrow in $Q$ is always greater than the arrow $a_{i_\theta}$, and an arrow in $Q$ is always smaller than the arrow $a_{i_\theta}$. Notice that if $L(i)\neq \{a_{i_+},a_{i_-}\}$, then there are exactly two possible choices for the pair $(L_+(i), L_-(i))$.

\begin{example}\label{exm:disjoint}
In Example~\ref{exm:clan}, one possible choice of the subsets $L_\theta(i)$ is
\begin{itemize}
  \item $L_+(1)=\{c^{-1}>a_{1_+}>a\}$ and $L_-(1)=\{\varepsilon_1^{-1}>a_{1_-}>\varepsilon_1\}$;
  \item $L_+(2)=\{a^{-1}>a_{2_+}>b\}$ and $L_-(2)=\{a_{2_-}\}$;
  \item $L_+(3)=\{b^{-1}>a_{3_+}>c\}$ and $L_-(3)=\{a_{3_-}>d\}$;
  \item $L_+(4)=\{d^{-1}>a_{4_+}\}$ and $L_-(4)=\{\varepsilon_4^{-1}>a_{4_-}>\varepsilon_4\}$.
\end{itemize}
Note that the inverse letters $a_{i_\theta}^{-1}$ are not listed here because they are not in any set $L_\theta(i)$.
However they can still appear as first letters in words (see below).
\end{example}

\subsection{Words}\label{subsec:words}

A \emph{word} $\m$ is a sequence $\omega_m\cdots\omega_2\omega_1$ of letters in $L$ satisfying that for any $1\leq j\leq m-1$, $\omega_{j}^{-1}\in L_\theta(i)$ and $\omega_{j+1}\in L_{\theta'}(i)$ for different $\theta,\theta'\in\{\pm\}$ and some $i\in Q_0\subset \widehat{Q}_0$. We call $\omega_1$ the first letter of $\m$ and $\omega_m$ the last letter of $\m$. Since both $L_\theta(i)$ and $L_{\theta'}(i)$ are subsets of $L(i)$, we have $t(\omega_j)=s(\omega_{j+1})=i$. The functions $s,t$ can be generalized to the set of words by $s(\m):=s(\omega_1)$ and $t(\m):=t(\omega_m)$.  The \emph{inverse} of a word $\m=\omega_m\cdots\omega_2\omega_1$ is defined as $\m^{-1}:=\omega_1^{-1}\omega_2^{-1}\cdots\omega_m^{-1}$. The product $\n\m$ of two words $\m=\omega_m\cdots\omega_2\omega_1$ and $\n=\omega_{m+r}\cdots\omega_{m+2}\omega_{m+1}$ is defined to be $\omega_{m+r}\cdots\omega_{m+2}\omega_{m+1}\omega_m\cdots\omega_2\omega_1$ if this is again a word.

A letter is called \emph{punctured} if it is of the from $a_{i_\theta}$ or $a_{i_\theta}^{-1}$ such that $L_{\theta}(i)=\{\varepsilon^{-1}>a_{i_\theta}>\varepsilon\}$ for some dashed arrow $\varepsilon$. A word $\m$ is called \emph{left inextensible} (resp. \emph{right inextensible}) if there is no letter $l$ such that $l\m$ (resp. $\m l$) is again a word. A word is called \emph{maximal} if it is both left and right inextensible. It is obvious that a word $\m=\omega_m\cdots\omega_1$ is right (resp. left) inextensible if and only if $\omega_1$ (resp. $\omega_m$) is of the form $a_{i_\theta}^{-1}$ (resp. $a_{i_\theta}$). In a word, except for its first letter (resp. last letter), there are no letters of the form $a_{i_\theta}^{-1}$ (resp. $a_{i_\theta}$). In particular, each word contains at most two punctured letters.

\begin{example}
In Example~\ref{exm:clan} with the disjoint subsets given in Example~\ref{exm:disjoint}, the punctured letters are $a_{1_-}$, $a_{1_-}^{-1}$, $a_{4_-}$ and $a_{4_-}^{-1}$. The sequence $\m=a_{3_+}d^{-1}\varepsilon_4^{-1}dc^{-1}$ is a word with $s(\m)=1$ and $t(\m)=3_+$, which is left inextensible but not right inextensible.
\end{example}

\subsection{Orders}\label{subsec:order}

The linear orders in the sets $L_\pm(i), i\in Q_0\subset \widehat{Q}_0$, induce a partial order $\geq$ on the set of words,
that $\m>\r$ if and only if $\m=\omega_m\cdots\omega_2\omega_1$ and $\r=\nu_r\cdots\nu_2\nu_1$
satisfy $\omega_j\cdots\omega_1=\nu_j\cdots\nu_1$ and $\omega_{j+1}>\nu_{j+1}$ for some $j\geq0$.

For each vertex $i\in Q_0$, let $W_\pm(i)$ be the set of left inextensible words whose first letter is in $L_\pm(i)$.
By construction, $W_\pm(i)$ are linearly ordered. For each word $\m$ in $W_\theta(i)$, we use $\bt\m$ to denote its successor (if exists) and use $\yd\m$ to denote its predecessor (if exists). So we have $\yd\m>\m>\bt\m$. In case $\m$ is right inextensible, we set $\m\bt:=(\bt\m^{-1})^{-1}$ and $\m\yd:=(\yd\m^{-1})^{-1}$.

Let $\m$ be a maximal word. By \cite{G}, $\bt(\m\bt)=(\bt\m)\bt$ provided both of them exist; denote by $\bt\m\bt$ one (or both) of them.

\begin{example}
In Example~\ref{exm:clan} with the disjoint subsets given in Example~\ref{exm:disjoint}, any word in the set $W_+(2)$ starts with one of the letters $b$, $a^{-1}$ and $a_{2_+}$. A word in $W_+(2)$ starting with $a_{2_+}$ has to be $a_{2_+}$ since it is already left inextensible. Moreover, any word in $W_+(2)$ that starts with $a^{-1}$ is bigger than $a_{2_+}$ and any word that starts with $b$ is less than $a_{2_+}$. Furthermore, we have $\bt a_{2_+}=a_{3_-}b$ and $\yd a_{2_+}=a_{2_-}a\varepsilon_1a^{-1}$.
\end{example}

\subsection{Admissible words}\label{subsec:strings}
For technical reasons, we introduce a special letter ${\varepsilon}^\ast$ for each dashed loop $\varepsilon$ and a map $F$ on letters which sends the elements in $\{{\varepsilon}^{-1}>{a_{i_\theta}}>{\varepsilon}\}$ to ${\varepsilon}^\ast$ and preserves the other letters.

A maximal word $\m=\omega_m\cdots\omega_1$ is called \emph{admissible} if the following conditions hold.
\begin{itemize}
  \item[(A1)] For each $\omega_i={\varepsilon}$ with $\varepsilon$ a dashed loop, we have that ${\omega}_1^{-1}\cdots{\omega}_{i-1}^{-1}>\omega_m\cdots\omega_{i+1}$, and for each $\omega_i={\varepsilon}^{-1}$ with $\varepsilon$ a dashed loop, we have that ${\omega}_1^{-1}\cdots{\omega}_{i-1}^{-1}<\omega_m\cdots\omega_{i+1}$.
  \item[(A2)] If $\m$ contains two punctured letters then $F(\m)$ is not a proper power of $F(\m')$ for any maximal word $\m'$ containing two punctured letters, where \[F(\m):=F(\omega_m)\cdots F(\omega_1)F(\omega_2^{-1})\cdots F(\omega_{m-1}^{-1}).\]
\end{itemize}
Let $\X$ be the set of admissible words. Note that if $\m$ is in $\X$ then so is its inverse $\m^{-1}$. Let $\bX$ be the set of all equivalence classes in $\X$ with respective to $\m\simeq\m^{-1}$. But when we say an element $\m$ in $\bX$, we always mean that $\m$ is a representative in an equivalence class.

\begin{example}\label{exm:admissible}
In Example~\ref{exm:clan} with the disjoint subsets given in Example~\ref{exm:disjoint}, consider the word $\m=a_{3_-}c^{-1}\varepsilon_1^{-1}cd^{-1}a_{4_-}^{-1}.$ It is clear that $\m$ is maximal. Since $a_{4_-}dc^{-1}<a_{3_-}c^{-1}$, $\m$ satisfies (A1). Moreover, $\m$ contains only one punctured letter $a_{4_-}^{-1}$, so (A2) holds automatically. Hence $\m$ is admissible, i.e. $\m\in\X.$
\end{example}

\subsection{Known results}\label{subsec:results}

In this subsection, we collect some results on indecomposable modules of skewed-gentle algebras and homomorphism spaces between them.

Let $\m\in\bX$. Associate an indeterminate to each punctured letter in $\m$ and let $A_{\m}$ be the $\k$-algebra generated by these indeterminates $x$ with relations $x^2=x$. A 1-dimensional module $N$ of $A_{\m}$ is an algebra homomorphism $N:A_\m\to\k$. It is determined (up to isomorphism) by the values $N(x)\in\{0,1\}$. More precisely,
\begin{itemize}
\item if $\m$ contains no punctured letters, then $A_{\m}=\k$ and there is one 1-dimensional module $N=\k$;
\item if $\m$ contains one punctured letter, then $A_{\m}=\k[x]/(x^2-x)$ and there are two 1-dimensional modules: $N=\k_a$ with $\k_a(x)=a$, for $a\in\{0,1\}$.
    Moreover, we have
    \begin{gather}\label{eq:xy}
        \dim_\k\Hom_{A_{\m}}(\k_u,\k_v)=\delta_{u,v}, \quad \forall u,v\in\{0,1\}.
    \end{gather}
\item if $\m$ contains two punctured letters, then $A_{\m}=\k\langle x,y\rangle/(x^2-x,y^2-y)$ and there are four 1-dimensional modules: $N=\k_{a,b}$ with $\k_{a,b}(x)=a$ and $\k_{a,b}(y)=b$, for $a,b\in\{0,1\}$.
    Moreover, we have
    \begin{gather}\label{eq:xy2}
        \dim_\k\Hom_{A_{\m}}\left(\k_{u,u'},\k_{v,v'}\right)=\delta_{u,v}\delta_{u',v'},
            \quad \forall u,v,u',v'\in\{0,1\}.
    \end{gather}
\end{itemize}

\begin{construction}\label{cstr:rep}
For each pair $(\m,N)$ with $\m=\omega_m\cdots\omega_1\in\bX$ and a 1-dimensional $A_\m$-module $N$, the associated representation $M=M(\m,N)$ of $Q$ bounded by $R$ is constructed as follows.
\begin{itemize}
  \item For each vertex $i\in Q_0$, let $I_i=\{1\leq j\leq m-1\mid t(\omega_j)=i\}$.
  \item Let $M_i$ be a vector space of dimension $|I_i|$, say with base vectors $z_j$, $j\in I_i$.
  \item If $\omega_{j+1}=\alpha$ an arrow in $Q_1$, define $M_\alpha(z_j)=z_{j+1}$, if $\omega_{j+1}=\alpha^{-1}$, with $\alpha$ an arrow in $Q_1$, define $M_\alpha(z_{j+1})=z_{j}$.
  \item If $\omega_{j+1}=\varepsilon$ an arrow in $Q_2$, define $M_\varepsilon(z_j)=M_\varepsilon(z_{j+1})=z_{j+1}$, if $\omega_{j+1}=\varepsilon^{-1}$, with $\varepsilon$ an arrow in $Q_2$, define $M_\varepsilon(z_{j+1})=M_\varepsilon(z_j)=z_j$.
  \item If $\omega_{1}$ (resp. $\omega_m$) is punctured with indeterminate $x$, define $M_{\varepsilon_{t(\omega_1)}}(z_1)=N(x)z_1$ (resp. $M_{\varepsilon_{s(\omega_m)}}(z_{m-1})=N(x)z_{m-1}$), where $\varepsilon_i$ denotes the dashed loop at $i$.
  \item All other components of $M_\beta$ are zero, for any $\beta\in Q_1\cup Q_2$.
\end{itemize}
\end{construction}

\begin{example}
Let $\m=a_{3_-}c^{-1}\varepsilon_1^{-1}cd^{-1}a_{4_-}^{-1}$ be the admissible word in Example~\ref{exm:admissible}. Since $\m$ contains one punctured letter $a_{4_-}^{-1}$, we have $A_{\m}=\k[x]/(x^2-x)$ and there are two different 1-dimensional $A_\m$-modules $\k_u$ for $u=0,1$. By Construction~\ref{cstr:rep}, the associated representations $M(\m,\k_u)$ are
\[
\xymatrix{
&0\ar[dr]^0\\
\k^2\ar@{-->}@(lu,ld)[]_{\left(\begin{smallmatrix}1&1\\0&0\end{smallmatrix}\right)}\ar[ur]^0
&&\k^2\ar[ll]^{\left(\begin{smallmatrix}1&0\\0&1\end{smallmatrix}\right)}\ar[rr]^{\left(\begin{smallmatrix}1&&0\end{smallmatrix}\right)}
&&\k\ar@{-->}@(ru,rd)[]^{u}
}\]
\end{example}

\begin{theorem}\label{thm:Deng}\cite{B,CB,De}
Let $\Lambda=\k Q/(R)$ be a skewed-gentle algebra. Then $(\m,N)\mapsto M(\m,N)$ is an injective map from the set of pairs $(\m,N)$, with $\m\in\bar{\X}$ and $N$ a 1-dimensional $A_{\m}$-module (up to isomorphism), to the set of indecomposable representations $M(\m,N)$ (up to isomorphism) of $Q$ bounded by $R$.
\end{theorem}

\begin{remark}\label{rem:De}
In fact, Bondarenko \cite{B}, Crawley-Boevey \cite{CB} and Deng \cite{De} proved the result above for general clannish algebras, where the map can be upgraded to a bijection by enlarging the set $\bar{\X}$ and taking $N$ to be an arbitrary indecomposable $A_{\m}$-module. Furthermore, any indecomposable module $M$, which is not in the image of the injective map in the above theorem, is in a homogeneous tube or in a tube of rank 2 and does not sit in the bottom of the tube. So in particular $\Hom(M,\tau M)\neq 0$ for such an indecomposable module $M$.
\end{remark}

The Auslander-Reiten translation $\tau$ can be interpreted by the order of words.

\begin{theorem}[\cite{G}]\label{thm:G}
For any $\m=\omega_m\cdots\omega_1\in\bX$ and any 1-dimensional $A_{\m}$-module $N$, if $M(\m,N)$ is not projective, then
\[\tau M(\m,N)=\begin{cases}M(\bt\m\bt,\k)& \text{if $\m$ contains no punctured letters and $N=\k$,}\\
M(\bt\m,\k_{1-a})&\text{if only $\omega_1$ is punctured and $N=\k_a$,}\\
M(\m,\k_{1-a,1-b})&\text{if both $\omega_1$ and $\omega_m$ are punctured and $N=\k_{a,b}$.}\end{cases}\]
\end{theorem}

For technical reasons, we also consider a trivial word $1_i$ corresponding to each vertex $i\in Q_0\subset\widehat{Q}_0$. Let $\m=\omega_m\cdots\omega_1$ be a word in $\bX$. For any integers $i,j$ with $0\leq i<j\leq m+1$, we consider the subword $\m_{(i,j)}$ of $\m$ between $i$ and $j$ defined as
\[\m_{(i,j)}=\begin{cases}\omega_{j-1}\cdots\omega_{i+1}&\text{if $i<j-1$,}\\1_{t(\omega_i)}&\text{if $i=j-1$,}\end{cases}\]
where $1_{t(\omega_0)}:=1_{s(\omega_1)}$.

Let $\m=\omega_m\cdots\omega_1$ and $\r=\nu_r\cdots\nu_1$ be two words in $\bar{\X}$. A pair $\left((i,j),(h,l)\right)$ of pairs of integers $i,j,h,l$ with $0\leq i<j\leq m+1$ and $0\leq h<l\leq r+1$ is called an \emph{int-pair} from $\m$ to $\r$ if one of the following conditions holds:
\begin{itemize}
  \item $\m_{(i,j)}=\r_{(h,l)}$, $\omega_{i}^{-1}<\nu_{h}^{-1}$ and $\omega_{j}<\nu_{l}$,
  \item $\m_{(i,j)}=\left(\r_{(h,l)}\right)^{-1}$, $\omega_{i}^{-1}<\nu_{l}$ and $\omega_{j}<\nu_{h}^{-1}$,
\end{itemize}
where if an inequality contains at least one of $\omega_0$, $\omega_{m+1}$, $\nu_0$ and $\nu_{r+1}$ then we assume that it holds automatically. Let $H^{\m,\r}$ be the set of int-pairs from $\m$ and $\r$.

\begin{example}\label{exm:int}
Let $\m=a_{4_+}a_{4-}^{-1}$ and $\r=a_{3_-}c^{-1}\varepsilon_1^{-1}cd^{-1}a_{4_-}^{-1}$ be two admissible words for the skewed-gentle pair in Example~\ref{exm:clan} with the disjoint subsets given in Example~\ref{exm:disjoint}. Then $H^{\m,\r}$ contains only one element $((0,2),(0,2))$ for which, $\m_{(0,2)}=a_{4-}^{-1}=\r_{(0,2)}$ with $\omega_0^{-1}<\nu_0^{-1}$ and $\omega_2=a_{4_+}<d^{-1}=\nu_2$.
\end{example}

\begin{notation}\label{notation}
Let $(\m,N_1)$ and $(\r,N_2)$ be two pairs, where $\m,\r\in\bX$ and $N_1$ (resp. $N_2$) is a 1-dimensional module of $A_{\m}$ (resp. $A_{\r}$). For each int-pair $J=\left((i,j),(h,l)\right)$ in $H^{\m,\r}$, denote by $A_J$ the $\k$-algebra generated by the indeterminates associated to punctured letters contained in $\m_{(i,j)}$ (or equivalently in $\r_{(k,l)}$). Then $A_J$ is a subalgebra of $A_{\m}$ and $A_\r$ and hence both $N_1$ and $N_2$ can be regarded as $A_J$-modules.
\end{notation}
\begin{theorem}[\cite{G}]\label{thm:G2}
Under Notation~\ref{notation}, we have
\[\dim_\k\Hom_{\Lambda}(M(\m,N_1),M(\r,N_2))=\sum_{J\in H^{\m,\r}}\dim_\k\Hom_{A_J}(N_1,N_2).\]
\end{theorem}

\section{Background on cluster categories for marked surfaces}\label{sec:2}
\subsection{Jacobian algebras and Ginzburg dg algebras}\label{subsec:dg}
Let $Q$ be a finite quiver and $W$ a potential on $Q$, that is, a sum of cycles in $Q$. The \emph{Jacobian algebra} of the quiver with potential $(Q, W)$ is the quotient
\[
    \hua{P}(Q,W):=\widehat{\k Q}/\overline{\partial W},
\]
where $\widehat{\k Q}$ is the complete path algebra of $Q$,
$\partial W=\<\partial_a W: {a \in Q_1}\>$ and
$\overline{\partial W}$ is the closure of $\partial W$ in $\widehat{\k Q}$ (cf. \cite{DWZ}).

The Jacobian algebra is the $0^{th}$ cohomology of
its refinement, the \emph{Ginzburg dg algebra} $\Gamma=\QP$ of $(Q,W)$
(see the construction in \cite[Section~7.2]{K10}).
There are three categories associated to $\Gamma$, namely,
\begin{itemize}
\item
the \emph{finite dimensional derived category} $\D_{fd}(\Gamma)$ of $\Gamma$,
which is a 3-Calabi-Yau category;
\item
the \emph{perfect derived category} $\per(\Gamma)$ of $\Gamma$,
which contains $\D_{fd}(\Gamma)$;
\item
the \emph{cluster category} $\hua{C}(\Gamma)$ of $\Gamma$, which is
the (triangulated) 2-Calabi-Yau quotient category
\begin{gather}\label{eq:Amiot}
\hua{C}(\Gamma):=\per(\Gamma)/\hua{D}_{fd}(\Gamma).
\end{gather}
\end{itemize}
Furthermore there is a canonical cluster tilting object $T_\Gamma$ in $\C(\Gamma)$ induced by
the silting object $\Gamma$ in $\per\Gamma$ such that
\begin{equation}\label{eq:eq}
\hua{C}(\Gamma)/(T_\Gamma)\simeq\mod\End_{\hua{C}(\Gamma)}(T_\Gamma)^{\text{op}}\cong\mod \hua{P}(Q,W).
\end{equation}
See \cite[Theorem 3.5]{A} and \cite[\S 2.1, Proposition (c)]{KR}.

For a vertex $i$ of $Q$, let $\mu_i(Q,W)$ be the mutation of $(Q,W)$ at $i$ in the sense of \cite{DWZ}, see also Appendix~\ref{app:DWZ}. By \cite{KY}, there exists a canonical triangulated equivalence
\begin{equation}\label{eq:sim}
\widetilde{\mu_i}:\hua{C}(\Gamma(Q,W))\simeq\hua{C}(\Gamma(\mu_i(Q,W))).
\end{equation}

\subsection{Quivers with potential from marked surfaces}\label{sec:QPtosurf}
Throughout the article, $\surf$ denotes a \emph{marked surface} with non-empty boundary in the sense of \cite{FST},
that is, a compact connected oriented surface $\surf$
with a finite set $\M$ of marked points on its boundary $\partial \surf$
and a finite set $\P$ of punctures in its interior $\surf\setminus\partial\surf$ such that the following conditions hold:
\begin{itemize}
\item each connected component of $\partial \surf$ contains at least one marked point,
\item $\surf$ is not closed, i.e. $\partial \surf \neq \emptyset$,
\item the rank
\begin{equation}\label{eq:rank}
n=6g+3p+3b+m-6
\end{equation}
of the surface is positive, where $g$ is the genus of $\surf$,
$b$ the number of boundary components,
$m=|\M|$ the number of marked points and $p=|\P|$ the number of punctures.
\end{itemize}

\begin{definition}[Curves and tagged curves]\label{def:curve}
Let $\surf$ be a marked surface with non-empty boundary.
\begin{itemize}
\item
An (ordinary) \emph{curve} in $\surf$ is a continuous function $\gamma:[0,1]\rightarrow \surf$ satisfying
 \begin{itemize}
   \item both $\gamma(0)$ and $\gamma(1)$ are in $\M\cup\P$;
   \item for any $0<t<1$, $\gamma(t)$ is in $\surf\setminus\left(\partial \surf\cup\P\right)$;
   \item $\gamma$ is not null-homotopic or homotopic to a boundary segment.
 \end{itemize}
\item The inverse of a curve $\gamma$ is defined as $\gamma^{-1}(t):=\gamma(1-t)$ for $t\in[0,1]$.
\item For two curves $\gamma_1,\gamma_2$, $\gamma_1\sim\gamma_2$ means that $\gamma_1$ is homotopic to $\gamma_2$ relative to $\{0,1\}$ (i.e. fixing the endpoints). Define an equivalence relation $\simeq$ on the set of curves in $\surf$ that $\gamma_1\simeq\gamma_2$ if and only if either $\gamma_1\sim\gamma_2$ or $\gamma_1^{-1}\sim\gamma_2$. Denote by $\CS$ the set of equivalence classes of curves in $\surf$ w.r.t. $\simeq$.
\item Let $\gamma$ be a curve in $\CS$ such that at least one of its endpoints is a puncture. Then define its \emph{completion} $\bar{\gamma}$ as in Figure~\ref{fig:completion}.
\begin{figure}[htpb]\centering
\begin{tikzpicture}[xscale=1.5,yscale=.3,rotate=90]
\draw[ultra thick]plot [smooth,tension=1] coordinates {(170:4.5) (180:4) (190:4.5)};
\draw[\BLUE,thick](180:4)to(0:4);\draw[thick](0:4)node{$\bullet$}(180:4)node{$\bullet$}
(1,-1.5)node{$\Longrightarrow$};
\end{tikzpicture}\qquad
\begin{tikzpicture}[xscale=1.5,yscale=.3,rotate=90]
\draw[ultra thick]plot [smooth,tension=1] coordinates {(170:4.5) (180:4) (190:4.5)};
\draw[\BLUE,thick,->-=.5,>=stealth]plot [smooth,tension=1] coordinates {(-4,0) (2,.5) (5,0) (2,-.5) (-4,0)};
\draw[\BLUE,thick](5,0)to(5,0.001);
\draw[\BLUE,dotted,thick](180:4)to(0:3);
\draw[thick](0:3)node{$\bullet$}(180:4)node{$\bullet$};
\end{tikzpicture}

\qquad\begin{tikzpicture}[xscale=.7,yscale=.3]
\draw[\BLUE,thick](0,4)to(0,-4) (0,-4.5)node{ }(0,7)node{ };\draw(3,0)node{$\Longrightarrow$};
\draw[thick](0,4)node{$\bullet$}(0,-4)node{$\bullet$};
\end{tikzpicture}\qquad\qquad
\begin{tikzpicture}[xscale=.7,yscale=.3]
\draw[\BLUE,thick](0,0) ellipse (1 and 5);
\draw[\BLUE,dotted,thick](0,3)to(0,-3);
\draw[thick](0,3)node{$\bullet$}(0,-3)node{$\bullet$};
\end{tikzpicture}
\caption{The completions of curves}
\label{fig:completion}
\end{figure}
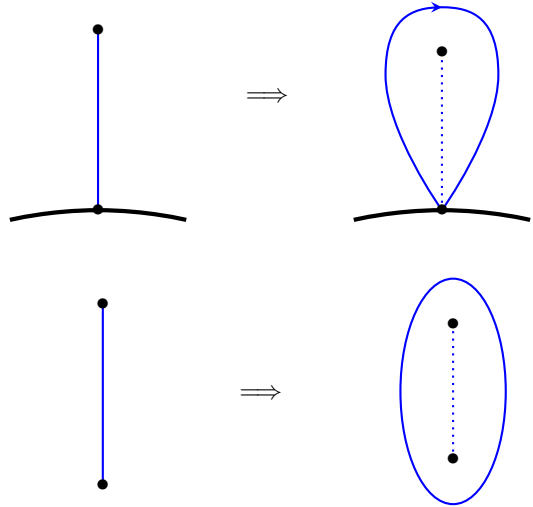
\item
A \emph{tagged curve} is a pair $(\gamma,\kappa)$, where $\gamma$ is a curve in $\surf$ and $\kappa:\{t\mid \gamma(t)\in \P\}\rightarrow\{0,1\}$ is a map, satisfying the following conditions:
\begin{enumerate}
  \item[(T1)] $\gamma$ does not cut out a once-punctured monogon by a self-intersection (including endpoints), cf. Figure~\ref{fig:monogon};
  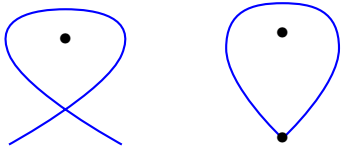
\begin{figure}[htpb]\centering
  	\begin{tikzpicture}[xscale=1.5,yscale=.2,rotate=90]
  	\draw[ultra thick,white]plot [smooth,tension=1] coordinates {(170:4.5) (180:4) (190:4.5)};
  	\draw[\BLUE,thick]plot [smooth,tension=1] coordinates {(-4,-.5) (2,.5) (5,0) (2,-.5) (-4,.5)};
  	\draw[thick](0:3)node{$\bullet$};
  	\end{tikzpicture}
  	\quad
  	\begin{tikzpicture}[xscale=1.5,yscale=.2,rotate=90]
  	\draw[ultra thick,white]plot [smooth,tension=1] coordinates {(170:4.5) (180:4) (190:4.5)};
  	\draw[\BLUE,thick]plot [smooth,tension=1] coordinates {(-4,0) (2,.5) (5,0) (2,-.5) (-4,0)};
  	\draw[thick](0:3)node{$\bullet$}(180:4)node{$\bullet$};
  	\end{tikzpicture}
  	\caption{Once-punctured monogons}
  	\label{fig:monogon}
  \end{figure}
  \item[(T2)] if $\gamma(0),\gamma(1)\in\P$, then the completion $\bar{\gamma}$ is not a proper power of a closed curve in the sense of the multiplication in the fundamental group of $\surf$.
\end{enumerate}
Note that $\kappa(t)\in\P$ implies $t\in\{0,1\}$.
Moreover, write $\kappa=\emptyset$ when $\{t\mid \gamma(t)\in \P\}=\emptyset$ by convention.
\item The inverse of a tagged curve $(\gamma,\kappa)$ is defined as $(\gamma,\kappa)^{-1}:=(\gamma^{-1},\kappa^{-1})$, where $\kappa^{-1}(t):=\kappa(1-t)$.
\item For two tagged curves $(\gamma_1,\kappa_1),(\gamma_2,\kappa_2)$, $(\gamma_1,\kappa_1)\sim(\gamma_2,\kappa_2)$ means that $\gamma_1\sim\gamma_2$ and $\kappa_1=\kappa_2$. Define an equivalence relation $\simeq$ on the set of tagged curves in $\surf$ that $(\gamma_1,\kappa_1)\simeq(\gamma_2,\kappa_2)$ if and only if either $(\gamma_1,\kappa_1)\sim(\gamma_2,\kappa_2)$ or $(\gamma_1,\kappa_1)\sim(\gamma_2,\kappa_2)^{-1}$. Denote by $\TC$ the set of equivalence classes of tagged curves in $\surf$ w.r.t. $\simeq$.
\end{itemize}
\end{definition}

\begin{definition}[Tagged rotation \cite{BQ}]\label{def:rotation}
The rotation $\rho(\gamma)$ of a curve $\gamma$ in $\CS$ is the curve
obtained from $\gamma$ by moving every endpoint of $\gamma$ that is in $\M$ along the boundary anticlockwise to the next marked point.
The tagged rotation $(\gamma',\kappa')=\rho(\gamma,\kappa)$ of
a tagged curve $(\gamma,\kappa)\in\TC$ consists of the curve $\gamma'=\rho(\gamma)$ and the map $\kappa'$ defined by $\kappa'(t)=1-\kappa(t)$ for $t$ with $\gamma(t)\in\P$, cf. Figure~\ref{fig:rotate}.

\end{definition}
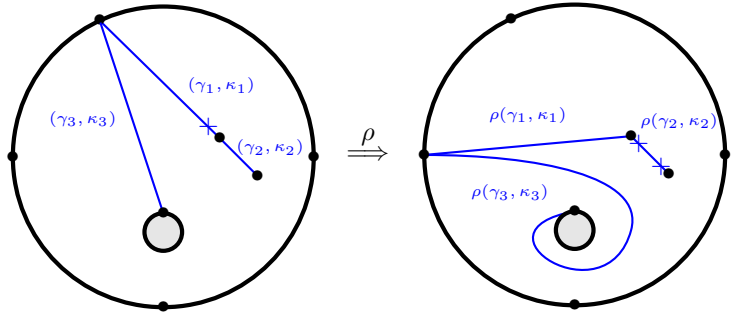
\begin{figure}[htpb]\centering
\begin{tikzpicture}[scale=.5]
\draw[thick](1.5,2.5-2)edge[\BLUE](115:4)edge[\BLUE](2.5,1.5-2);
\draw[thick] (115:4)node{$\bullet$};
\draw[\BLUE,thick] (115:4)to(0,-1.5);
\draw[\BLUE] (.4,3.5-1.6)node[right]{\tiny{$(\gamma_1,\kappa_1)$}} (2-.3,2.25-2)node[right]{\tiny{$(\gamma_2,\kappa_2)$}}
(-1,1)node[left]{\tiny{$(\gamma_3,\kappa_3)$}};
\draw[ultra thick](0,0) circle (4)(1.5,2.5-2)node{$\bullet$}(2.5,1.5-2) node{$\bullet$};
\draw[ultra thick,fill=gray!20](0,-2) circle (.5) (0,-1.5) node{$\bullet$}
    (1.5-.3,2.5-2+.3) node[\BLUE]{$+$};
  \foreach \j in {1,...,3}{\draw(90-90*\j:4) node{$\bullet$};}
\end{tikzpicture}
\begin{tikzpicture}[scale=.5]
\draw[\BLUE] (-1.25,.6)node[above]{\tiny{$\rho(\gamma_1,\kappa_1)$}}
(2-.5,2.4-1.5)node[right]{\tiny{$\rho(\gamma_2,\kappa_2)$}} (-3,-1)node[right]{\tiny{$\rho(\gamma_3,\kappa_3)$}};
\draw[thick](1.5,2.5-2)edge[\BLUE](-4,0)edge[\BLUE](2.5,1.5-2);
\draw[\BLUE,thick] (-4,0).. controls +(0:3) and +(75:1.5) ..(1.5,-2)
    .. controls +(-105:2.5) and +(190:3) ..(0,-1.5);
\draw[thick](1.7,2.3-2) node[\BLUE]{$+$}  (2.3,1.7-2) node[\BLUE]{$+$};;
\draw(-5.5,0)node{$\Longrightarrow$}node[above]{$\rho$};
\draw[thick](0,0) circle (4)(115:4)node{$\bullet$};
\draw[ultra thick](0,0) circle (4)(1.5,2.5-2)node{$\bullet$}(2.5,1.5-2) node{$\bullet$};
\draw[ultra thick,fill=gray!20](0,-2) circle (.5)(0,-1.5) node{$\bullet$};
  \foreach \j in {1,...,3}{\draw(90-90*\j:4) node{$\bullet$};}
\end{tikzpicture}
\caption{The tagged rotations of three tagged curves in $\surf$}
\label{fig:rotate}
\end{figure}

\def\t{\mathfrak{T}}
\def\p{\mathfrak{P}}

\begin{definition}[Intersection numbers]\label{def:Int}
For any two curves $\gamma_1,\gamma_2\in\CS$,
\begin{itemize}
\item let $\gamma_1\cap\gamma_2=\{(t_1,t_2)\mid\gamma_1(t_1)=\gamma_2(t_2)\notin\P\cup\M\}\subset(0,1)^2$
be the set of interior intersections between $\gamma_1$ and $\gamma_2$;
\item the intersection number between them is defined to be
\[\Int(\gamma_1,\gamma_2):=\min\{|\gamma_1'\cap\gamma_2'|\mid\text{$\gamma_1'\sim\gamma_1$, $\gamma_2'\sim\gamma_2$}\}.\]
\end{itemize}
For any two tagged curves $(\gamma_1,\kappa_1)$ and $(\gamma_2,\kappa_2)\in\TC$,
\begin{itemize}
\item let $\p(\gamma_1,\gamma_2)=\{(t_1,t_2)\mid\gamma_1(t_1)=\gamma_2(t_2)\in\P\}\subset\{0,1\}^2$
be the set of intersections between $\gamma_1$ and $\gamma_2$ at $\P$;
\item A pair $(t_1,t_2)$ in $\p(\gamma_1,\gamma_2)$ is called a tagged intersection between $(\gamma_1,\kappa_1)$ and $(\gamma_2,\kappa_2)$ if
    \begin{itemize}
      \item $\kappa_1(t_1)\neq\kappa_2(t_2)$, and
      \item when $\gamma_1|_{t_1\rightarrow (1-t_1)}\sim\gamma_2|_{t_2\rightarrow (1-t_2)}$, we have $\gamma_1(1-t_1)=\gamma_2(1-t_2)$ belongs to $\P$ and $\kappa_1(1-t_1)\neq\kappa_2(1-t_2)$, where $\gamma|_{0\to 1}=\gamma$ and $\gamma|_{1\to 0}=\gamma^{-1}$;
    \end{itemize}
    let $\t((\gamma_1,\kappa_1),(\gamma_2,\kappa_2))$ be the set of tagged intersections between $(\gamma_1,\kappa_1)$ and $(\gamma_2,\kappa_2)$;
\item the intersection number between them is defined to be
\[\Int\left((\gamma_1,\kappa_1),(\gamma_2,\kappa_2)\right):=\Int(\gamma_1,\gamma_2)+
|\t((\gamma_1,\kappa_1),(\gamma_2,\kappa_2))|.\]
\end{itemize}
\end{definition}

We explain the intersection number of two tagged curves in some special cases.

\begin{example}
 Let $(\gamma_1,\kappa_1)$ and $(\gamma_2,\kappa_2)$ be two tagged curves in $\TC$.
\begin{itemize}
  \item If all the endpoints of $\gamma_1$ and $\gamma_2$ are in $\M$, then \[\Int\left((\gamma_1,\kappa_1),(\gamma_2,\kappa_2)\right)=\Int(\gamma_1,\gamma_2).\]
  \item If $\gamma_1$ and $\gamma_2$ are not in the same equivalence class in $\CS$ (i.e. $\gamma_1\nsim\gamma_2$ and $\gamma_1\nsim\gamma_2^{-1})$, then \[\Int\left((\gamma_1,\kappa_1),(\gamma_2,\kappa_2)\right)=\Int(\gamma_1,\gamma_2)
  +|\left\{(t_1,t_2)\mid \gamma_1(t_1)=\gamma_2(t_2)\in\P,\ \kappa_1(t_1)\neq\kappa_2(t_2)\right\}|.\]
  \item If $\gamma_1\sim \gamma_2$ whose endpoints are two different punctures, then
  \[|\t((\gamma_1,\kappa_1),(\gamma_2,\kappa_2))|=\begin{cases}2&\text{if $\kappa_1(0)\neq\kappa_2(0)$ and $\kappa_1(1)\neq\kappa_2(1)$,}\\0&\text{otherwise.}\end{cases}\]
  \item If the two tagged curves are as in Figure~\ref{fig:exmint} where $\gamma_1\sim\gamma_2^{-1}$ and
  \[\kappa_a(t)=\begin{cases}1&\text{if $a=1$ and $t=0$,}\\0&\text{otherwise},\end{cases}\]
  then $(0,0)\in\t \left((\gamma_1,\kappa_1),(\gamma_2,\kappa_2)\right)$. This is because $\kappa_1(0)=1\neq \kappa_2(0)=0$ and $\gamma_1|_{0\to 1}\nsim \gamma_2|_{0\to 1}$. For the pair $(0,1)$, we also have $\kappa_1(0)=1\neq \kappa_2(1)=0$. But since $\gamma_1|_{0\to 1}\sim \gamma_2|_{1\to 0}$, we need to compare the values of $\kappa_1(1)$ and $\kappa_2(0)$. Because $\kappa_1(1)=0=\kappa_2(0)$, the pair $(0,1)\notin\t \left((\gamma_1,\kappa_1),(\gamma_2,\kappa_2)\right)$. It is easy to see that neither $(1,0)$ nor $(1,1)$ is in $\t \left((\gamma_1,\kappa_1),(\gamma_2,\kappa_2)\right)$. Hence in this case we have
  \[\Int((\gamma_1,\kappa_1),(\gamma_2,\kappa_2))=1.\]
  We also mention that
  \[\Int((\gamma_1,\kappa_1),(\gamma_1,\kappa_1))=2.\]
  Since $(\gamma_1,\kappa_1)$ has self-intersections, it is not a tagged arc in the sense of \cite{FST} but it is in $\TC$.
  \begin{figure}[htpb]\centering
  	\begin{tikzpicture}[rotate=45]
  	\draw[\BLUE,->-=.45,>=stealth, thick](0,0).. controls +(45:2) and +(90:1) .. (2.4,0)
  	.. controls +(-90:1) and +(-30:2) .. (0,0);
  	\draw[\BLUE](0.17,0.17)node{$\times$};
  	\draw[](3,.7)node{$\gamma_2$}  (2.5,.5)node{$\gamma_1$};
  	\draw[\BLUE,-<-=.47,>=stealth, thick](0,0).. controls +(75:3) and +(90:1) .. (3,0)
  	.. controls +(-90:1) and +(-40:3) .. (0,0);
  	\draw[ultra thick,fill=gray!20] (1.4,.15) circle (.4);
  	\draw(0,0)node{$\bullet$};
  	\end{tikzpicture}
  	\caption{The punctured intersections}
  	\label{fig:exmint}
  \end{figure}
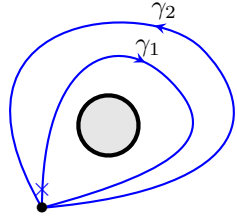
\end{itemize}
\end{example}

\begin{remark}\label{rem:FST}
Note that we have less restriction for curves and tagged curves than FST, that is, we allow self-intersections. More precisely,
\begin{itemize}
  \item the curves $\gamma$ in $\CS$ without self-intersections (i.e. $\Int(\gamma,\gamma)=0$) are precisely the arcs in the sense of \cite[Definition~2.2]{FST};
  \item the tagged curves $(\gamma,\kappa)$ in $\TC$ without self-intersections (i.e. $\Int((\gamma,\kappa),(\gamma,\kappa))=0$) are precisely the tagged arcs in the sense of \cite[Definition~2.4]{FST};
  \item for two curves $\gamma_1$ and $\gamma_2$ in $\CS$ without self-intersections, we have that $\Int(\gamma_1,\gamma_2)=0$ if and only if they are compatible in the sense of \cite[Definition~7.1]{FST};
  \item for two tagged curves $(\gamma_1,\kappa_1)$ and $(\gamma_2,\kappa_2)$ in $\TC$ without self-intersections, we have that $\Int((\gamma_1,\kappa_1),(\gamma_2,\kappa_2))=0$ if and only if they are compatible in the sense of \cite[Definition~7.4]{FST}.
\end{itemize}
\end{remark}

By Remark~\ref{rem:FST}, the following definitions of ideal triangulations and tagged triangulations are equivalent to the original ones in \cite{FST}.

\begin{definition}[Ideal triangulations and tagged triangulations \cite{FST}]\label{def:triangulation}
Let $\surf$ be a marked surface with non-empty boundary.
\begin{itemize}
\item
An \emph{ideal triangulation} is a maximal collection $\T$ of curves in $\CS$ such that $\Int(\gamma_1,\gamma_2)=0$ for any $\gamma_1,\gamma_2\in\T$.
\item
A \emph{tagged triangulation} is a maximal collection $\T$ of tagged curves in $\TC$ such that $\Int\left((\gamma_1,\kappa_1),(\gamma_2,\kappa_2)\right)=0$ for any $(\gamma_1,\kappa_1),(\gamma_2,\kappa_2)\in\T$.
\end{itemize}

\end{definition}

Any ideal/tagged triangulation $\T$ of $\surf$ consists of $n$ ordinary/tagged curves (see \cite[Proposition 2.10, Theorem 7.9]{FST}), where $n$ is the rank of $\surf$ (cf. \eqref{eq:rank}). We require $n>0$ and exclude the case of once-punctured monogon (where $n=1$) in the proofs. However, all the results hold in this case by a direct checking and thus we will not exclude this case in the statements.

A triangle in $\T$ has three distinct sides unless it is a \emph{self-folded triangle} as in the left picture of Figure~\ref{fig:self-folded}, where we call $\alpha$ the \emph{folded side} and $\beta$ the \emph{remaining side}.

\begin{figure}[htpb]\centering
\begin{tikzpicture}[scale=.7]
\draw[thick] (0,0) circle (2) to (0,-2);
\draw (0,0) node {$\bullet$} (0,-2) node {$\bullet$};
\draw (0,2) node[above] {$\beta$};\draw (0,-1) node[left] {$\alpha$};
\draw[thick](4,-2)to[bend left=60](4,0);
\draw (3.5,-1) node[left] {$\alpha$};
\draw[thick
](4,-2)to[bend right=60](4,0);
\draw (4.5,-1) node[right] {$\beta^{\bowtie}$};
\draw(4,-2)node{$\bullet$}
(4,0)node{$\bullet$}(4.19,-.15)node{$+$}(4,0)node[above]{$P$};
\end{tikzpicture}
\caption{The self-folded triangle and the corresponding tagged version}
\label{fig:self-folded}
\end{figure}
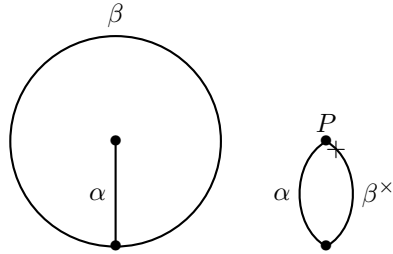

The \emph{flip} of an ideal triangulation $\T$, w.r.t. a curve $\alpha$ in $\T$, is the unique ideal triangulation $\T'$ (if it exists) that shares all curves in $\T$ but $\alpha$. One can always flip an ideal triangulation w.r.t. a curve unless it is the folded side of a self-folded triangle. To overcome this shortcoming, FST \cite{FST} introduced tagged triangulations, with tagged flips, so that every tagged triangulation can be flipped w.r.t. any tagged curve in it. The exchange graph of tagged triangulations with tagged flips is denoted by $\EGT(\surf)$,
that is, the graph whose vertices are tagged triangulations and whose edges are tagged flips.

For each curve $\gamma$ in $\CS$ with $\Int(\gamma,\gamma)=0$, we define its tagged version $\gamma^\times$ to be $(\gamma,\emptyset)$ unless $\gamma$ is a loop enclosing a puncture, as $\beta$ in the left picture of Figure~\ref{fig:self-folded}. In that case $\beta^\times$, as in the right picture of Figure~\ref{fig:self-folded}, is defined to be $(\alpha,\kappa)$, where $\alpha$ is the unique curve without self-intersections enclosed by $\beta$ and $\kappa(t)=1$ for $t$ with $\alpha(t)\in\P$. In this way, each ideal triangulation $\T$ induces a tagged triangulation $\T^\times$ consisting of the tagged versions of all curves in $\T$.

For each ideal triangulation $\T$, there is an associated quiver with potential $(Q_\T,W_\T)$ (cf. \cite{FST,ILF}). In the paper, we only study $(Q_\T,W_\T)$ in the case when $\T$ is an admissible triangulation in the following sense (see Figure~\ref{fig:exT} for example).

\begin{definition}\label{def:admissible}
An ideal triangulation $\T$ is called admissible if every puncture in $\P$ is contained in a self-folded triangle in $\T$.
\end{definition}

In particular, in such a triangulation, the folded side of each self-folded triangle connect a marked point in $\M$ and a puncture in $\P$.

In an admissible triangulation $\T$, for a curve $\alpha\in\T$, let $\pi_\T(\alpha)$ be the curve defined as follows: if $\alpha$ is the folded side of a self-folded triangle in $\T$ (see the left picture of Figure~\ref{fig:self-folded}), then $\pi_\T(\alpha)$ is the corresponding remaining side (i.e. $\beta$ in  the left picture of Figure~\ref{fig:self-folded}); if there is no such triangle, set $\pi_\T(\alpha)=\alpha$. The associated quiver with potential $(Q_\T,W_\T)$ is given by the following data (see Figure~\ref{fig:first}):
\begin{itemize}
\item the vertices of $Q_\T$ are labeled by the curves in $\T$;
\item there is an arrow from $i$ to $j$ whenever there is a non-self-folded triangle in $\T$ having $\pi_\T(i)$ and $\pi_\T(j)$ as edges with $\pi_\T(j)$ following $\pi_\T(i)$ in the clockwise orientation (which is induced by the orientation of $\surf$). For instance, the quiver for a non-self-folded triangle is shown in Figure~\ref{fig:first}.
\item each subset $\{i,j,k\}$ of $\T$ with $\pi_\T(i)$, $\pi_\T(j)$, $\pi_\T(k)$ forming an interior non-self-folded triangle in $\T$ yields a unique 3-cycle up to cyclic permutation. The potential $W_{\T}$ is the sum of all such 3-cycles.
\end{itemize}

\begin{figure}[htpb]\centering
  \begin{tikzpicture}[scale=.7]
  \foreach \j in {1,...,3}  { \draw (120*\j-30:2) coordinate (v\j);}
    \path (v1)--(v2) coordinate[pos=0.5] (x3)
              --(v3) coordinate[pos=0.5] (x1)
              --(v1) coordinate[pos=0.5] (x2);
    \foreach \j in {1,...,3}{\draw (x\j) node[\RED] (x\j){};}
    \draw[->,>=stealth,\RED] (x1) to (x3);
    \draw[->,>=stealth,\RED] (x3) to (x2);
    \draw[->,>=stealth,\RED] (x2) to (x1);
    \draw[thick] (v1)node{$\bullet$}--(v2)node{$\bullet$}--(v3)node{$\bullet$}--cycle;
  \end{tikzpicture}
\caption{The quiver associated to a non-self-folded triangle}
\label{fig:first}
\end{figure}
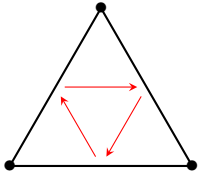

Then by Section~\ref{subsec:dg}, there is an associated cluster category, denoted by $\C(\T)$.

\subsection{Correspondence}

The objects and morphisms in $\hua{C}(\T)$ are expected to correspond to curves and intersection numbers, respectively. A \emph{cluster tilting object} $T=\oplus_{j=1}^n T_j$ in a cluster category $\C$ is an object satisfying $\Ext_{\C}^1(T,X)=0$ if and only if $X\in\add T$. The \emph{mutation} $\mu_i$ at the $i$-th indecomposable direct summand acts on a cluster tilting object $T=\oplus_{j=1}^n T_j$, by replacing $T_i$ with the unique indecomposable object $T'_i\ncong T_i$ satisfying $(T\setminus T_i)\oplus T'_i$ is a cluster tilting object.

In the unpunctured case, we have the following known results.

\begin{theorem}[\cite{BZ}]\label{thm:BZ}
If $\surf$ is unpunctured,
then there is a bijection between the set of curves and valued closed curves in $\surf$ and the set of indecomposable objects in $\C(\T)$. Under such a bijection,
\begin{itemize}
  \item rotation of curves is compatible with shift of objects;
  \item triangulations of $\surf$ one-to-one correspond to cluster tilting objects in $\C(\T)$ while
    flip of triangulations is compatible with mutation of cluster tilting objects.
\end{itemize}
\end{theorem}

\begin{theorem}[\cite{ZhZZ}]
If $\surf$ is unpunctured, then for any two curves $\gamma,\delta$, we have
\[\Int(\gamma,\delta)=\dim_\k\Ext_{\C(\T)}^1(X_\gamma,X_\delta),\]
where $X:\eta\mapsto X_\eta$ is the bijection in Theorem~\ref{thm:BZ}.
\end{theorem}

In the punctured case, we also know the following. Recall that a rigid indecomposable object in $\C(\T)$ is \emph{reachable} if it is a summand of some cluster tilting object, which is obtained from the canonical cluster tilting object by a sequence of mutations.

\begin{theorem}[\cite{BQ}]\label{thm:BQ}
Let $\TA$ be the set of tagged curves in $\TC$ without self-intersections and $\obj$ the set of reachable rigid indecomposable objects in $\C(\T)$. Under a canonical bijection
\[
    \varepsilon\colon\TA\to\obj,
\]
the tagged rotation $\rho$ on $\TA$ becomes the shift $[1]$ on $\obj$.
\end{theorem}

\subsection{Cluster(-tilting) exchange graphs}
The \emph{cluster(-tilting) exchange graph} $\CEG(\C)$ of a cluster category $\C$ is the graph whose vertices are cluster tilting objects and whose edges are mutations. There are the following known results about connectedness of cluster exchange graphs.

\begin{theorem}[\cite{BMRRT}]
If $\C$ is the cluster category of an acyclic quiver, then $\CEG(\C)$ is connected.
\end{theorem}

\begin{theorem}[\cite{BZ}]
If $\C$ is the cluster category from an unpunctured marked surface, then $\CEG(\C)$ is connected.
\end{theorem}

\section{Strings and tagged curves}\label{sec:3}
\subsection{Skewed-gentle algebras from admissible triangulations}\label{sec:QP2}

\def\X{\mathfrak{X}}
\def\bX{\overline{\X}}
\def\tX{\widetilde{\X}}

Let $\T$ be an admissible triangulation of $\surf$, i.e. every puncture is in a self-folded triangle (see Lemma~\ref{lem:ex} for the existence of $\T$), with
the associated quiver with potential $\left(Q_{\T},W_{\T}\right)$
and the cluster category $\C(\T)$.

Let $\T^o$ be the subset of $\T$ consisting of curves whose endpoints are in $\M$. Now we associate a biquiver $Q^\T=(Q^\T_0,Q^\T_1,Q^\T_2)$ with potential $W^\T$ as follows:
\begin{itemize}
  \item $Q_0^\T=\T^o$, i.e., curves in $\T$ which are sides of non-self-folded triangles correspond to vertices in $Q_0^\T$;
  \item there is a solid arrow from $i$ to $j$ in $Q_1^\T$ whenever there is a non-self-folded triangle $\Delta$ in $\T$ such that $\Delta$ has sides $i$ and $j$ with $j$ following $i$ in the clockwise orientation;
  \item there is a dashed loop at $i$ in $Q_2^\T$, denoted by $\varepsilon_i$,
        whenever $i$ is the remaining side of a self-folded triangle;
  \item each non-self-folded triangle in $\T$ induces
        a unique 3-cycle up to cyclic permutation.
        The potential $W^{\T}$ is the sum of all such 3-cycles.
\end{itemize}
See the example in Section~\ref{sec:ex}. By construction, there are no loops in $Q^\T_1$, any arrow in $Q^\T_2$ is a loop and there is at most one loop in $Q^\T_2$ at each vertex. Hence the biquiver $Q^\T$ satisfies the conditions on biquivers in Section~\ref{subsec:clannish}. Let $Z=\partial W^\T=\{\partial_a W^\T:a\in Q_1^\T\}$. We have the following straightforward observation.

\begin{lemma}\label{lem:z}
The set $Z$ consists of $\beta\alpha$ for each pair $\alpha,\beta\in Q_1^\T$ such that they are from the same non-self-folded triangle in $\T$.
\end{lemma}

\begin{proposition}\label{prop:clan}
The pair $(Q^\T,Z)$ is skewed-gentle and the algebra $\Lambda^\T:=\k Q^\T/(R)$ is a skewed-gentle algebra, where $R=Z\cup\{\varepsilon^2=\varepsilon\mid\varepsilon\in Q^\T_2\}$.
\end{proposition}
\begin{proof}
By Lemma~\ref{lem:z}, each element in $Z$ is of the form $\beta\alpha$ for some $\alpha,\beta$ in $Q_1^\T$. Since every vertex $i\in Q^\T_0$ is a side of a non-self-folded triangle in $\T$, there are two possible cases.
\begin{enumerate}
  \item[(1)] The curve $i$ is a common side of two non-self-folded triangles in $\T$, see Figure~\ref{fig:ordinary quiver}. Then there are no dashed loop at $i$ and there are at most two solid arrows $\alpha_1,\alpha_2$ ending at $i$ and at most two solid arrows $\beta_1,\beta_2$ starting at $i$. By Lemma~\ref{lem:z}, $\beta_1\alpha_1\in Z$, $\beta_2\alpha_2\in Z$, $\beta_1\alpha_2\notin Z$ and $\beta_2\alpha_1\notin Z$ (if exist).
  \item[(2)] The curve $i$ is a common side of a non-self-folded triangle and a self-folded triangle in $\T$, see Figure~\ref{fig:self-folded quiver}. Then there is a dashed loop at $i$ and there is at most one solid arrow $\alpha$ ending at $i$ and at most one solid arrow $\beta$ starting at $i$. By Lemma~\ref{lem:z}, $\beta\alpha\in Z$ (if exist).
\end{enumerate}
Hence by Definition~\ref{def:gentle}, $(Q^\T,Z)$ is skewed-gentle and the algebra $\Lambda^\T=\k Q^\T/(R)$ is a skewed-gentle algebra.
\end{proof}

\begin{figure}[htpb]\centering
	\begin{tikzpicture}[scale=.4]
	\draw(-1,0)node{};
	\draw[ultra thick]plot [smooth,tension=1] coordinates {(-120:6) (-95:4.5) (-60:6)};
	\draw[ultra thick]plot [smooth,tension=1] coordinates {(110:5) (85:4) (70:5)};
	\draw[ultra thick]plot [smooth,tension=1] coordinates {(30:5) (10:4) (-10:5)};
	\draw[ultra thick]plot [smooth,tension=1] coordinates {(150:5) (170:4) (185:5)};
	\draw[\RED,thick](-95:4.5)to[bend left](10:4)to[bend left](85:4)to(-95:4.5)
	to[bend right](170:4)to[bend right](85:4);
	\draw[<-,>=stealth,bend left](.9,2.65)to(0.25,2.5)node[below right]{$_{\beta_2}$};
	\draw[<-,>=stealth,bend left](0.15,2.4)to(-0.5,2.5);\draw(-0.2,2.3)node[below]{$_{\alpha_1}$};
	
	\draw[<-,>=stealth,bend left](-1,-2)to(-.25,-2);\draw(-.45,-2)node[above]{$_{\beta_1}$};
	\draw[<-,>=stealth,bend left](-.1,-1.5)to(0.55,-1.7);\draw(.4,-1.6)node[above]{$_{\alpha_2}$};
	\draw(0:8)node{$\Longrightarrow$};
	\draw[thick](-95:4.5)node{$\bullet$}(85:4)node{$\bullet$}
	(10:4)node{$\bullet$}(170:4)node{$\bullet$}(0,0.5)node[left]{$i$};
	\end{tikzpicture}
	\begin{tikzpicture}[xscale=.6,yscale=.4]
	\draw(0,0)node{$\xymatrix@R=1pc{
			\ar[dr]^{\alpha_1}&&\\&i\ar@{<-}[dr]_{\alpha_2}\ar[ur]^{\beta_2} \\\ar@{<-}[ur]_{\beta_1}&&}$};
	\draw[bend left,dashed](45:1)to(-45:1);\draw[bend right,dashed](135:1)to(-135:1);
	\draw(0,-5)node{$ $};
	\end{tikzpicture}
	\caption{Non-self-folded triangles with the corresponding quivers}
	\label{fig:ordinary quiver}
	
	\begin{tikzpicture}[scale=.4,rotate=0]
	\draw(0,6)node{};
	\draw[ultra thick]plot [smooth,tension=1] coordinates {(-120:5.5) (-95:4.5) (-60:5.5)};
	\draw[ultra thick]plot [smooth,tension=1] coordinates {(120:5) (85:4) (60:5)};
	\draw[\RED,thick](-95:4.5)to[bend left=60](85:4);
	\draw[\RED,thick](-95:4.5)to[bend right=60](85:4);
	
	\draw[\RED, thick](-95:4.5)
	.. controls +(60:6) and +(100:7) .. (-95:4.5) (0,0);
	\draw(0,1)node[below,black]{$i$};
	
	\draw[<-,>=stealth,bend left](-1.3,-3.5)to(-.6,-3.2);\draw(-1,-2.7)node{$_{\beta}$};
	\draw[<-,>=stealth,bend left](.4,-3)to(.8,-3.5);\draw(1,-2.7)node{$_{\alpha}$};
	\draw(7,0)node{$\Longrightarrow$};
	\draw[thick](-95:4.5)node{$\bullet$} (85:4)node{$\bullet$}(0,-1)node{$\bullet$};
	\end{tikzpicture}
	\begin{tikzpicture}[xscale=-.6,yscale=.4,rotate=180]
	\draw(0,0)node{$\xymatrix{
			\ar@{<-}[dr]_{\beta}&& \\&i \ar@{<-}[ur]_{\alpha}  \\\\ }$};
	\draw[bend right,dashed](-115:1.2)to(-65:1.2);
	\draw(0,-5)node{$ $};\draw(0,2)node{$\varepsilon$};
	\draw[->,>=stealth,dashed](120:.5).. controls +(135:2) and +(60:2) .. (45:.5);
	\end{tikzpicture}
	\caption{Self-folded triangles with the corresponding quivers}
	\label{fig:self-folded quiver}
\end{figure}

Let $Q_0^{\Sp}$ be the subset of $Q^\T_0$ consisting of vertices where there are dashed loops.

\begin{remark}\label{rem:compare}
Comparing the constructions of $(Q_\T,W_\T)$ and $(Q^\T,W^\T)$, one can obtain $(Q_\T,W_\T)$ (up to isomorphism) from $(Q^\T,W^\T)$ by splitting each vertex in $Q_0^{\Sp}$ into two vertices and removing all the dashed loops. More precisely,
\begin{itemize}
  \item vertices in $Q_\T$ are indexed by elements in $Q_0^\T\cup \{i'\mid i\in Q_0^{\Sp}\}$;
  \item arrows from $i$ to $j$ in $Q_\T$ are indexed by $^j\alpha^i$ induced by arrows $\alpha:\pi(i)\rightarrow \pi(j)$ in $Q^\T_1$, where $\pi$ is the map on $Q_0^\T\cup \{i'\mid i\in Q_0^{\Sp}\}$ sending $i'$ to $i$ and being identity on $Q_0^\T$;
  \item the potential $W_\T$ is the sum of cycles that are obtained from cycles in $W^\T$ by replacing each $\alpha$ with $^j\alpha^i$ for possible $i$ and $j$.
\end{itemize}
See the example in Section~6.
\end{remark}

Now we prove that the Jacobian algebra $\hua{P}(Q_\T,W_\T)$ is isomorphic to the skewed-gentle algebra $\Lambda^\T$.

\begin{proposition}\label{prop:iso}
There is an algebra isomorphism \[\varphi:\hua{P}(Q_{\T},W_{\T})\cong\Lambda^\T.\]
\end{proposition}

\begin{proof}
For each quiver, denote by $e_i$ the trivial path associated to a vertex $i$.
Noticing that $\varepsilon_i^2-\varepsilon_i\in R$ for each $i\in Q^{Sp}_0$,
a complete set of primitive orthogonal idempotents of $kQ^\T/(R^{\Sp})$ is
\[\{e_i+(R^{\Sp})\mid i\in Q_0\setminus Q_0^{\Sp}\}\cup\{\varepsilon_i+(R^{\Sp}), e_i-\varepsilon_i+(R^{\Sp})\mid i\in Q^{Sp}_0\},\]
where $R^{\Sp}=\{\varepsilon_i^2-\varepsilon_i\mid i\in Q_0^{\Sp}\}$.
Then using the recovery of $(Q_\T,W_\T)$ from $(Q^\T,W^\T)$ in Remark~\ref{rem:compare}, there is an isomorphism $\varphi$ of algebras from $kQ_\T$ to $kQ^\T/(R^{\Sp})$, which sends ${}^j\alpha^i$ to $\varphi(e_j)\alpha\varphi(e_i)$ where
\[\varphi(e_i)=\begin{cases}e_i+(R^{Sp})&\text{if $i\in Q^\T_0\setminus Q^{Sp}_0,$}\\ e_i-\varepsilon_i+(R^{Sp}) &\text{if $i\in Q^{Sp}_0$,}\\
\varepsilon_k+(R^{Sp}) & \text{if $i=k'$ for $k\in Q^{\Sp}_0$.}
\end{cases}\]
By \cite[Theorem 5.7]{ILF}, the Jacobian algebra $\hua{P}(Q_{\T},W_{\T})$ is isomorphic to
$kQ_{\T}/\partial W_{\T}$.
Then what is left to show is $\varphi(\partial W_{\T})=\partial W^{\T}$,
which directly follows from the recovery of $W_\T$ from $W^\T$ in Remark~\ref{rem:compare}.
\end{proof}

\begin{remark}
Note that Proposition~\ref{prop:clan} and Proposition~\ref{prop:iso} implies that
the Jacobian algebra $\hua{P}(Q_{\T},W_{\T})$ is a skewed-gentle algebra for any admissible triangulation $\T$
of a marked surface. This result was first announced by Labardini-Fragoso (cf. \cite{LF3}) and was first proved in \cite{GLS}.
\end{remark}

\subsection{Correspondence}\label{sec:corr.}

Denote by $\rep(Q^\T,W^\T)$ the category of finite dimensional $\k$-linear representations of $Q$ bounded by $R=\partial W^\T\cup\{\varepsilon^2=\varepsilon\mid\varepsilon\in Q^\T_2\}$.
By the equivalence \eqref{eq:eq} and Proposition~\ref{prop:iso}, we have an equivalence
\begin{equation}\label{eq:eq2}
F_\T\colon\C(\T)/(T_\T)\simeq\rep(Q^\T,W^\T).
\end{equation}
where $T_\T$ denotes the canonical cluster tilting object in $\C(\T)$ (cf. Section~\ref{subsec:dg}). So one can regard the set of indecomposable objects in $\C(\T)$ as
the union of the set of indecomposable representations in $\rep(Q^\T,W^\T)$ and the set of indecomposable direct summands of $T_\T$. Note that indecomposable direct summands of $T_\T$ are indexed by curves in $\T$ and hence also by tagged curves in $\T^\times$ (which is the tagged version of $\T$). Thus, we can write \begin{equation}\label{eq:ccto}
T_\T=\oplus_{(\gamma,\kappa)\in\T^\times}T_{(\gamma,\kappa)}.
\end{equation}

\begin{definition}\label{def:stringob}
We call an indecomposable object $X$ in $\C(\T)$ a \emph{string object} if either $X$ is a direct summand of the canonical cluster tilting object $T_\T$, or $F_\T(X)$ is an indecomposable representation in $\rep(Q^\T,W^\T)$ which is in the image of the injective map in Theorem~\ref{thm:Deng}. Denote the set of string objects in $\C(\T)$ by $\so(\T)$.
\end{definition}

\def\as{\mathfrak{a}}

To describe indecomposable representations in $\rep(Q^\T,W^\T)$, we shall use notions and notations stated in Section~\ref{app:clan}. Recall that from the biquiver $Q^\T$, in Section~\ref{subsec:letters}, we constructed a new biquiver $\widehat{Q}=(\widehat{Q}_0,\widehat{Q}_1,\widehat{Q}_2)$ by adding two new solid arrows $a_{i_\pm}$ (whose terminal vertices are also new vertices) for each vertex $i\in Q^\T_0$. Recall that a letter is an arrow in $\widehat{Q}$ or its inverse. For each $i\in Q^\T_0\subset\widehat{Q}_0$ the set $L(i)=\{l\in L\mid s(l)=i\}$ is divided into two disjoint subsets $L_\theta(i)$, $\theta\in\{\pm\}$, with linear orders satisfying certain conditions.

\begin{construction}\label{cstr:as}
We associate an arc segment $\as(l)$ to each letter $l$ as follows.
\begin{itemize}
  \item For $l=\alpha$ with $\alpha$ an arrow from $i$ to $j$ in $Q^\T_1\subset\widehat{Q}_1$, $\alpha$ is induced from a triangle $\Delta$ in $\T$ having $i$ and $j$ as sides with $j$ following $i$ in the clockwise orientation. Then we associate the arc segment (which is unique up to homotopy) in $\Delta$ starting at (a point in) $i$ and ending at (a point in) $j$.
  \item For $l=\varepsilon$ with $\varepsilon$ a dashed loop at $i$, we associate the arc segment in the self-folded triangle whose remaining side is $i$ with the clockwise orientation.
  \item For $l=a_{i_\theta}$ with $a_{i_\theta}$ an arrow in $\widehat{Q}_1\setminus Q^\T_1$, its associated segment is in the same triangle as the associated arc segments of other letters in $L_\theta(i)$. It starts at (a point in) $i$ and ends at a point in $\M\cup \P$.
  \item The arc segment $\as(l^{-1})$ is the same as $\as(l)$ but with the opposite orientation.
\end{itemize}
\end{construction}

See Figure~\ref{fig:seglet} (cf. Figure~\ref{fig:ordinary quiver} and \ref{fig:self-folded quiver}), where the disjoint subsets $L_\theta(i)$ are given by $L_+(i)=\{\alpha_1^{-1}>a_{i_+}>\beta_1\}$ and $L_-(i)=\{\alpha_2^{-1}>a_{i_-}>\beta_2\}$ for the left picture and by $L_+(i)=\{\alpha^{-1}>a_{i_+}>\beta\}$ and $L_-(i)=\{\varepsilon^{-1}>a_{i_-}>\varepsilon\}$ for the right picture. Note that $\alpha,\beta,\alpha_1,\beta_1,\alpha_2,\beta_2$ might not exist.

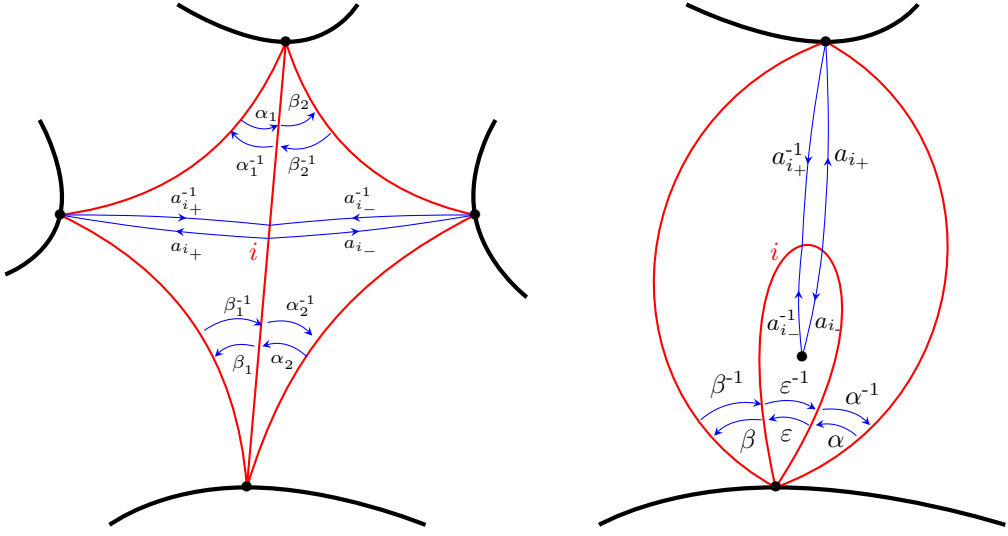
\begin{figure}[htpb]\centering
\begin{tikzpicture}[scale=.7]
\draw(-1,0)node{};
\draw[ultra thick]plot [smooth,tension=1] coordinates {(-120:6) (-95:4.5) (-60:6)};
\draw[ultra thick]plot [smooth,tension=1] coordinates {(110:5) (85:4) (70:5)};
\draw[ultra thick]plot [smooth,tension=1] coordinates {(30:5) (10:4) (-10:5)};
\draw[ultra thick]plot [smooth,tension=1] coordinates {(150:5) (170:4) (185:5)};
\draw[\RED,thick](-95:4.5)to[bend left=23](10:4)to[bend left](85:4)to(-95:4.5)
to[bend right](170:4)to[bend right](85:4);

\draw[->-=.6,>=stealth,bend left=3,\BLUE] (170:4) to (0.05,.5);
\draw(-1.5,.95)node{${_{a_{i_+}^\text{-1}}}$};
\draw[-<-=.6,>=stealth,bend left=-3,\BLUE] (170:4) to (0,0.25);
\draw(-1.5,0)node{${_{a_{i_+}}}$};

\draw[\RED](0,0)node[left]{$i$};

\draw[->-=.6,>=stealth,bend left=-3,\BLUE](10:4) to (0.05,.5);
\draw(1.8,.95)node{${_{a_{i_-}^\text{-1}}}$};
\draw[-<-=.6,>=stealth,bend left=3,\BLUE](10:4) to (0,.25);
\draw(1.8,0.05)node{${_{a_{i_-}}}$};

\draw[<-,>=stealth,bend left,\BLUE](.9,2.65)to(0.25,2.4);
\draw(0.21,2.5)node[above right]{${_{\beta_2}}$};
\draw[<-,>=stealth,bend left,\BLUE](0.2,2.4)to(-0.5,2.5);
\draw(0,2.3)node[above]{${_{\alpha_1}}$};

\draw[->,>=stealth,bend left,\BLUE](1.2,2.65-.2-.2)to(0.25,2.4-.2-.2);
\draw(0.21,2.5-.5)node[below right]{${_{\beta_2^\text{-1}}}$};
\draw[->,>=stealth,bend left,\BLUE](0.1,2.4-.4)to(-0.7,2.5-.2);
\draw(-0.3,2.3-.3)node[below]{${_{\alpha_1^\text{-1}}}$};

\draw[<-,>=stealth,bend left,\BLUE](-1,-2)to(-.25,-2+.2);
\draw(-.45,-1.8)node[below]{${_{\beta_1^{}}}$};
\draw[->,>=stealth,bend left,\BLUE](0,-1.35)to(0.85,-1.6);
\draw(.65,-1.4)node[above]{${_{\alpha_2^\text{-1}}}$};

\draw[->,>=stealth,bend left,\BLUE](-1.2,-2+.5)to(-.1,-2+.6);
\draw(-.545,-1.4)node[above]{${_{\beta_1^\text{-1}}}$};
\draw[<-,>=stealth,bend left,\BLUE](-.1,-1.5-.3)to(0.75,-1.7-.3);
\draw(.3,-1.8)node[below]{${_{\alpha_2^{}}}$};

\draw[thick](-95:4.5)node{$\bullet$}(85:4)node{$\bullet$}
    (10:4)node{$\bullet$}(170:4)node{$\bullet$};
\end{tikzpicture}
\qquad
\begin{tikzpicture}[yscale=.7,xscale=.9]
\draw(-1,0)node{}(0,6)node{};
\draw[ultra thick]plot [smooth,tension=1] coordinates {(-120:6) (-95:4.5) (-60:6)};
\draw[ultra thick]plot [smooth,tension=1] coordinates {(110:5) (85:4) (70:5)};
\draw[\RED,thick](-95:4.5)to[bend left=60](85:4);
\draw[\RED,thick](-95:4.5)to[bend right=60](85:4);

\draw[\RED, thick](-95:4.5).. controls +(63:6) and +(100:7) .. (-95:4.5);

\draw[->-=.6,>=stealth,bend left=-3,\BLUE] (85:4) to (0,.1);
\draw(-.17,1.75)node{${{a_{i_+}^\text{-1}}}$};
\draw[-<-=.6,>=stealth,bend left=3,\BLUE] (85:4) to (0.3,0.05);
\draw(.75,1.75)node{${{a_{i_+}}}$};
\draw[\RED](-.2,0)node[left]{$i$};

\draw[-<-=.6,>=stealth,bend left=-3,\BLUE] (0,-2) to (0.3,0.05);
\draw[->-=.6,>=stealth,bend right=-5,\BLUE] (0,-2) to (0,.1);
\draw(-0.25,-1.4)node{\small{${a_{i_-}^\text{-1}}$}} (.4,-1.4)node{\small{${a_{i_\text{-}}}$}};

\draw[<-,>=stealth,bend left,\BLUE](-.5,-3.2)to(.1,-3.3);
\draw(-.2,-3.2)node[below]{${{{\varepsilon}^{_{^{\text{}}}}}}$};
\draw[<-,>=stealth,bend left,\BLUE](-1.3,-3.5)to(-.6,-3.2);
\draw(-.8,-3.2)node[below]{${{\beta}}$};
\draw[<-,>=stealth,bend left,\BLUE](.2,-3.3)to(.8,-3.5);
\draw(.5,-3.6)node{${{\alpha}}$};

\draw[->,>=stealth,bend left,\BLUE](-.55,-3.2+.3)to(.2,-3.3+.3);
\draw(-.1,-3.2+.3+.8)node[below]{${{{\varepsilon}^{\text{-1}}}}$};
\draw[->,>=stealth,bend left,\BLUE](-1.5,-3.5+.3)to(-.6,-3.2+.3);
\draw(-1.1,-3.2+.4+.7)node[below]{${{\beta}^{\text{-1}}}$};
\draw[->,>=stealth,bend left,\BLUE](.3,-3.3+.3)to(1,-3.5+.2);
\draw(.9,-3.6+.3+.6)node{${{\alpha}^{\text{-1}}}$};

\draw[thick](-95:4.5)node{$\bullet$} (85:4)node{$\bullet$}(0,-2)node{$\bullet$};
\end{tikzpicture}
\caption{The arc segments associated to letters}
\label{fig:seglet}
\end{figure}

By Construction~\ref{cstr:as}, any arc segment associated to a letter in $L(i)$ starts at $i$. The following lemma gives some basic topological interpretation of notions about letters.

\begin{lemma}\label{lem:compprod}
For any two letters $l_1,l_2\in L(i)$,
\begin{enumerate}
  \item[(1)] $l_1$ and $l_2$ are in the same subset $L_\theta(i)$ if and only if $\as(l_1)$ and $\as(l_2)$ are in the same triangle;
  \item[(2)] $l_1>l_2$ if and only if they are of one of the forms in Figure~\ref{fig:orders}~(i) with
      $\as(l_i)$ being the segment of $\gamma_i$ cut out by the triangle;
  \item[(3)] $l=l_i$ is punctured if and only if one of the endpoints of $\as(l)$ is a puncture.
\end{enumerate}
\end{lemma}

\begin{figure}[htpb]\centering
	{\large(i)}\quad\qquad
	\begin{tikzpicture}[yscale=.4,rotate=180]
	\draw[\RED,thick,bend right=15] (0,0)to (2,3) to (2,-3) to(0,0);
	\draw (0,0)node{$\bullet$} (2,3)node{$\bullet$} (2,-3)node{$\bullet$};
	\draw[\BLUE,->-=.7,>=stealth,thick,bend right=15](3,-.5)to(.5,-1.5);
	\draw[\BLUE,->-=.7,>=stealth,thick,bend left=15](3,.5)to(.5,1.5);
	\draw[\BLUE,thick](2.3,1.6)node[above]{$\gamma_1$};
	\draw[\BLUE,thick](2.3,-1.6)node[below]{$\gamma_2$};
	\end{tikzpicture}\qquad
	\begin{tikzpicture}[yscale=.4,rotate=180]
	\draw[\RED,thick,bend right=15] (0,0)to (2,3) to (2,-3) to(0,0);
	\draw (0,0)node{$\bullet$} (2,3)node{$\bullet$} (2,-3)node{$\bullet$};
	\draw[\BLUE,->-=.7,>=stealth,thick,bend right=5](3,-.5)to(0,0);
	\draw[\BLUE,->-=.7,>=stealth,thick,bend left=15](3,.5)to(.5,1.5);
	\draw[\BLUE,thick](2.3,1.6)node[above]{$\gamma_1$};
	\draw[\BLUE,thick](2.3,-1.4)node[below]{$\gamma_2$};
	\end{tikzpicture}\qquad
	\begin{tikzpicture}[yscale=.4,rotate=180]
	\draw[\RED,thick,bend right=15] (0,0)to (2,3) to (2,-3) to(0,0);
	\draw (0,0)node{$\bullet$} (2,3)node{$\bullet$} (2,-3)node{$\bullet$};
	\draw[\BLUE,->-=.7,>=stealth,thick,bend left=5](3,.5)to(0,0);
	\draw[\BLUE,->-=.7,>=stealth,thick,bend right=15](3,-.5)to(.5,-1.5);
	\draw[\BLUE,thick](2.3,1.4)node[above]{$\gamma_1$};
	\draw[\BLUE,thick](2.3,-1.6)node[below]{$\gamma_2$};
	\end{tikzpicture}
	\begin{tikzpicture}[xscale=.8,yscale=1.3]
	\draw(-1.3,0)node[white]{x};
	\draw[thick,->-=.5,>=stealth,\BLUE](.7,.4)node[above]{$\gamma_2$}.. controls +(0:1) and +(120:1) ..(3.5,-.5);
	\draw[thick,->-=.5,>=stealth,\BLUE](.7,-.4)node[below]{$\gamma_1$}.. controls +(0:1) and +(-120:1) ..(3.5,.5);
	\draw[\RED,thick]plot [smooth,tension=1] coordinates {(4,0) (2.5,.6)(1.3,0) (2.5,-.6) (4,0)};
	\draw(4,0)node{$\bullet$}(2,0)node{$\bullet$};
	\end{tikzpicture}\qquad
	\begin{tikzpicture}[xscale=.8,yscale=1.3]
	\draw[thick,->-=.5,>=stealth,\BLUE](.7,.4)node[above]{$\gamma_2$}.. controls +(0:1) and +(120:1) ..(3.5,-.5);
	\draw[thick,->-=.7,>=stealth,\BLUE](.7,0)node[below]{$\gamma_1$}to(2,0);
	\draw[\RED,thick]plot [smooth,tension=1] coordinates {(4,0) (2.5,.6)(1.3,0) (2.5,-.6) (4,0)};
	\draw(4,0)node{$\bullet$}(2,0)node{$\bullet$};
	\end{tikzpicture}\qquad
	\begin{tikzpicture}[xscale=.8,yscale=1.3]
	\draw[thick,->-=.5,>=stealth,\BLUE](.7,-.4)node[below]{$\gamma_1$}.. controls +(0:1) and +(-120:1) ..(3.5,.5);
	\draw[thick,->-=.7,>=stealth,\BLUE](.7,0)node[above]{$\gamma_2$}to(2,0);
	\draw[\RED,thick]plot [smooth,tension=1] coordinates {(4,0) (2.5,.6)(1.3,0) (2.5,-.6) (4,0)};
	\draw(4,0)node{$\bullet$}(2,0)node{$\bullet$};
	\end{tikzpicture}
	
	\qquad
	
	\qquad
	
	\qquad
	
	\begin{tikzpicture}[yscale=.5,rotate=180]
	\draw(5.2,-1)node[left]{\large(ii)};
	\draw[\RED,thick,bend right=15] (0,0)to (2,3) to (2,-3) to(0,0);
	\draw (0,0)node{$\bullet$} (2,3)node{$\bullet$} (2,-3)node{$\bullet$};
	\draw[\BLUE,->-=.7,>=stealth,thick,bend left=25](3,-.7)to(0,0);
	\draw[\BLUE,->-=.7,>=stealth,thick,bend right=25](3,.7)to(0,0);
	\draw[\BLUE,thick](2.3,1.6)node[above]{$\gamma_1$};
	\draw[\BLUE,thick](2.3,-1.6)node[below]{$\gamma_2$};
	\end{tikzpicture}\qquad\qquad
	\begin{tikzpicture}[xscale=.8,yscale=1.3]
	\draw(0,1)node[white]{x} (0,-1)node[white]{x};
	\draw[thick,->-=.8,>=stealth,\BLUE](.7,.5)to[bend left](2,0);
	\draw[thick,\BLUE](1,.5)node[above]{$\gamma_2$};
	\draw[thick,\BLUE](1,-.5)node[below]{$\gamma_1$}(4,-1)node{$ $};
	\draw[thick,->-=.8,>=stealth,\BLUE](.7,-.5)to[bend right](2,0);
	\draw[\RED,thick]plot [smooth,tension=1] coordinates {(4,0) (2.5,.6)(1.3,0) (2.5,-.6) (4,0)};
	\draw(4,0)node{$\bullet$}(2,0)node{$\bullet$}(6,0)node{$ $};
	\end{tikzpicture}
	\caption{The orders}
	\label{fig:orders}
\end{figure}
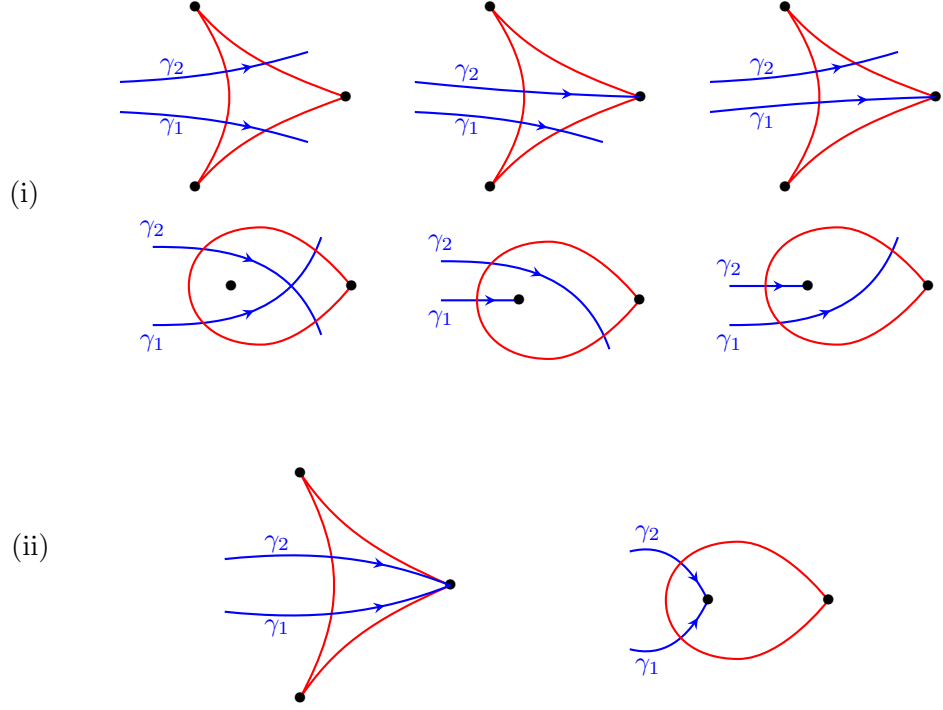

\begin{proof}
For (1), recall that if $\gamma,\delta^{-1}\in L(i)$ for two solid arrows $\gamma$ and $\delta$ then $\gamma$ and $\delta^{-1}$ are in the same subset $L_\theta(i)$ if and only if $\gamma\delta\in Z$. Then by Lemma~4.1, $\gamma$ and $\delta$ are from the same triangle. By Construction~\ref{cstr:as}, $\as(\gamma)$ and $\as(\delta^{-1})$ are in the same triangle. Also by  Construction~\ref{cstr:as}, $\as(\varepsilon)$ and $\as(\varepsilon^{-1})$ are in the same triangle for a dashed loop $\varepsilon$ and $\as(a_{i_\theta})$ is in the same triangle as the arc segments associated to other letters in $L_\theta(i)$. Hence we are done.

For (2), by definition we know two letters $l_1$ and $l_2$ are comparable if they are in the same subset $L_\theta(i)$. Then by (1) this is equivalent to that $\as(l_1)$ and $\as(l_2)$ are in the same triangle and start at the same curve $i$. Note that the forms in Figure~\ref{fig:orders}~(i) correspond to $\alpha^{-1}>\beta$ for $\beta\alpha\in Z$, $\alpha^{-1}>a_{i_\theta}$, $a_{i_\theta}>\beta$, $\varepsilon^{-1}>\varepsilon$, $a_{i_\theta}>\varepsilon$ and $\varepsilon^{-1}>a_{i_\theta}$, respectively. These give all the possible cases for the order.

For (3), by definition $l$ is punctured if and only if $l$ is of the form $a_{i_\theta}$ or $a_{i_\theta}^{-1}$ such that $L_{i_\theta}=\{\varepsilon^{-1}>a_{i_\theta}>\varepsilon\}$. Then by Construction~\ref{cstr:as}, $\as(l)$ is in a self-folded triangle and hence it has to connect the puncture in this triangle.
\end{proof}

Recall that a word is a sequence $\m=\omega_m\cdots\omega_1$ of letters in $L$ such that for any $1\leq j\leq m-1$, $\omega_j^{-1}\in L_\theta(i)$ and $\omega_{j+1}\in L_{\theta'}(i)$ for different $\theta,\theta'\in\{\pm\}$ and some $i\in Q^\T_0$. By Lemma~\ref{lem:compprod}~(1), this condition is equivalent to that $\as(\omega_j^{-1})$ and $\as(\omega_{j+1})$ start at the same curve in $\T$, but they are in the two adjacent triangles to this curve respectively. Hence we can glue the two arc segments $\as(\omega_j)$ and $\as(\omega_{j+1})$ to get a curve segment.

\begin{construction}
For each word $\m=\omega_m\cdots\omega_1$, we glue the corresponding arc segments $\as(\omega_1)$, $\cdots$, $\as(\omega_m)$ in order to get a curve segment, denoted by $\as(\m)$.
\end{construction}

Recall that for two words $\m=\omega_m\cdots\omega_2\omega_1$ and $\r=\nu_r\cdots\nu_2\nu_1$, $\m>\r$ if and only if there is $j$ such that $\omega_j\cdots\omega_1=\nu_j\cdots\nu_1$ and $\omega_{j+1}>\nu_{j+1}$.

\begin{lemma}\label{lem:order}
Let $\m_1$, $\m_2$ be two words and $\gamma_1=\as(\m_1), \gamma_2=\as(\m_2)$. Then $\m_1\geq\m_2$ if and only if $\gamma_1$ and $\gamma_2$ separate as in one of the forms in Figure~\ref{fig:orders} after they share the same curve segments from the start. The equality holds if and only if one of the forms in Figure~\ref{fig:orders}~(ii) occurs, and in this case $\gamma_1\sim\gamma_2$.
\end{lemma}
\begin{proof}
By definition, $\m_1=\omega_m\cdots\omega_1>\m_2=\omega'_r\cdots\omega'_1$ if and only if there is a maximal integer $j\geq 0$ such that the first $j$ letters (from right to left) of $\m_1$ and $\m_2$ are the same pointwise and $\omega_{j+1}>\omega'_{j+1}$. By Lemma~\ref{lem:compprod}~(2), this is equivalent to that $\as(\m_1)$ and $\as(\m_2)$ share the first $j$ arc segments and their $(j+1)$-th arc segments have one of the forms in Figure~\ref{fig:orders}~(i). Clearly $\m_1=\m_2$ if and only if one of the forms in Figure~\ref{fig:orders}~(ii) occurs and hence $\gamma_1\sim\gamma_2$. Thus the lemma holds.
\end{proof}

Recall that a word $\m=\omega_m\cdots\omega_1$ is maximal if and only if $\omega_1=a_{i_\theta}^{-1}$ and $\omega_m=a_{j_{\theta'}}$ for some $i,j\in Q_0^\T$ and some $\theta,\theta'\in\{\pm\}$.

\begin{lemma}\label{lem:bi1}
The map $\m \mapsto \as(\m)$ is a bijection from the set of maximal words to the set of curves (up to homotopy) in $\surf$ that are not in $\T$. Moreover, $\as(\m^{-1})=\as(\m)^{-1}$.
\end{lemma}

\begin{proof}
Let $\m=\omega_m\cdots\omega_1$. Since $\m$ is maximal, we have $\omega_1=a_{i_\theta}^{-1}$ and $\omega_m=a_{j_{\theta'}}$ for some $i,j\in Q_0^\T$ and some $\theta,\theta'\in\{\pm\}$. By Construction~\ref{cstr:as}, the endpoints of $\as(\m)$ are in $\M\cup\P$. Because by Lemma~\ref{lem:compprod}~(1), among the arc segments $\as(\omega_1),\cdots,\as(\omega_m)$, any two adjacent arc segments are not in the same triangle, so that the curve $\as(\m)$ has minimal intersections with the curves in $\T^o$. In particular, the intersection number of $\as(\m)$ with $\T$ is not zero. Hence $\as(\m)$ is a curve in $\surf$ which is not in $\T$.

On the other hand, for a curve $\gamma$ in $\CS\setminus\T$, since we consider it up to homotopy, we can assume $\gamma$ has minimal intersections with the curves in $\T$. Take the product, denoted by $\m_\gamma$, of letters corresponding to the arc segments of $\gamma$ divided by its intersections with $\T^o$ in order. Then by Lemma~\ref{lem:compprod}~(1), $\m_\gamma$ is a word and clearly it is maximal. Moreover, the correspondence between arc segments and letters implies that $\m_{\as(\m)}=\m$ and $\as(\m_{\gamma})=\gamma$ up to homotopy. Therefore, $\m \mapsto \as(\m)$ is the required bijection with $\as(\m^{-1})=\as(\m)^{-1}$.
\end{proof}

Recall from Section~\ref{subsec:strings} that a maximal word $\m=\omega_m\cdots\omega_1$ is called admissible if the following hold.
\begin{itemize}
  \item[(A1)] For each $\omega_i={\varepsilon}$ with $\varepsilon$ a dashed loop, we have that ${\omega}_1^{-1}\cdots{\omega}_{i-1}^{-1}>\omega_m\cdots\omega_{i+1}$, and for each $\omega_i={\varepsilon}^{-1}$ with $\varepsilon$ a dashed loop, we have that ${\omega}_1^{-1}\cdots{\omega}_{i-1}^{-1}<\omega_m\cdots\omega_{i+1}$.
  \item[(A2)] If $\m$ contains two punctured letters then $F(\m)$ is not a proper power of $F(\m')$ for any maximal word $\m'$ containing two punctured letters, where \[F(\m):=F(\omega_m)\cdots F(\omega_1)F(\omega_2^{-1})\cdots F(\omega_{m-1}^{-1}).\]
\end{itemize}
On the other hand, recall from Definition~\ref{def:curve} that a pair $(\gamma,\kappa)$ is called a tagged curve if the following hold.
\begin{enumerate}
  \item[(T1)] The curve $\gamma$ does not cut out a once-punctured monogon by a self-intersection (including endpoints), cf. Figure~\ref{fig:monogon};
  \item[(T2)] If $\gamma(0),\gamma(1)\in\P$, then the completion $\bar{\gamma}$ is not a proper power of a closed curve in the sense of the multiplication in the fundamental group of $\surf$.
\end{enumerate}

\begin{lemma}\label{lem:AT}
Let $\m$ be a maximal word. Then $\m$ satisfies (A1) if and only if $\as(\m)$ satisfies (T1); $\m$ satisfies (A2) if and only if $\as(\m)$ satisfies (T2).
\end{lemma}

\begin{proof}
Let $\m=\omega_m\cdots\omega_1$. Note that the curve $\as(\m)$ does not satisfies (T1) if and only if it cuts out a once-punctured monogon as in Figure~\ref{fig:monogon}. This is equivalent to that,
there is an arc segment $\as(\omega_i)$ of $\as(\m)$ with $\omega_i=\varepsilon$ or $\varepsilon^{-1}$
(for some dashed loop $\varepsilon$) which
has the form as in Figure~\ref{fig:gg7} with the right part being one of the forms in Figure~\ref{fig:orders}.
Let $\gamma_1:=\as(\omega_m\cdots\omega_{i+1})$ and $\gamma_2:=\as(\omega_1^{-1}\cdots\omega_{i-1}^{-1})$.
If $\omega_i=\varepsilon$, then the orientation of $\gamma$ is as shown in Figure~\ref{fig:gg7}.
\begin{figure}[htpb]\centering
	\begin{tikzpicture}[scale=0.56]
	\draw(-6,0)node{$\bullet$};
	\draw[\RED,thick]plot[smooth,tension=1.5] coordinates
	{(-8,0) (-6,1) (-6,-1)  (-8,0)};
	\draw[\BLUE,thick,bend right=75,-<-=.5,>=stealth](-6,1)to(-6,-1);
	\draw[dashed] (8,0) circle (2)node{right part}
	(-3,1)node[above,\BLUE]{$\gamma_{1}$}(-3,-1)node[below,\BLUE]{$\gamma_{2}$};
	\draw[dashed] (5,2)to(-5,2) (5,-2)to(-5,-2);
	\draw[\RED, thick] (5,2)(5,-2) (-5,-2)(-5,2);
	\draw[\RED, thick] (4,2)(4,-2) (-4,-2)to(-4,2);
	\draw[\BLUE, thick,->-=.4,>=stealth] (-5,1).. controls +(0:3.9) and +(180:3.9) .. (5,-1);
	\draw[\BLUE, thick,-<-=.4,>=stealth] (-5,-1).. controls +(0:3.9) and +(180:3.9) .. (5,1);
	\draw(2,0)node{$\cdots$}(-2,0)node{$\cdots$};
	\draw(-8,0)node{$\bullet$};
	\draw[\BLUE](-7,0)node{$\gamma$};
	\end{tikzpicture}
	\caption{A once-punctured monogon from a curve}
	\label{fig:gg7}
\end{figure}
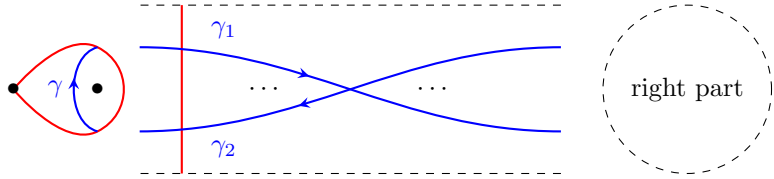

By Lemma~\ref{lem:order}, $\omega_m\cdots\omega_{i+1}\geq\omega_1^{-1}\cdots\omega_{i-1}^{-1}$. Similarly, if $\omega_i={\varepsilon}^{-1}$ then $\omega_m\cdots\omega_{i+1}\leq\omega_1^{-1}\cdots\omega_{i-1}^{-1}$. This implies that (T1) does not hold for $\as(\m)$ if and only if (A1) does not hold for $\m$.

Now consider the condition (A2). Note that by Lemma~\ref{lem:compprod}~(3) both of the endpoints of $\as(\m)$ are punctures if and only if $\m$ has two punctured letters $\omega_1$ and $\omega_m$. In this case, $F(\m)$ (as a cycle) corresponds to the completion of $\as(\m)$. Hence $\as(\m)$ satisfies (T2) if and only if $\m$ satisfies (A2).
\end{proof}

\def\U{\mathbf{C}_0(\surf)}

Let $\U$ be the subset of $\CS\setminus\T$ consisting of curves
satisfying the conditions (T1) and (T2). That is, $\U$ consists of curves $\gamma$ in $\CS\setminus\T$
such that there exists a tagged curve $(\gamma,\kappa)$ for some $\kappa$.

By Lemma~\ref{lem:bi1} and Lemma~\ref{lem:AT}, we have the following.

\begin{lemma}\label{cor:binew}
The map $\m\mapsto \as(\m)$ is a bijection
\begin{gather}\label{eq:as}
    \as\colon \bX  \longrightarrow \U,
\end{gather}
where $\bX$ is the set of admissible words up to inverse.
\end{lemma}

For each $\gamma\in\U$, we denote by $\m_\gamma$ the preimage of $\gamma$ under the bijection \eqref{eq:as}. Recall from Section~\ref{subsec:results} that the algebra $A_{\m_\gamma}$ is generated by the indeterminates $x$ associated to punctured letters in $\m_\gamma$ with $x^2=x$. So by Lemma~\ref{lem:compprod}~(3) indeterminates of $A_{\m_\gamma}$ are indexed by endpoints of $\gamma$ that are punctures.

\begin{construction}\label{cstr:N}
The 1-dimensional $A_{\m_\gamma}$-modules are classified in Section~\ref{subsec:results}. Using the notation there, each map $\kappa$ gives a 1-dimensional $A_{\m_\gamma}$-module $N(\gamma,\kappa)$ as follows.
\begin{enumerate}
  \item[i)] If neither of the endpoints of $\gamma$ is a puncture, then $A_{\m_\gamma}=\k$. Let $N_{(\gamma,\kappa)}=\k$.
  \item[ii)] If exactly one endpoint of $\gamma$ is a puncture, then $A_{\m_\gamma}=\k[x_a]/(x_a^2-x_a)$, where $a\in\{0,1\}$ with $\gamma(a)\in\P$. Let $N_{(\gamma,\kappa)}=\k_{\kappa(a)}$.
  \item[iii)] If both of the endpoints of $\gamma$ are punctures, then $A_{\m_{\gamma}}=\k\<x_0,x_1\>/(x_0^2-x_0,x_1^2-x_1)$. Let $N_{(\gamma,\kappa)}=\k_{\kappa(0),\kappa(1)}$.
\end{enumerate}
Thus, for each tagged curve $(\gamma,\kappa)\in\TC\setminus\T^\times$,
where $\T^\times$ is the tagged version of $\T$ (see Section~\ref{sec:QPtosurf}),
there is an associated indecomposable representation
\begin{gather}\label{eq:M^T}
    M^\T_{(\gamma,\kappa)}:=M^\T(\m_\gamma,N_{(\gamma,\kappa)})
\end{gather}
in $\rep(Q^\T,W^\T)$ by Construction~\ref{cstr:rep} and Theorem~\ref{thm:Deng}. For $(\gamma,\kappa)\in\T^\times$, let $M^\T_{(\gamma,\kappa)}=0$.
\end{construction}

\begin{definition}\label{def:X}
Define a map $X^\T$ from $\TC$ to the set $\so(\T)$ of string objects in $\C(\T)$ by
\[
X^\T_{(\gamma,\kappa)}=
\begin{cases}
M^\T_{(\gamma,\kappa)}&\text{if $(\gamma,\kappa)\in\TC\setminus\T^\times$,}\\
T_{(\gamma,\kappa)}&\text{if ${(\gamma,\kappa)}\in\T^\times$.}
\end{cases}
\]
\end{definition}

\begin{theorem}\label{thm:bi}
The map $X^\T$ is a bijection.
\end{theorem}

\begin{proof}
It is sufficient to prove that $X^\T$ is a bijection from the set $\TC\setminus\T^\times$ to the set of indecomposable representations $M^\T(\m,N)$ with $\m\in\bX$ and $\dim_\k N=1$. This follows from the bijection \eqref{eq:as} in Lemma~\ref{cor:binew} and
the description of the 1-dimensional $A_{\m_\gamma}$-modules in Construction~\ref{cstr:N}.
\end{proof}

\subsection{Flips and mutations}\label{sec:4}
We study $\lozenge$-flips of an admissible triangulation $\T$ in this subsection.
Recall that the function $\pi_\T$ on $\T$ is defined as follows: if $\gamma$ is the folded side of a self-folded triangle in $\T$,
then $\pi_\T(\gamma)$ is the corresponding remaining side;
otherwise, $\pi_\T(\gamma)=\gamma$.

\begin{definition}\cite[Definition~9.11]{FST}
Let $\T$ be an admissible triangulation of $\surf$. The $\lozenge$-flip $f_i(\T)$ associated to a curve $i\in\T^o$ is the unique admissible triangulation that shares all curves with $\T$ except for the curves $j$ satisfying $\pi_\T(j)=i$.
\end{definition}

Note that there are two types of $\lozenge$-flips: when $i$ is not a side of a self-folded triangle, the corresponding $\lozenge$-flip is an ordinary flip,
and when $i$ is the remaining side of a self-folded triangle, the corresponding $\lozenge$-flip is a combination of two ordinary flips occurring inside a once-punctured digon (see Figure~\ref{fig:lozenge}). Recall that an ordinary flip of a triangulation $\T'$ is a new triangulation $\T''$ which shares all curves with $\T'$ except for one.

The $\lozenge$-flips of an admissible triangulation $\T$ are indexed by the vertices of the quiver $Q^\T$. Define the mutation $(Q',W')=\mu_i(Q^\T,W^\T)$ at a vertex $i\in Q^\T_0$ to be $(Q'_0,Q'_1,W')=\mu_i(Q_0,Q_1,W)$ in the sense of \cite{DWZ} with $Q'_2=Q_2$. By \cite{FST,LF}, \[\mu_i(Q^\T,W^\T)\simeq(Q^{f_i(\T)},W^{f_i(\T)})\]
for each $i\in\T^o=Q^\T_0$, where $f_i$ is the $\lozenge$-flip associated to $i$. Then by \cite{KY}, there is an equivalence $\widetilde{\mu}_i:\C(\T)\simeq\C(\T')$.

For each tagged curve $(\gamma,\kappa)$, the corresponding object $X^\T_{(\gamma,\kappa)}$ in $\C(\T)$ is given by the associated representation $M^\T_{(\gamma,\kappa)}$ in $\rep(Q^\T,W^\T)$ through the equivalence $F_\T:\C(\T)/(T_\T)\to \rep(Q^\T,W^\T)$. Then the equivalence $\widetilde{\mu}_i$ should be compatible with some mutation of representations. However, there is not a well-defined mutation on $\rep(Q^\T,W^\T)$. Instead, we shall use decorated representations and their mutations introduced in \cite{DWZ}. A decorated representation of $(Q^\T,W^\T)$ is a pair $(M,V)$, where $M\in\rep(Q^\T,W^\T)$ is a usual representation of $(Q^\T,W^\T)$ and $V$ is a representation of $(Q_0^\T,Q_2^\T)$ bounded by $\{\varepsilon^2-\varepsilon\mid \varepsilon\in Q_2^\T\}$.

\begin{construction}\label{cstr:dec}
Let $(\gamma,\kappa)$ be a tagged curve in $\TC$.
If $(\gamma,\kappa)\notin\T^\times$, define $V^\T_{(\gamma,\kappa)}=0$;
if $(\gamma,\kappa)\in\T^\times$, define $V^\T_{(\gamma,\kappa)}$ by
\begin{itemize}
  \item $V_j=0$ and $V_{\varepsilon_j}=0$, for $j\neq\pi_\T(\gamma)$,
  \item $V_{\pi_\T(\gamma)}=\k$;
  \item $V_{\varepsilon_{\pi_\T(\gamma)}}=1-\kappa(a)$ if there exists a (unique) $a$ with $\gamma(a)\in\P$.
\end{itemize}
This construction, together with \eqref{eq:M^T}, gives a decorated representation
$(M^\T_{(\gamma,\kappa)},V^\T_{(\gamma,\kappa)})$ associated to each tagged curve $(\gamma,\kappa)$. \end{construction}

Let $(\gamma,\kappa)$ be a tagged curve in $\TC$ and $(M^\T_{(\gamma,\kappa)},V^\T_{(\gamma,\kappa)})$ the corresponding decorated representation. Construct $\mu_i(M^\T_{(\gamma,\kappa)},V^\T_{(\gamma,\kappa)})$ as in Appendix~~\ref{app:DWZ}.
Now we prove that the map $X^\T$ from tagged curves to string objects is independent of the choice of the admissible triangulation $\T$.

\begin{theorem}\label{thm:ind}
For any two admissible triangulations $\T$ and $\T'$, there is an equivalence $\Theta\colon\C(\T)\simeq\C(\T')$ such that $\Theta\left(X^\T_{(\gamma,\kappa)}\right)\cong X^{\T'}_{(\gamma,\kappa)}$, for every tagged curve $(\gamma,\kappa)\in\TC$.
\end{theorem}

\begin{proof}
By Lemma~\ref{lem:loz-conn}, any two admissible triangulations are connected by a sequence of $\lozenge$-flips. Then using induction,
it is sufficient to consider the case of $\T'=f_i(\T)$ a $\lozenge$-flip of $\T$.
Recall that there is an equivalence $\widetilde{\mu}_i:\C(\T)\simeq\C(\T')$.

We claim that for any tagged curve $(\gamma,\kappa)\in\TC$,
\begin{equation}\label{eq:mu1}
\mu_i(M^\T_{(\gamma,\kappa)},V^\T_{(\gamma,\kappa)})\cong(M^{\T'}_{(\gamma,\kappa)},V^{\T'}_{(\gamma,\kappa)})
\end{equation}
Indeed, since $M^\T_{(\gamma,\kappa)}$ and $V^\T_{(\gamma,\kappa)}$ are constructed locally, we only need to prove that for each segment of $\gamma$ crossing $i$, the corresponding decorated representations $(M,V)$ of $(Q^\T,W^\T)$ and $(M',V')$ of $(Q^{f_i(\T)}, W^{f_i(\T)})$ satisfy $\mu_i(M,V)\cong (M',V')$. We list all the possible cases in Table~\ref{table1} and Table~\ref{table2} in Appendix~\ref{app:DWZ} for the first and second types of $\lozenge$-flips, respectively, up to symmetry. And in the same row of each case, we list the corresponding decorated representations, using Construction~\ref{cstr:rep} and Construction~\ref{cstr:dec}. Then one can check \eqref{eq:mu1} on a case by case basis.

Let $F^\T$ denote the equivalence \eqref{eq:eq2}. Consider the map $\Phi_\T$ from the set of (isoclasses of) objects in $\C(\T)$ to the set of (isoclasses of) decorated representations of $(Q^\T,W^\T)$ defined as follows. For any indecomposable object $X\in\C(\T)$ which is not isomorphic to a direct summand of $T_\T$, define $\Phi_\T(X)\cong(F_\T(X),0)$; for any tagged curve $(\gamma,\kappa)\in\T^\times$, define $\Phi_\T(T_{(\gamma,\kappa)})=(0,V^\T_{(\gamma,\kappa)})$. By definition, for any tagged curve $(\gamma,\kappa)\in\TC$, $\Phi_\T(X^\T_{(\gamma,\kappa)})\cong (M^\T_{(\gamma,\kappa)},V^\T_{(\gamma,\kappa)})$. Hence by \eqref{eq:mu1}, we have
\[\mu_i\left(\Phi_\T(X^\T_{(\gamma,\kappa)})\right)=\Phi_{\T'}(X^{\T'}_{(\gamma,\kappa)}).\]
By \cite[Proposition~4.1]{P}, $\Phi_\T$ is a bijection and $\mu_i\Phi_\T(X)\simeq\Phi_{\T'}\widetilde{\mu}_i(X)$ for any indecomposable object $X$ in $\C(\T)$. Hence $\widetilde{\mu}_i\left(X^\T_{(\gamma,\kappa)}\right)\cong X^{\T'}_{(\gamma,\kappa)}$, as required.
\end{proof}

\section{Homological interpretation of marked surfaces}\label{sec:5}
\subsection{AR-translation and AR-triangles}
Note that we have chosen an admissible triangulation $\T$ and have a bijection $X^\T:\TC\rightarrow\so(\T)$ from the set of tagged curves to the set of string objects in the cluster category $\C(\T)$.
The tagged rotation (cf. Definition~\ref{def:rotation} and Figure~\ref{fig:rotate})
$\rho$ is a permutation on $\TC$ while the shift functor $[1]$ in $\hua{C}(\T)$ gives a permutation on the set $\so(\T)$. We will give a straightforward proof of Theorem~\ref{thm:BQ},
with a slight generalization to tagged curves.

For a curve $\gamma$ in $\CS$ with $\gamma(1)\in\M$, denote by $\bt\gamma$ the curve obtained from $\gamma$ by moving $\gamma(1)$ along the boundary anticlockwise to the next marked point; dually, for a curve $\gamma$ in $\CS$ with $\gamma(0)\in\M$, denote by $\gamma\bt$ the curve obtained from $\gamma$ by moving $\gamma(0)$ along the boundary anticlockwise to the next marked point. We first show the following lemma, where $\m_\gamma$ is the word associated to $\gamma$ defined by the bijection \eqref{eq:as} and $\bt\m$ and $\m\bt$ are defined in Section~\ref{subsec:order}. Recall that the set $\U$ consists of curves $\gamma$ in $\CS\setminus\T$
such that there exists a tagged curve $(\gamma,\kappa)$ for some $\kappa$.

\begin{lemma}\label{lem:+}
If $\gamma$ is a curve in $\U$ with $\gamma(1)\in\M$ such that $\bt\gamma$ is in $\U$, then $\bt\m_\gamma$ exists and $\bt\m_\gamma=\m_{\bt\gamma}$. Dually, if $\gamma$ is a curve in $\U$ with $\gamma(0)\in\M$ such that $\gamma\bt$ is in $\U$, then $\m_\gamma\bt$ exists and $\m_\gamma\bt=\m_{\gamma\bt}$.
\end{lemma}

\begin{proof}
We only prove the first assertion.
By construction, $\gamma$ and $\bt\gamma$ start at the same point, go through the same way at the beginning and then separate in a triangle as one of the following two forms, where $\delta$ is the boundary segment from $\gamma(1)$ to $\bt\gamma(1)$ anticlockwise.
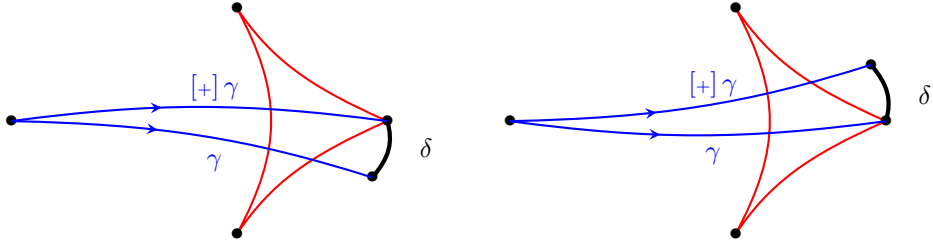
\[\begin{tikzpicture}[yscale=.5,rotate=180]
\draw (5,0)node{$\bullet$} (.2,1.5)node{$\bullet$};
\draw[\RED,thick,bend right=15] (0,0)to (2,3) to (2,-3) to(0,0);
\draw[thick](-.3,.7)node[right]{$\delta$};
\draw[ultra thick,bend right=-15] (0,0) to (.2,1.5);
\draw (0,0)node{$\bullet$} (2,3)node{$\bullet$} (2,-3)node{$\bullet$};
\draw[\BLUE,->-=.4,>=stealth,thick,bend left=15](5,0)to(0,0);
\draw[\BLUE,->-=.4,>=stealth,thick,bend left=15](5,0)to(.2,1.5);
\draw[\BLUE,thick](2.3,1.6)node[above]{$\gamma$};
\draw[\BLUE,thick](2.3,-1.5)node[below]{$\bt\gamma$};
\end{tikzpicture}\qquad
\begin{tikzpicture}[yscale=.5,rotate=180]
\draw (5,0)node{$\bullet$} (.2,-1.5)node{$\bullet$};
\draw[ultra thick,bend left=-15] (0,0) to (.2,-1.5);
\draw[thick](-.3,-.7)node[right]{$\delta$};
\draw[\RED,thick,bend right=15] (0,0)to (2,3) to (2,-3) to(0,0);
\draw (0,0)node{$\bullet$} (2,3)node{$\bullet$} (2,-3)node{$\bullet$};
\draw[\BLUE,->-=.4,>=stealth,thick,bend right=15](5,0)to(0,0);
\draw[\BLUE,->-=.4,>=stealth,thick,bend right=15](5,0)to(.2,-1.5);
\draw[\BLUE,thick](2.3,1.4)node[above]{$\gamma$};
\draw[\BLUE,thick](2.3,-1.6)node[below]{$\bt\gamma$};
\end{tikzpicture}\]
By Lemma~\ref{lem:order}, $\m_\gamma>\m_{\bt\gamma}$.
Moreover, since by construction $\gamma$, $\bt\gamma$ and $\delta$ enclose a contractible triangle in the surface (i.e. a triangle which is homotopic to a point), by Lemma~\ref{lem:order} again there is no curve $\gamma'\in\U$ starting at $\gamma(0)=\bt\gamma(0)$ such that $\m_\gamma>\m_{\gamma'}>\m_{\bt\gamma}$.
Therefore, the bijection \eqref{eq:as} between curves in $\U$ and words in $\bX$
implies that $\m_{\bt\gamma}$ is the successor of $\m_\gamma$, i.e. $\bt\m_\gamma=\m_{\bt\gamma}$.
\end{proof}

\begin{theorem}\label{thm:T-rotation}
Under the bijection $X^\T:\TC\rightarrow\so(\T)$, the tagged rotation $\rho$ on $\TC$ becomes the shift $[1]$ on the set $\so(\T)$, i.e. we have the following commutative diagram
\[\xymatrix{
\TC\ar[r]^{X^\T}\ar[d]^{\rho}&\so(\T)\ar[d]^{[1]}\\ \TC\ar[r]^{X^\T}&\so(\T)
}.\]
In particular, restricting to $\TA$, we get Theorem~\ref{thm:BQ}.
\end{theorem}
\begin{proof}
Let $(\gamma,\kappa)\in\TC$ such that neither $(\gamma,\kappa)$ nor $\rho(\gamma,\kappa)$ is in $\T^\times$.
So both $\gamma$ and $\rho(\gamma)$ are in $\U$.
Using Theorem~\ref{thm:G}, we show that $\tau M^\T_{(\gamma,\kappa)}=M^\T_{\rho(\gamma,\kappa)}$,
where there are three cases:
\begin{itemize}
  \item[(1)] If both $\gamma(0)$ and $\gamma(1)$ are in $\M$, then
  at least one of $\bt\gamma$ and $\gamma\bt$ is in $\U$.
  This is because when they are both in $\T$, one of
  $(\gamma,\kappa)$ and $\rho(\gamma,\kappa)$ is forced to be in $\T$, which contradicts our assumption.
  Then we deduce that $\m_{\rho(\gamma)}=\bt\m_\gamma\bt$ by Lemma~\ref{lem:+}.
  On the other hand, $\m_\gamma$ contains no punctured letters by Lemma~\ref{lem:compprod}~(3).
  Therefore, $$\tau M^\T_{(\gamma,\kappa)}=M^\T(\bt\m_\gamma\bt,\k)=M^\T_{\rho(\gamma,\kappa)},$$ where $\kappa=\emptyset$.
  \item[(2)] If exactly one of $\gamma(1)$ and $\gamma(0)$ is a puncture,
  assume that $\gamma(0)\in\P$ and $\gamma(1)\in\M$ without loss of generality. Then $\rho(\gamma)=\bt\gamma$, $\m_{\rho(\gamma)}=\bt\m_\gamma$ and
      $$\tau M^\T_{(\gamma,\kappa)}=M^\T(\bt\m_\gamma,\k_{1-\kappa(0)})=M^\T_{\rho(\gamma,\kappa)}.$$
  Note that the part $\k_{1-\kappa(0)}$ is constructed in Construction~\ref{cstr:N}, which is determined by the tagging.
  \item[(3)] If both $\gamma(0)$ and $\gamma(1)$ are in $\P$, then $\rho(\gamma)=\gamma$ and
  $$\tau M^\T_{(\gamma,\kappa)}=M^\T(\m_\gamma,\k_{1-\kappa(0),1-\kappa(1)})=M^\T_{\rho(\gamma,\kappa)}.$$
\end{itemize}
By \cite[Lemma in \S 3.5]{KR},
the shift $[1]$ in the triangulated category $\C(\T)$ gives the AR-translation $\tau$ in $\rep(Q^\T,W^\T)$. Then we have
\begin{gather}\label{eq:TAU-}
    X^{\T}_{(\gamma,\kappa)}[1]=X^{\T}_{\rho(\gamma,\kappa)},
\end{gather}
for $(\gamma,\kappa)\notin\T^\times\cup\rho^{-1}(\T^\times)$.
Furthermore, $M^\T_{(\gamma,\kappa)}$ is a projective representation for $(\gamma,\kappa)\in\rho^{-1}(\T^\times)$ and $P_{(\gamma,\kappa)}[1]=T_{(\gamma,\kappa)}$ for any $(\gamma,\kappa)\in\T^\times$, where $P_{(\gamma,\kappa)}$ is the projective representation associated to the primitive idempotent indexed by $(\gamma,\kappa)$. In particular,
\[\{X^\T_{(\gamma,\kappa)}[1]\mid(\gamma,\kappa)\in\rho^{-1}(\T^\times)\}
    =\{X^\T_{(\gamma,\kappa)}\mid(\gamma,\kappa)\in\T^\times\}.\]

To finish the proof, we only need to show that
$X_{\rho^{\pm1}(\gamma,\kappa)}=T_{(\gamma,\kappa)}[\pm1]$ for $(\gamma,\kappa)\in\T^\times$. There are two cases and we only deal $\rho^{-1}(\gamma,\kappa)$. Note that $\rho(\gamma,\kappa)$ intersects $(\gamma,\kappa)$ and the following figure shows the local situation near this intersection.
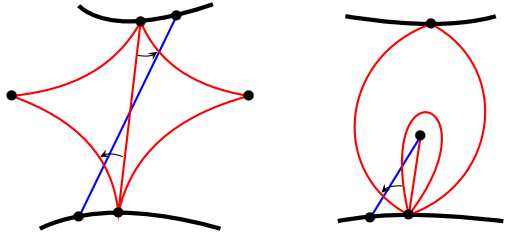
\begin{figure}[htpb]\centering
\begin{tikzpicture}[xscale=.4,yscale=.3]
\draw[\BLUE,thick](70:4.5)to(-110:5);
\draw(-1,0)node{};
\draw[ultra thick]plot [smooth,tension=1] coordinates {(-120:6) (-95:4.5) (-60:6)};
\draw[ultra thick]plot [smooth,tension=1] coordinates {(110:5) (90:4) (60:5.5)};
\draw[\RED,thick](-95:4.5)to[bend left](10:4)to[bend left](85:4)to(-95:4.5)
to[bend right](170:4)to[bend right](85:4);
\draw[<-,>=stealth,bend left](.9,2.65)to(0.25,2.5);
\draw[<-,>=stealth,bend left](-1,-2)to(-.25,-2);
\draw[thick](-95:4.5)node{$\bullet$}(85:4)node{$\bullet$}(70:4.5)node{$\bullet$} (-110:5)node{$\bullet$}
    (10:4)node{$\bullet$}(170:4)node{$\bullet$};
\end{tikzpicture}
\qquad
\begin{tikzpicture}[xscale=.4,yscale=.3]
\draw[\BLUE,thick](0,-1)to(-110:5);
\draw(0,6)node{};
\draw[ultra thick]plot [smooth,tension=1] coordinates {(-120:5.5) (-95:4.5) (-60:5.5)};
\draw[ultra thick]plot [smooth,tension=1] coordinates {(120:5) (85:4) (60:5)};
\draw[\RED,thick](-95:4.5)to[bend left=60](85:4);
\draw[\RED,thick](-95:4.5)to[bend right=60](85:4);
\draw[\RED, thick](-95:4.5)
.. controls +(60:6) and +(100:7) .. (-95:4.5);
\draw[\RED, thick] (0,-1)to(-95:4.5);
\draw[<-,>=stealth,bend left](-1.3,-3.5)to(-.6,-3.2);
\draw[thick](-95:4.5)node{$\bullet$}(85:4)node{$\bullet$}(0,-1)node{$\bullet$}(-110:4.9)node{$\bullet$};
\end{tikzpicture}\end{figure}

This enables us to deduce that
$M^\T_{\rho^{-1}(\gamma,\kappa)}$ is not any projective other than $P_{(\gamma,\kappa)}$.
So
$M^\T_{\rho^{-1}(\gamma,\kappa)}=P_{(\gamma,\kappa)}$ and
$X_{\rho^{-1}(\gamma,\kappa)}[1]=T_{(\gamma,\kappa)}$.
\end{proof}

\begin{remark}\label{rem:AR}
Using the description of AR-sequences in \cite[Section~5.4]{G}
and their relation with AR-triangles in \cite[Proposition~4.7]{KZ},
we can describe the AR-triangle ending at $X^\T_{(\gamma,\kappa)}$ where $(\gamma,\kappa)$ is a tagged curve in $\TC$ that does not connect two punctures.
In the case when $(\gamma,\kappa)$ does connect two punctures, the middle term
of the AR-triangle is not a string object (and hence we do not have a description).

In the following, let $\bar{\delta}$ be the completion of a curve $\delta$ which connects a marked point in $\M$ and a puncture in $\P$ (see the top picture of Figure~\ref{fig:completion} for the induced orientation). Let $X=X^\T$ and
$$X_{(\bar{\delta},\emptyset)}:=X_{(\delta,\emptyset)}\oplus X_{(\delta,\kappa')},$$
where $\kappa'(t)=1$ for $t$ with $\delta(t)\in\P$.
\begin{itemize}
  \item If both $\gamma(0)$ and $\gamma(1)$ are in $\M$ (which implies $\kappa=\emptyset$), the AR-triangle ending at $X_{(\gamma,\emptyset)}$ is
\begin{gather*}
    X_{\rho(\gamma,\emptyset)}\rightarrow X_{({}\gamma\bt,\emptyset)}\oplus X_{(\bt\gamma,\emptyset)}\rightarrow X_{(\gamma,\emptyset)} \rightarrow.
\end{gather*}
  \item If exactly one of $\gamma(0)$ and $\gamma(1)$ is in $\P$, the AR-triangle ending at $X_{(\gamma,\kappa)}$ is
\begin{gather*}
X_{\rho(\gamma,\kappa)}\rightarrow X_{(\bt\bar{\gamma},\emptyset)}\rightarrow X_{(\gamma,\kappa)} \rightarrow .
\end{gather*}
\end{itemize}
\end{remark}

\subsection{Cutting and Calabi-Yau reductions}\label{sec:cut}

Given a tagged curve $(\gamma,\kappa)\in\TC$ without self-intersections (i.e. $\Int((\gamma,\kappa),(\gamma,\kappa))=0$),
let $\surf/(\gamma,\kappa)$ be the marked surface obtained from $\surf$ by cutting along $\gamma$. More precisely, there are five cases listed below.
\begin{itemize}
\item[(1)]
If $\gamma$ connects two different marked points $M_1, M_2\in\M$,
then the resulting surface is shown as in the first column of Figure~\ref{fig:5c}.
\item[(2)]
If $\gamma$ is a loop based on a marked point $M\in\M$,
then the resulting surface is shown as in the second column of Figure~\ref{fig:5c}.
\item[(3)]
If $\gamma$ connects a marked point $M\in\M$ and a puncture $P\in\P$,
then the resulting surface is shown as in the third column of Figure~\ref{fig:5c}.
\item[(4)]
If $\gamma$ connects two different punctures $P_1,P_2\in\P$,
then the resulting surface is shown as in the fourth column of Figure~\ref{fig:5c}.
\item[(5)]
If $\gamma$ is a loop based on a puncture $P\in\P$,
then the resulting surface is shown as in the fifth column of Figure~\ref{fig:5c}.
\end{itemize}
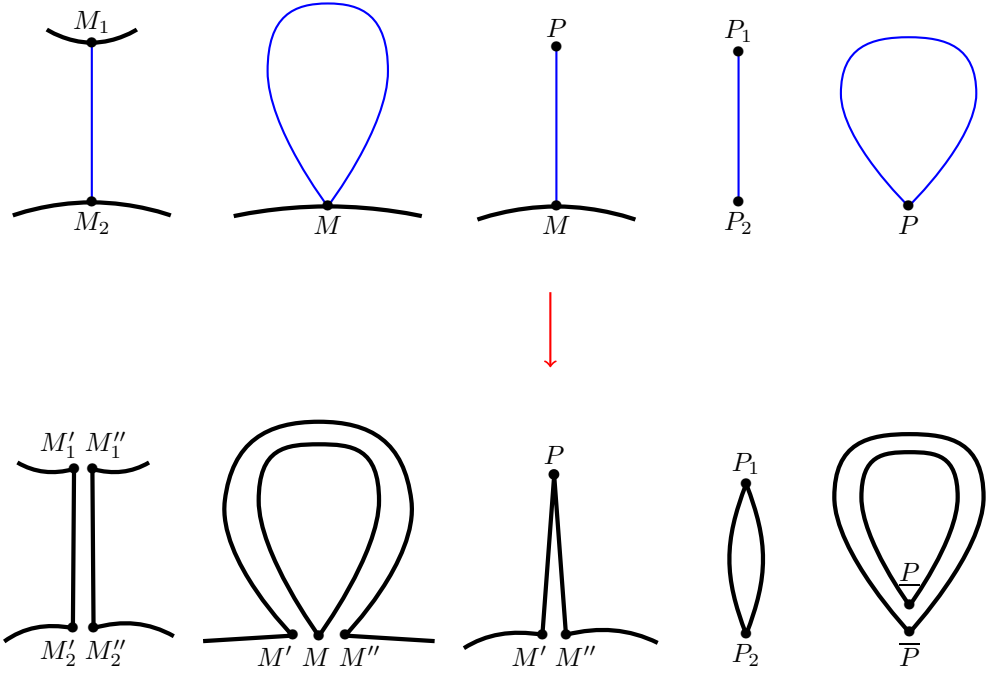
\begin{figure}[htpb]\centering
	\begin{tikzpicture}[xscale=.35,yscale=.25]
	\draw[ultra thick]plot [smooth,tension=1] coordinates {(-120:6) (-90:4.5) (-60:6)};
	\draw[ultra thick]plot [smooth,tension=1] coordinates {(110:5) (90:4) (70:5)};
	\draw[\BLUE,thick](90:4)to(-90:4.5);
	\draw[thick](90:4)node[above]{$M_1$}node{$\bullet$} (-90:4.5)node[below]{$M_2$}node{$\bullet$};
	\end{tikzpicture}
	\quad\;\;
	\begin{tikzpicture}[xscale=1.6,yscale=.3,rotate=90]
	\draw[ultra thick]plot [smooth,tension=1] coordinates {(170:4.5) (180:4) (190:4.5)};
	\draw[\BLUE,thick]plot [smooth,tension=1] coordinates {(-4,0) (2,.5) (5,0) (2,-.5) (-4,0)};
	\draw[thick](180:4)node{$\bullet$}node[below]{$M$};
	\end{tikzpicture}
	\quad\;
	\begin{tikzpicture}[xscale=.35,yscale=.25]
	\draw(-1,0)node{};
	\draw[\BLUE,thick](90:4)to(-90:4.5);
	\draw[thick](90:4)node[above]{$P$}node{$\bullet$};
	\draw[ultra thick]plot [smooth,tension=1] coordinates {(-120:6) (-90:4.5) (-60:6)};
	\draw[thick](-90:4.5)node{$\bullet$}node[below]{$M$};
	\end{tikzpicture}
	\quad\quad\;
	\begin{tikzpicture}[xscale=1,yscale=.25]
	\draw[\BLUE,thick](0,4)to(0,-4);
	\draw[thick](0,4)node{$\bullet$}node[above]{$P_1$}(0,-4)node{$\bullet$}node[below]{$P_2$};
	\end{tikzpicture}
	\qquad\;
	\begin{tikzpicture}[xscale=1.8,yscale=.25,rotate=90]
	\draw[\BLUE,thick]plot [smooth,tension=1] coordinates {(-4,0) (2,.5) (5,0) (2,-.5) (-4,0)};
	\draw[thick](180:4)node{$\bullet$}node[below]{$P$};
	\end{tikzpicture}\quad
	
	\begin{tikzpicture}[]
	\draw[thick,\RED](0,0)edge[->=stealth](0,-1) (6,-1.5)node{}(-7.5,.5)node{};
	\end{tikzpicture}
	
	\begin{tikzpicture}[xscale=.35,yscale=.25]
	\draw[ultra thick](-120:6)edge[bend left](-95:4.5);
	\draw[ultra thick](-85:4.5)edge[bend left](-55:6);
	\draw[ultra thick](120:5)edge[bend right](95:4);
	\draw[ultra thick](85:4)edge[bend right](60:5);
	\draw[ultra thick](90:4)node[above]{$M_1'\;M_1''$}(85:4)node{$\bullet$}to(-85:4.5)node{$\bullet$}
	(95:4)node{$\bullet$}to(-95:4.5)node{$\bullet$}
	(-90:4.5)node[below]{$M_2'\;M_2''$};
	\end{tikzpicture}
	\;
	\begin{tikzpicture}[xscale=1,yscale=.3,rotate=90]
	\draw[ultra thick]plot [smooth,tension=1] coordinates {(160:4.5) (175:4)};
	\draw[ultra thick]plot [smooth,tension=1] coordinates {(185:4) (200:4.5)};
	\draw[ultra thick]plot [smooth,tension=1] coordinates {(175:4) (2,1.25) (5.5,0) (2,-1.23) (185:4)};
	\draw[ultra thick]plot [smooth,tension=1] coordinates {(180:4) (2,.8) (4.5,0) (2,-.8) (180:4)};
	\draw[thick](175:4)node{$\bullet$}(185:4)node{$\bullet$}
	(180:4)node{$\bullet$}node[below]{$M'$ ${M}$ $M''$};
	\end{tikzpicture}
	\;
	\begin{tikzpicture}[xscale=.4,yscale=.25]
	\draw(-1,0)node{};
	\draw[ultra thick](-120:6)edge[bend left](-95:4.5);
	\draw[ultra thick](-85:4.5)edge[bend left](-55:6);
	\draw[ultra thick](90:4)node[above]{$P$}node{$\bullet$}(-85:4.5)node{$\bullet$}to
	(90:4)node{$\bullet$}to(-95:4.5)node{$\bullet$}
	(-90:4.5)node[below]{$M'\;M''$};
	\end{tikzpicture}
	\qquad
	\begin{tikzpicture}[xscale=.55,yscale=.25]
	\draw[ultra thick](0,4)edge[bend left=10](0,-4)
	edge[bend right=10](0,-4);
	\draw(0,4)node[above]{$P_1$}node{$\bullet$}(0,-4)node[below]{$P_2$}node{$\bullet$};
	\end{tikzpicture}
	\qquad
	\begin{tikzpicture}[xscale=.8,yscale=.24,rotate=90]
	\draw[ultra thick]plot [smooth,tension=1]
	coordinates {(180:5.5) (2,1.25) (5.5,0) (2,-1.23) (180:5.5)};
	\draw[ultra thick]plot [smooth,tension=1] coordinates {(180:4) (2,.8) (4.5,0) (2,-.8) (180:4)};
	\draw[thick](180:5.5)node{$\bullet$}node[below]{$\overline{P}$}
	(180:4)node{$\bullet$}(180:3.5)node[above]{$\underline{P}$} (6,0)node{};
	\end{tikzpicture}
	\caption{Cutting: five cases}
	\label{fig:5c}
\end{figure}

There is a canonical bijection between the tagged curves in $\surf$ that do not intersect $(\gamma,\kappa)$ and the tagged curves in $\surf/(\gamma,\kappa)$. This bijection is straightforward to see for the cases (1), (2) and (5), while there are several non-obvious correspondence between the tagged curves in the cases (3) and (4), which have been shown in Figure~\ref{fig:corres} (up to tagging). It is easy to check that this bijection preserves the intersection numbers.

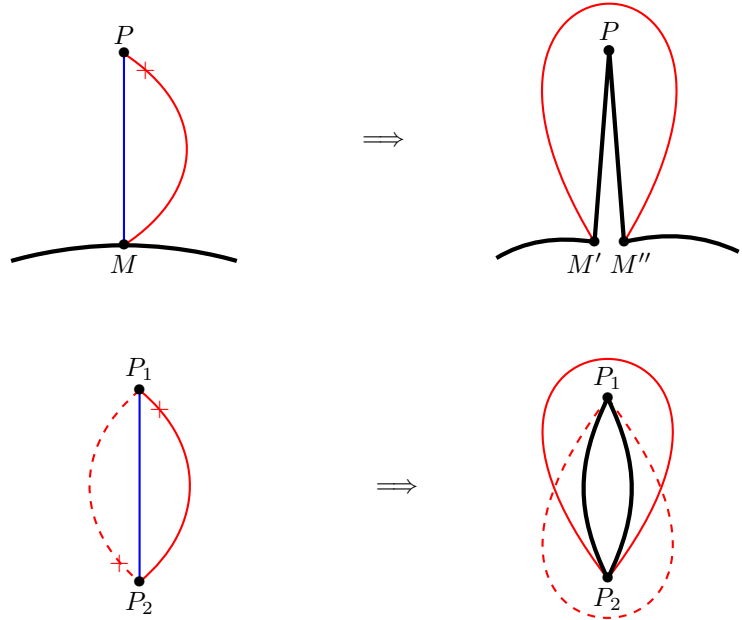
\begin{figure}[htpb]\centering
\qquad
\begin{tikzpicture}[xscale=.5,yscale=.3]
\draw[\RED,thick](90:4)to[bend left=42.5](-90:4.5)
(80:3.3)node{+};
\draw(-1,0)node{};
\draw[\BLUE,thick](90:4)to(-90:4.5);
\draw[thick](90:4)node[above]{$P$}node{$\bullet$};
\draw[ultra thick]plot [smooth,tension=1] coordinates {(-120:6) (-90:4.5) (-60:6)};
\draw[thick](-90:4.5)node[]{$\bullet$}node[below]{$M$};
\end{tikzpicture}
\qquad\qquad
\begin{tikzpicture}[xscale=.5,yscale=.3]
\draw(-6,0)node{$\Longrightarrow$};
\draw[\RED,thick](-85:4.5).. controls +(70:15) and +(110:15) .. (-95:4.5);
\draw[ultra thick](-120:6)edge[bend left](-95:4.5);
\draw[ultra thick](-85:4.5)edge[bend left](-55:6);
\draw[ultra thick](90:4)node[above]{$P$}node{$\bullet$}(-85:4.5)node{$\bullet$}to
(90:4)node{$\bullet$}to(-95:4.5)node{$\bullet$}
(-90:4.5)node[below]{$M'\;M''$};
\end{tikzpicture}

\qquad\qquad\quad
\begin{tikzpicture}[xscale=.4,yscale=.3]
\draw[\RED,thick](90:4)node[below right]{+}to[bend left=42.5](-90:4.5);
\draw[\RED,thick,dashed](90:4)to[bend right=42.5](-90:4.5)node[above left]{+};
\draw[\BLUE,thick](0,4)to(0,-4.5);
\draw[thick](0,4)node{$\bullet$}node[above]{$P_1$}
    (0,-4.5)node{$\bullet$}node[below]{$P_2$};
    \draw(2,-9)node{};
\end{tikzpicture}\qquad\qquad\quad\quad
\begin{tikzpicture}[xscale=.4,yscale=.3]
\draw(-7,0)node{$\Longrightarrow$};
\draw[\RED,thick](0,-4).. controls +(60:15) and +(120:15) .. (0,-4);
\draw[\RED,thick,dashed](0,4).. controls +(-60:15) and +(-120:15) .. (0,4);
\draw[ultra thick](0,4)edge[bend left=20](0,-4)
    edge[bend right=20](0,-4);
\draw(0,4)node[above]{$P_1$}node{$\bullet$}(0,-4)node[below]{$P_2$}node{$\bullet$};
\end{tikzpicture}
\caption{The non-trivial bijections for cutting}
\label{fig:corres}
\end{figure}

Let $\R$ be a subset of an admissible triangulation $\T$.
We define $\surf/\R$ to be the marked surface obtained from $\surf$ by cutting successively along each tagged curve in $\R$. Clearly the new marked surface is independent of the choice of orders of tagged curves in $\R$ and it inherits an admissible triangulation $\T\setminus\R$ from $\surf$. Denote by $\TC_{\R}$ the set of tagged curves $(\gamma,\kappa)$ in $\TC\setminus\R$
that do not intersect the tagged curves in $\R$. By induction, there is a canonical bijection from $\TC_{\R}$ to $\mathbf{C}^\times(\surf/\R)$. For each tagged curve $(\gamma,\kappa)\in\TC_{\R}$, we still use the notation $(\gamma,\kappa)$ to denote its image under this bijection.

One the other hand, the object $R=\bigoplus_{(\gamma,\kappa)\in\R}X^\T_{(\gamma,\kappa)}$ is a direct summand of the canonical cluster tilting object $T_\T$ (cf. \eqref{eq:ccto}) in $\C(\T)$.
Then the Calabi-Yau reduction ${}^\bot R[1]/(R)$ is a $2$-Calabi-Yau category with a cluster tilting object $T_\T\setminus R$ (see \cite[Section~4]{IY}). The following lemma will be used in the proof of the main result in the next subsection. Indeed, this generalizes a result in \cite{MP} on the relation between Calabi-Yau reduction and cutting to the punctured case.

\begin{lemma}\label{lem:cutting}
Let $\R$ be a subset of an admissible triangulation $\T$ and $R=\bigoplus_{(\gamma,\kappa)\in\R}X^\T_{(\gamma,\kappa)}$.
Then there is a canonical triangle equivalence $\xi\colon{}^\bot R[1]/(R)\simeq\C(\T\setminus\R)$
satisfying
\begin{gather}\label{eq:fff}
    \xi\left(X^\T_{(\gamma,\kappa)}\right)\cong X^{\T\setminus\R}_{(\gamma,\kappa)}
\end{gather}
for each $(\gamma,\kappa)\in\TC_{\R}$.
\end{lemma}
\begin{proof}
Noticing that the corresponding biquiver with potential $(Q^{\T\setminus{\R}},W^{\T\setminus{\R}})$ can be obtained from
the biquiver with potential $(Q^{\T_{}},W^{\T_{}})$
by deleting the vertices corresponding to the tagged curves in $\R$.
By \cite[Theorem~7.4]{K09},
the canonical projection $\pi:\Lambda^\T\to\Lambda^{\T\setminus{\R}}$ induces the required equivalence $\xi\colon {}^\bot \hua{R}[1]/\hua{R} \simeq\C(\surf/\R)$.
Furthermore, for each tagged curve $(\gamma,\kappa)\in\TC_{\R}$,
the support of the representation $M^\T_{(\gamma,\kappa)}$ does not contain  the vertices corresponding to tagged curves in $\R$.
Hence it is preserved by the projection $\pi$.
Thus we deduce that \eqref{eq:fff} holds.
\end{proof}


\subsection{Intersection numbers}

\begin{theorem}\label{thm:Int}
Let $(\gamma_1,\kappa_1)$ and $(\gamma_2,\kappa_2)$ be two tagged curves in $\TC$. Then
\[
\Int((\gamma_1,\kappa_1),(\gamma_2,\kappa_2))=
\dim_\k\Ext^1_{\C(\T)}(X^\T_{(\gamma_1,\kappa_1)},X^\T_{(\gamma_2,\kappa_2)})
\]
for any admissible triangulation $\T$ of $\surf$.
\end{theorem}
\begin{proof}
The proof is given in Section~\ref{sec:proof}.
\end{proof}

\subsection{Connectedness of cluster exchange graphs}
We apply our main result to study the exchange graph $\CEG(\C(\T))$ of the cluster category $\C(\T)$.

\begin{corollary}\label{cor:bi}
The correspondence $X^\T$ in Theorem~\ref{thm:bi} induces bijections
\begin{enumerate}
  \item between tagged curves without self-intersections in $\surf$ and indecomposable rigid objects in $\C(\T)$;
  \item between tagged triangulations of $\surf$ and cluster tilting objects in $\C(\T)$.
\end{enumerate}
Moreover, under the last bijection, flip of tagged triangulations is compatible with mutation of cluster tilting objects.
\end{corollary}
\begin{proof}
By Remark~\ref{rem:De}, for any indecomposable representation $M$ of $(Q^\T,W^\T)$, if there is no tagged curve $(\gamma,\kappa)$ such that $M\cong M^\T_{(\gamma,\kappa)}$, then $\Hom_{\Lambda^\T}(M,\tau M)\neq 0$. This implies that $\Ext^1_{\hua{C}(\T)}(M,M)\neq0$. Hence $M$ is not a rigid object in $\C(\T)$.
Thus, all rigid objects are string objects and hence the proposition follows from Theorem~\ref{thm:Int}.
\end{proof}

\begin{theorem}\label{thm:conn}
The cluster exchange graph $\CEG(\C(\T))$ is connected and $\CEG(\C(\T))\cong\EGT(\surf)$.
In particular, each rigid object in $\C(\T)$ is reachable.
\end{theorem}
\begin{proof}
The isomorphism $\EGT(\surf)\cong\CEG(\C(\T))$ of graphs follows directly from Corollary~\ref{cor:bi}. The connectedness of $\EGT(\surf)$ is proved in \cite[Proposition 7.10]{FST}.
\end{proof}

\section{An example}\label{sec:ex}
Let $\surf$ be a disk with three marked points on the boundary and two punctures in the interior. The corresponding cluster category $\C(\T)$ is the classical cluster category of type $\widetilde{D}_5$.
Let $\T$ be the admissible triangulation shown in the top left picture of Figure~\ref{fig:exT}.
\begin{figure}[h]\centering
\begin{tikzpicture}[scale=.7]
\draw[ultra thick](0,0)circle(4);
\draw[\RED,thick](180+10:4).. controls +(-45:6) and +(135:6) .. (0-10:4);
\draw[\RED,thick](0-10:4).. controls +(160:5) and +(230:5) .. (0-10:4);
\draw[\RED,thick](180+10:4).. controls +(160:-5) and +(230:-5) .. (180+10:4);
\draw[\RED,thick](180+10:4).. controls +(75:4.5) and +(105:4.5) .. (0-10:4);
\draw[\RED,thick](-10:4)+(195:2.3)to(-10:4) (190:4)+(195:-2.3)to(190:4);
\draw[thick,\RED](-1.2,.9)node{1}(0,2.9)node{2}(1.5,.8)node{3}(.7,-1.8)node{4};
\draw[\BLUE,>=stealth,-<-=.15](-.95,0)to(0,-.7);
\draw[\BLUE,>=stealth,->-=.95](-.95,0)to(0,2.55);
\draw[\BLUE,>=stealth,->-=.95](0,2.55)to(0,-.7);
\draw[\BLUE,>=stealth,->-=.95](0,-.7)to(.95,-1.4);
\draw(-.7,1.4)node{$a$}(.3,1.2)node{$b$}(-.6,-.5)node{$c$}(.5,-.7)node{$d$};
\draw[very thick](-10:4)+(195:2.3)node{$\bullet$}(190:4)+(195:-2.3)node{$\bullet$};
\draw[very thick](90:4)node{$\bullet$}(0-10:4)node{$\bullet$}(180+10:4)node{$\bullet$};
\end{tikzpicture}
\qquad
\begin{tikzpicture}[scale=.7]
\draw[ultra thick](0,0)circle(4);
\draw[\RED,thick](180+10:4).. controls +(-45:6) and +(135:6) .. (0-10:4);
\draw[\RED,thick](0-10:4).. controls +(160:5) and +(230:5) .. (0-10:4);
\draw[\RED,thick](180+10:4).. controls +(160:-5) and +(230:-5) .. (180+10:4);
\draw[\RED,thick](180+10:4).. controls +(75:4.5) and +(105:4.5) .. (0-10:4);
\draw[\BLUE,>=stealth,->-=.6,bend right](-2.8,-1)to(-3,0.2);
\draw[\BLUE,>=stealth,->-=.6,bend right](2.5,-.35)to(2.8,-1.75);
\draw(-2.3,-.3)node{$\varepsilon_1$} (2.2,-1.3)node{$\varepsilon_4$};
\draw[\BLUE,>=stealth,->-=.6](0,2.55)to(0,4);
\draw[\BLUE,>=stealth,->-=.1] (0,2.55).. controls +(-90:3) and +(-30:3)..(190:4);
\draw(-.4,3.2)node{$a_{2_-}$} (-.4,2)node{$a_{2_+}$};
\draw[\BLUE,>=stealth,-<-=.6](-10:4)+(195:2.3)to(.95,-1.4);
\draw[\BLUE,>=stealth,-<-=.6](190:4)+(195:-2.3)to(-.95,0);
\draw[\BLUE,>=stealth,->-=.4](-.95,0).. controls +(30:3) and +(115:3)..(0-10:4);
\draw[\BLUE,>=stealth,->-=.4](.95,-1.4).. controls +(200:3) and +(-60:4)..(180+10:4);
\draw(1.35,-1.4)node[above]{$a_{4_-}$} (-1.35,-.1)node[above]{$a_{1_-}$}
(1,1)node[above]{$a_{1_+}$} (-2,-2.5)node{$a_{4_+}$};
\draw[\BLUE,>=stealth,->-=.5](0,-.7).. controls +(90:2) and +(70:4)..(180+10:4);
\draw[\BLUE,>=stealth,->-=.5](0,-.7).. controls +(-90:2) and +(70:-4)..(0-10:4);
\draw(-1.7,1.7)node{$a_{3_+}$} (1.7,-2.7)node{$a_{3_-}$};
\draw[very thick](-10:4)+(195:2.3)node{$\bullet$}(190:4)+(195:-2.3)node{$\bullet$};
\draw[very thick](90:4)node{$\bullet$}(0-10:4)node{$\bullet$}(180+10:4)node{$\bullet$};
\end{tikzpicture}

\begin{tikzpicture}[scale=.7]
\draw[ultra thick](0,0)circle(4);
\draw[\RED,thick](180+10:4).. controls +(-45:6) and +(135:6) .. (0-10:4);
\draw[\RED,thick](0-10:4).. controls +(160:5) and +(230:5) .. (0-10:4);
\draw[\RED,thick](180+10:4).. controls +(160:-5) and +(230:-5) .. (180+10:4);
\draw[\RED,thick](180+10:4).. controls +(75:4.5) and +(105:4.5) .. (0-10:4);
\draw[\BLUE,thick,->-=.5,>=stealth](-10:4)+(195:2.3).. controls +(-135:3.5) and +(-60:5)..(190:4);
\draw[\BLUE,thick,->-=.5,>=stealth](-10:4)+(195:2.3).. controls +(120:3.5) and +(55:5)..(190:4);
\draw[\BLUE,thick](0,1.9)node[below]{$\gamma_2$} (-110:3.1)node[right]{$\gamma_1$};
\draw[\BLUE,thick](1.56,-1)node{$+$};
\draw[very thick](90:4)node{$\bullet$}(0-10:4)node{$\bullet$}(180+10:4)node{$\bullet$};
\draw[very thick](-10:4)+(195:2.3)node{$\bullet$}(190:4)+(195:-2.3)node{$\bullet$};
\end{tikzpicture}
\begin{tikzpicture}[scale=.7]
\draw[ultra thick](0,0)circle(4);
\draw[\RED,thick](180+10:4).. controls +(-45:6) and +(135:6) .. (0-10:4);
\draw[\RED,thick](0-10:4).. controls +(160:5) and +(230:5) .. (0-10:4);
\draw[\RED,thick](180+10:4).. controls +(160:-5) and +(230:-5) .. (180+10:4);
\draw[\RED,thick](180+10:4).. controls +(75:4.5) and +(105:4.5) .. (0-10:4);
\draw[\BLUE,thick,->-=.5,>=stealth](-10:4)+(195:2.3)coordinate(v1)to(-10:4);
\draw[\BLUE,thick,-<-=.77,>=stealth]plot [smooth,tension=1] coordinates {(-10:4) (0,-3.5) (-2,0) (0,1.5) (v1)};
\draw[\BLUE,thick](85:2.5)node[below]{$\rho(\gamma_2)$} (2.6,-.6)node{$\rho(\gamma_1)$} (2,-1.23)node{$\times$};
\draw[very thick](90:4)node{$\bullet$}(0-10:4)node{$\bullet$}(180+10:4)node{$\bullet$};
\draw[very thick](-10:4)+(195:2.3)node{$\bullet$}(190:4)+(195:-2.3)node{$\bullet$};
\end{tikzpicture}
\caption{An example}\label{fig:exT}
\end{figure}
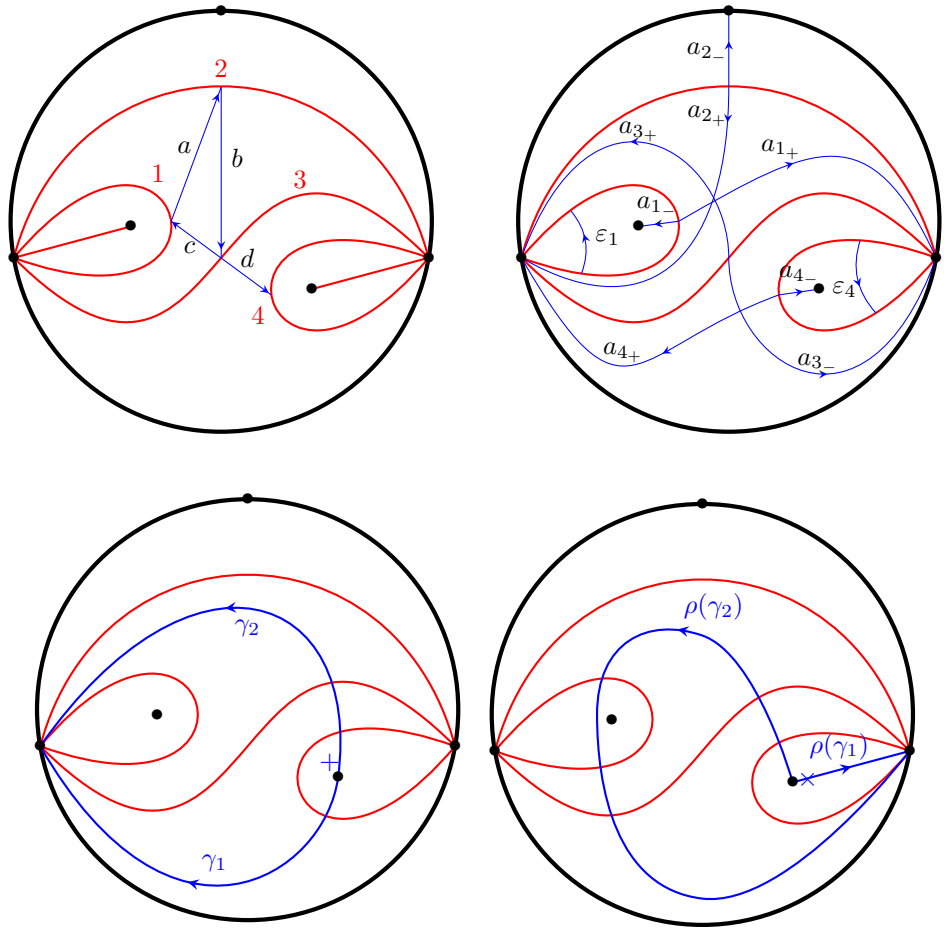
Then the associated biquiver with potential $(Q^\T,W^\T)$ introduced in Section~\ref{sec:QP2} is
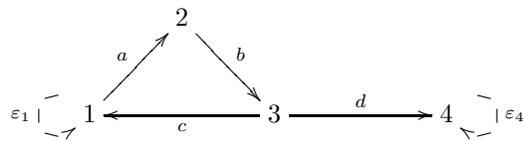
\[
\xymatrix{
&2\ar[dr]^b\\
1\ar@{-->}@(lu,ld)[]_{\varepsilon_1}\ar[ur]^a&&3\ar[ll]^c\ar[rr]^d&&4\ar@{-->}@(ru,rd)[]^{\varepsilon_4}
}\]
with $W^\T=cba$ and the associated quiver with potential $(Q_\T,W_\T)$ introduced in Section~\ref{sec:QPtosurf} is
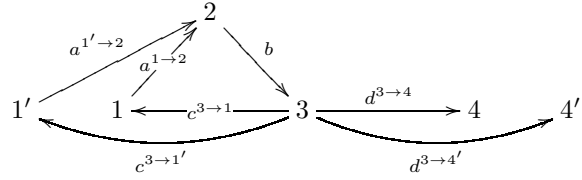
\[
\xymatrix{
&&2\ar[dr]^b\\
1'\ar[urr]^{a^{1'\to 2}}&1\ar[ur]|{a^{1\to2}}&&
3\ar@/^1pc/[lll]^{c^{3\to1'}}\ar[ll]|{c^{3\to1}}\ar[rr]^{d^{3\to 4}}\ar@/_1pc/[rrr]_{d^{3\to4'}}&&4&4'
}\]
with $W_\T=c^{3\to 1}b a^{1\to2}+c^{3\to 1'}b a^{1'\to2}$ (where the terms $c^{3\to 1'}b a^{1\to2}$ and $c^{3\to 1}b a^{1\to2}$ do not appear in $W_\T$ because they are not cycles).

We have $Z=\partial W^\T=\{ba,cb,ac\}$. Then the biquiver $Q^\T$ with $Z$ is precisely the one in Example~\ref{exm:clan}. Use the choice of disjoint subsets given in Example~\ref{exm:disjoint}:
\begin{itemize}
  \item $L_+(1)=\{c^{-1}>a_{1_+}>a\}$ and $L_-(1)=\{\varepsilon_1^{-1}>a_{1_-}>\varepsilon_1\}$;
  \item $L_+(2)=\{a^{-1}>a_{2_+}>b\}$ and $L_-(2)=\{a_{2_-}\}$;
  \item $L_+(3)=\{b^{-1}>a_{3_+}>c\}$ and $L_-(3)=\{a_{3_-}>d\}$;
  \item $L_+(4)=\{d^{-1}>a_{4_+}\}$ and $L_-(4)=\{\varepsilon_4^{-1}>a_{4_-}>\varepsilon_4\}$.
\end{itemize}
Using Construction~\ref{cstr:as}, the arc segments corresponding to direct letters are shown in the top pictures of Figure~\ref{fig:exT}, and inverse of letters corresponds to reverse of direction of arc segments.

Consider the tagged curves $(\gamma_1,\kappa_1)$ and $(\gamma_2,\kappa_2)$ shown in the bottom left figure of Figure~\ref{fig:exT}, where $\kappa_1(0)=0$ and $\kappa_2(0)=1$. The rotations $\rho(\gamma_1,\kappa_1)$ and $\rho(\gamma_2,\kappa_2)$ are shown in the bottom right figure of Figure~\ref{fig:exT}. The intersection number between $(\gamma_1,\kappa_1)$ and $(\gamma_2,\kappa_2)$ is
\[\Int((\gamma_1,\kappa_1),(\gamma_2,\kappa_2))=1.\]
On the representation side, since $\rho(\gamma_1,\kappa_1)\in\T^\times$, we have $M^\T_{\rho(\gamma_1,\kappa_1)}=0$ by Construction~\ref{cstr:N}. Hence
\[\begin{array}{rcl}
&&\dim\Ext^1_{\C(\T)}(X^\T_{(\gamma_1,\kappa_1)},X^\T_{(\gamma_2,\kappa_2)})\\
&=&\dim\Hom_{\Lambda^\T}(M^\T_{(\gamma_1,\kappa_1)},M^\T_{\rho(\gamma_2,\kappa_2)})
+\dim\Hom_{\Lambda^\T}(M^\T_{(\gamma_2,\kappa_2)},M^\T_{\rho(\gamma_1,\kappa_1)})\\
&=&\dim\Hom_{\Lambda^\T}(M^\T_{(\gamma_1,\kappa_1)},M^\T_{\rho(\gamma_2,\kappa_2)})
\end{array}\]
The words corresponding to $\gamma_1$ and $\rho(\gamma_2)$ are \[\m_{\gamma_1}=a_{4_+}a_{4-}^{-1}\qquad\text{and}\qquad
\m_{\rho(\gamma_2)}=a_{3_-}c^{-1}\varepsilon_1^{-1}cd^{-1}a_{4_-}^{-1},\] respectively. These two admissible words are precisely the two given in Example~\ref{exm:int} and we have $H^{\m_{\gamma_1},\m_{\rho(\gamma_2)}}$ contains only one element, which is $J=((0,2),(0,2))$ containing one punctured letter $a_{4_-}^{-1}$. Thus, we have
\[\begin{array}{rcl}
\dim\Ext^1_{\C(\T)}(X^\T_{(\gamma_1,\kappa_1)},X^\T_{(\gamma_2,\kappa_2)})
&=&\dim_\k\Hom_{\Lambda^\T}(M^\T_{(\gamma_1,\kappa_1)},M^\T_{\rho(\gamma_2,\kappa_2)})\\
&=&\dim_\k\Hom_{A_J}(\k_{\kappa_1(0)},\k_{1-\kappa_2(0)})\\
&=&1.
\end{array}\]

\section{Proof of Theorem~\ref{thm:Int}}\label{sec:proof}
\subsection{Adding marked points}

Recall that we fix an admissible triangulation $\T$ of $\surf$ and two tagged curves $(\gamma_1,\kappa_1)$ and $(\gamma_2,\kappa_2)$ in $\TC$.

For technical reasons, we add some new marked points on the boundary of $\surf$ as follows.
Let $E$ be the set of marked points $P$ in $M$ that are endpoints of $\gamma_1$ or $\gamma_2$. For each $P\in E$, we add two marked points on each side of $P$ on the boundary component containing $P$, denoted by $P_l', P_l''$ and $P_r', P_r''$ respectively, see Figure~\ref{fig:g1}.

\begin{figure}[htpb]\centering
\begin{tikzpicture}[scale=1.5,rotate=30]
\draw[ultra thick] (45-4*7:4) arc(45-4*7:45+4*7:4);
\draw[\RED,thick](45-14:4).. controls +(60:.2) and +(30:1) .. (45:4);
\draw[\RED,thick](45+14:4).. controls +(30:1) and +(60:.2) .. (45:4);
\draw[\RED,thick](45-21:4).. controls +(45:1) and +(45:1) .. (45:4);
\draw[\RED,thick](45+21:4).. controls +(45:1) and +(45:1) .. (45:4);
\foreach \j in {-3,...,3} {\draw[\RED] (45+7*\j:4)node{$\bullet$};}
\foreach \j in {-3,0,3} {\draw (45+7*\j:4)node{$\bullet$};}
\draw[\RED](45+7*2:3.8)node{$P_l''$};
\draw[\RED](45+7*1:3.8)node{$P_l'$};
\draw(45:3.8)node{$P$};
\draw[\RED](45-7:3.8)node{$P_r'$};
\draw[\RED](45-7*2:3.8)node{$P_r''$};
\draw[\RED](30:5)node{$\mathbf{R}$};
\end{tikzpicture}
\qquad
\begin{tikzpicture}[scale=.6]
\draw[ultra thick] (0,0) circle (1.3);
\foreach \j in {1,...,5}
    {\draw[\RED, thin] (90+72*\j:1.3) node{$\bullet$};}
\draw[\RED, thin]  (90:.85) node[black]{\small{$P$}}
    (90-72*1:.85)  node{\small{$P_r'$}}(90-72*2:.85) node{\small{$P_r''$}}
    (90+72*1:.85) node{\small{$P_l'$}}(90+72*2:.85) node{\small{$P_l''$}};
\draw[\RED,thick] plot [smooth,tension=1.5] coordinates {(90:1.3) (90-72:1.8) (90-72*2:1.3)};
\draw[\RED,thick] plot [smooth,tension=1.5] coordinates {(90:1.3) (90+72:1.8) (90+72*2:1.3)};
\draw[\RED,thick] plot [smooth,tension=1] coordinates {(90:1.3) (90+72:2.3) (90+72*2:2.3) (90-72*2:1.3)};
\draw[\RED,thick] plot [smooth,tension=1.5] coordinates {(90:1.3) (165:3) (-90:3) (15:3) (90:1.3)};
\draw[thin] (90:1.3) node{$\bullet$};
\end{tikzpicture}
\caption{Adding marked points}
\label{fig:g1}
\end{figure}
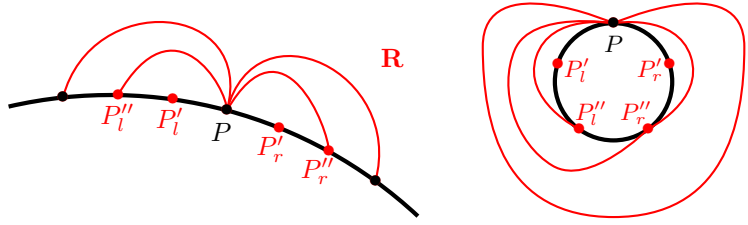

Let $\surf'$ be the new marked surface obtained from $\surf$ by adding these new marked points. For each $P\in E$, let $\R_P$ be the set of tagged curves as in Figure~\ref{fig:g1}, where the right picture is for the case that $P$ is a unique marked point on a boundary component. Take $\R$ to be the disjoint union of $\R_P$ for $P\in E$. Note that any two tagged curves in $\R$ do not cross each other and the cutting $\surf'/\R$ is canonically homeomorphic to $\surf$. Hence $\T\cup\R$ is an admissible triangulation of $\surf'$.

\begin{lemma}\label{lem:predim}
Under the notation above, we have
\[
\dim_\k\Ext^1_{\C(\T)}(X^{\T}_{(\gamma_1,\kappa_1)},X^{\T}_{(\gamma_2,\kappa_2)})
=\dim_\k\Ext^1_{\C(\T\cup\R)}(X^{\T\cup\R}_{(\gamma_1,\kappa_1)},X^{\T\cup\R}_{(\gamma_2,\kappa_2)}).
\]
\end{lemma}

\begin{proof}
Let $R=\bigoplus_{(\gamma,\kappa)\in\R}X^{\T\cup\R}_{(\gamma,\kappa)}$, a direct summand of $T_{\T\cup\R}$ in $\C(\T\cup\R)$. By Lemma~\ref{lem:cutting},
there is an equivalence
$\xi\colon{}^\bot R[1]/(R)\simeq\C(\T)$
sending $X^{\T\cup\R}_{(\gamma_i,\kappa_i)}$ to $X^\T_{(\gamma_i,\kappa_i)}$, $i=1, 2$ (noting that any tagged curve in $\R$ does not cross $(\gamma_i,\kappa_i)$). Hence,
\[\Ext^1_{\C(\T)}(X^\T_{(\gamma_1,\kappa_1)},X^\T_{(\gamma_2,\kappa_2)})\cong\Ext^1_{{}^\bot R[1]/ (R)}(X^{\T\cup\R}_{(\gamma_1,\kappa_1)},X^{\T\cup\R}_{(\gamma_2,\kappa_2)}).\]
By \cite[Lemma~4.8]{IY}, for any two objects $X_1,X_2\in{}^\bot R[1]$, there is an isomorphism \[\Ext^1_{{}^\bot R[1]/ (R)}(X_1,X_2)\cong\Ext^1_{\C(\T\cup\R)}(X_1,X_2).\]
Therefore, we get the equality, as required.
\end{proof}

Now we consider another admissible triangulation of $\surf'$. For each $P\in E$, let $\R'_P$ be the set of tagged curves as in Figure~\ref{fig:g2} and $\R'$ the disjoint union of $\R'_P$. By Lemma~\ref{lem:ex}, we can extend $\R'$ to an admissible triangulation $\T'$ of $\surf'$.
Due to Theorem~\ref{thm:ind}, there is an equivalence $\Theta\colon\C(\T\cup\R)\simeq\C(\T')$ such that $\Theta\left(X^{\T\cup\R}_{(\gamma_i,\kappa_i)}\right)\cong X^{\T'}_{(\gamma_i,\kappa_i)}$, for $i=1,2$.
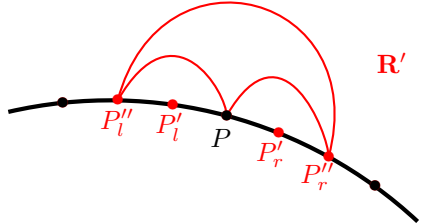
\begin{figure}[htpb]\centering
\begin{tikzpicture}[scale=1.5,rotate=30]
\draw[ultra thick] (45-4*7:4) arc(45-4*7:45+4*7:4);
\draw[\RED,thick](45-14:4).. controls +(60:.2) and +(30:1) .. (45:4);
\draw[\RED,thick](45+14:4).. controls +(30:1) and +(60:.2) .. (45:4);
\draw[\RED,thick](45-14:4).. controls +(45:1.5) and +(45:1.5) .. (45+14:4);
\foreach \j in {-3,...,3} {\draw[\RED] (45+7*\j:4)node{$\bullet$};}
\foreach \j in {-3,0,3} {\draw (45+7*\j:4)node{$\bullet$};}
\draw[thick](45:3.8)node{$P$};
\draw[\RED](30:5)node{$\mathbf{R}'$}
(45+7*2:3.8)node{$P_l''$}(45+7*1:3.8)node{$P_l'$}
(45-7:3.8)node{$P_r'$}(45-7*2:3.8)node{$P_r''$};
\end{tikzpicture}
\caption{A partial triangulation associated to new marked points}
\label{fig:g2}
\end{figure}

Therefore, this equivalence together with Lemma~\ref{lem:predim} implies that it is sufficient to prove $\Int((\gamma_1,\kappa_1),(\gamma_2,\kappa_2))=
\dim_\k\Ext^1_{\C(\T')}(X^{\T'}_{(\gamma_1,\kappa_1)},X^{\T'}_{(\gamma_2,\kappa_2)}).$

Without loss of generality, fix representatives of
$\gamma_1$ and $\gamma_2$ with minimal intersections with $\T'$ and with each other.
We further require that any intersection $\gamma_1\cap\gamma_2$ does not intersect $\T'$.

\subsection{Normal intersections in the interior}

We will use the notion of int-pairs from Section~\ref{subsec:results}. Recall that for any two admissible words $\m$ and $\r$, we use $H^{\m,\r}$ to denote the set of int-pairs from $\m$ to $\r$. For $v=0,1,2$, let $H_{v}^{\m,\r}$ be the subset of $H^{\m,\r}$ consisting of int-pairs that contain $v$ punctured letters. Then $H^{\m,\r}=H_{0}^{\m,\r}\cup H_{1}^{\m,\r}\cup H_{2}^{\m,\r}$.

\begin{lemma}\label{lem:above}
Let $(\gamma_1,\kappa_1)$ and $(\gamma_2,\kappa_2)$ be the two tagged curves in the theorem. There is a bijection between $\gamma_1\cap\gamma_2$ and the disjoint union $H_0^{\m_{\gamma_1},\m_{\rho(\gamma_2)}}\cup H_0^{\m_{\gamma_2},\m_{\rho(\gamma_1)}}$.
\end{lemma}

\begin{proof}
Consider the triangle $\Delta_I$ that contains
an intersection $I$ in $\gamma_1\cap\gamma_2$.
Since the triangulation $\T'$ contains $\R'$ which is constructed as in Figure~\ref{fig:g2}, the following situations do not occur.
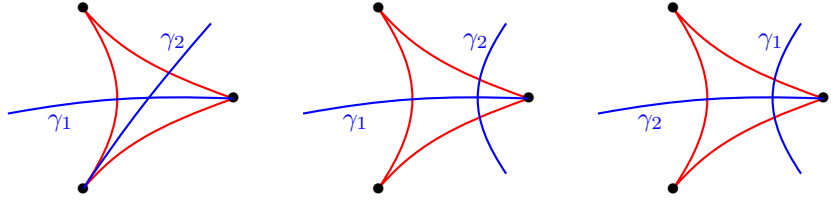
\[\begin{centering}
\begin{tikzpicture}[yscale=.4,rotate=180]
\draw[\RED,thick,bend right=15] (0,0)to (2,3) to (2,-3) to(0,0);
\draw (0,0)node{$\bullet$} (2,3)node{$\bullet$} (2,-3)node{$\bullet$};
\draw[\BLUE,thick,bend left=2](2,3)to(.3,-2.5);
\draw[\BLUE,thick,bend left=15](3,.5)to(0,0);
\draw[\BLUE,thick](2.3,1.4)node[above]{$\gamma_1$};
\draw[\BLUE,thick](0.8,-2.5)node[below]{$\gamma_2$};
\end{tikzpicture}\qquad
\begin{tikzpicture}[yscale=.4,rotate=180]
\draw[\RED,thick,bend right=15] (0,0)to (2,3) to (2,-3) to(0,0);
\draw (0,0)node{$\bullet$} (2,3)node{$\bullet$} (2,-3)node{$\bullet$};
\draw[\BLUE,thick,bend left=15](.3,2.5)to(.3,-2.5);
\draw[\BLUE,thick,bend left=15](3,.5)to(0,0);
\draw[\BLUE,thick](2.3,1.4)node[above]{$\gamma_1$};
\draw[\BLUE,thick](0.7,-2.5)node[below]{$\gamma_2$};
\end{tikzpicture}\qquad
\begin{tikzpicture}[yscale=.4,rotate=180]
\draw[\RED,thick,bend right=15] (0,0)to (2,3) to (2,-3) to(0,0);
\draw (0,0)node{$\bullet$} (2,3)node{$\bullet$} (2,-3)node{$\bullet$};
\draw[\BLUE,thick,bend left=15](.3,2.5)to(.3,-2.5);
\draw[\BLUE,thick,bend left=15](3,.5)to(0,0);
\draw[\BLUE,thick](2.3,1.4)node[above]{$\gamma_2$};
\draw[\BLUE,thick](0.7,-2.5)node[below]{$\gamma_1$};
\end{tikzpicture}
\end{centering}\]
Therefore, we can deduce that
$\gamma_1$ and $\gamma_2$ share at least one curve in $\T'$.
Hence the curve segments of $\gamma_1$ and $\gamma_2$ near $I$ has the form in Figure~\ref{fig:g3} with four possible right (resp. left) parts, where $\{a,b\}=\{1,2\}$.

\begin{figure}[htpb]\centering
	\begin{tikzpicture}[scale=0.56]
	
	\begin{scope}[xshift=-9cm,yshift=3cm,xscale=.7,yscale=.3]
	\draw[\RED,thick,bend right=15] (0,0)to (2,3) to (2,-3) to(0,0);
	\draw (0,0)node{$\bullet$} (2,3)node{$\bullet$} (2,-3)node{$\bullet$};
	\draw[\BLUE,thick,bend right=15](3,-.5)to(.5,-1.5);
	\draw[\BLUE,thick,bend left=15](3,.5)to(.5,1.5);
	\draw[\BLUE,thick](2.8,3)node[below]{$\gamma_a$};
	\draw[\BLUE,thick](2.8,-3)node[above]{$\gamma_b$};
	\end{scope}
	\begin{scope}[xshift=-6.7cm,yshift=1cm,xscale=.56,yscale=.91,rotate=180]
	\draw(-1.3,0)node[white]{x};
	\draw[thick,\BLUE](.7,.4).. controls +(0:1) and +(120:1) ..(3.5,-.5);
	\draw[thick,\BLUE](.7,-.4).. controls +(0:1) and +(-120:1) ..(3.5,.5);
	\draw[thick,\BLUE] (.7,.3)node[below]{$\gamma_b$} (.7,-.3)node[above]{$\gamma_a$};
	\draw[\RED,thick](4,0).. controls +(150:4) and +(-150:4) .. (4,0);
	\draw(4,0)node{$\bullet$}(2,0)node{$\bullet$};
	\end{scope}
	\begin{scope}[xshift=-6.7cm,yshift=-1cm,xscale=.56,yscale=.91,rotate=180]
	\draw[thick,\BLUE](.7,.4).. controls +(0:1) and +(120:1) ..(3.5,-.5);
	\draw[thick,\BLUE](.7,0)to(2,0);
	\draw[thick,\BLUE] (.7,.3)node[below]{$\gamma_b$} (.7,.1)node[above]{$\gamma_a$};
	\draw[\RED,thick](4,0).. controls +(150:4) and +(-150:4) .. (4,0);
	\draw(4,0)node{$\bullet$}(2,0)node{$\bullet$};
	\end{scope}
	\begin{scope}[xshift=-6.7cm,yshift=-3cm,xscale=.56,yscale=.91,rotate=180]
	\draw[thick,\BLUE](.7,-.4).. controls +(0:1) and +(-120:1) ..(3.5,.5);
	\draw[thick,\BLUE](.7,0)to(2,0);
	\draw[thick,\BLUE] (.7,-.1)node[below]{$\gamma_b$} (.7,-.3)node[above]{$\gamma_a$};
	\draw[\RED,thick](4,0).. controls +(150:4) and +(-150:4) .. (4,0);
	\draw(4,0)node{$\bullet$}(2,0)node{$\bullet$};
	\end{scope}
	
	\draw (-8,4.5)node{left part};
	
	\begin{scope}[xshift=9cm,yshift=3cm,xscale=.7,yscale=.3,rotate=180]
	\draw[\RED,thick,bend right=15] (0,0)to (2,3) to (2,-3) to(0,0);
	\draw (0,0)node{$\bullet$} (2,3)node{$\bullet$} (2,-3)node{$\bullet$};
	\draw[\BLUE,thick,bend right=15](3,-.5)to(.5,-1.5);
	\draw[\BLUE,thick,bend left=15](3,.5)to(.5,1.5);
	\draw[\BLUE,thick](2.8,3)node[above]{$\gamma_a$};
	\draw[\BLUE,thick](2.8,-3)node[below]{$\gamma_b$};
	\end{scope}
	\begin{scope}[xshift=6.7cm,yshift=1cm,xscale=.56,yscale=.91]
	\draw(-1.3,0)node[white]{x};
	\draw[thick,\BLUE](.7,.4).. controls +(0:1) and +(120:1) ..(3.5,-.5);
	\draw[thick,\BLUE](.7,-.4).. controls +(0:1) and +(-120:1) ..(3.5,.5);
	\draw[thick,\BLUE] (.7,.3)node[above]{$\gamma_b$} (.7,-.3)node[below]{$\gamma_a$};
	\draw[\RED,thick](4,0).. controls +(150:4) and +(-150:4) .. (4,0);
	\draw(4,0)node{$\bullet$}(2,0)node{$\bullet$};
	\end{scope}
	\begin{scope}[xshift=6.7cm,yshift=-1cm,xscale=.56,yscale=.91]
	\draw[thick,\BLUE](.7,.4).. controls +(0:1) and +(120:1) ..(3.5,-.5);
	\draw[thick,\BLUE](.7,0)to(2,0);
	\draw[thick,\BLUE] (.7,.3)node[above]{$\gamma_b$} (.7,.1)node[below]{$\gamma_a$};
	\draw[\RED,thick](4,0).. controls +(150:4) and +(-150:4) .. (4,0);
	\draw(4,0)node{$\bullet$}(2,0)node{$\bullet$};
	\end{scope}
	\begin{scope}[xshift=6.7cm,yshift=-3cm,xscale=.56,yscale=.91]
	\draw[thick,\BLUE](.7,-.4).. controls +(0:1) and +(-120:1) ..(3.5,.5);
	\draw[thick,\BLUE](.7,0)to(2,0);
	\draw[thick,\BLUE] (.7,-.1)node[above]{$\gamma_b$} (.7,-.3)node[below]{$\gamma_a$};
	\draw[\RED,thick](4,0).. controls +(150:4) and +(-150:4) .. (4,0);
	\draw(4,0)node{$\bullet$}(2,0)node{$\bullet$};
	\end{scope}
	
	\draw  (8,4.5)node{right part}
	(0,4.2)node[above]{middle part}
	(2.5,1)node[above,\BLUE]{$\gamma_{b}$}(2.5,-1)node[below,\BLUE]{$\gamma_{a}$};
	\draw[dashed] (5,2)to(-5,2) (5,-2)to(-5,-2)
	(2.5,0)to(3.5,0)(-2.5,0)to(-3.5,0);
	\draw[\RED, thick] (4,2)to(4,-2) (-4,-2)to(-4,2);
	\draw[\RED,dashed](0,0)circle(1)(0,1)node[above]{$\Delta_I$};
	\draw[\BLUE](0,0)node[below]{$I$};
	\draw[\BLUE, thick] (-5,1).. controls +(0:3.9) and +(180:3.9) .. (5,-1);
	\draw[\BLUE, thick] (-5,-1).. controls +(0:3.9) and +(180:3.9) .. (5,1);
	\end{tikzpicture}
	\caption{An interior intersection}
	\label{fig:g3}
\end{figure}
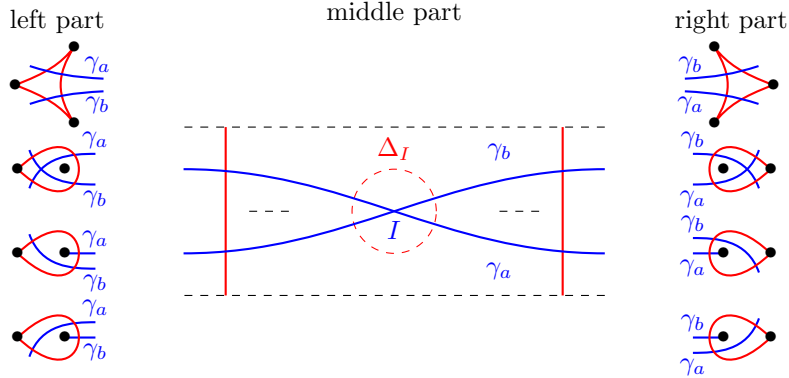

We will prove that such an intersection induces an int-pair in $H_0^{\m_{\gamma_b},\m_{\rho(\gamma_a)}}$.
Let
\[
    \m=\m_{\gamma_b}=\omega_m\cdots\omega_1\quad\text{and}\quad\r=\m_{\gamma_a}=\nu_r\cdots\nu_1.
\]
Set $\r'=\m_{\rho(\gamma_a)}$. We call an arc segment is in $\R'$ if it is in a triangle formed by curves in $\R'$ and boundary segments and we call a letter is from $\R'$ if its corresponding arc segment is in $\R'$. Note that $\gamma_a$ and $\rho(\gamma_a)$ intersect the triangulation $\T'$ in the same way except for the parts near endpoints of $\gamma_a$ in $\M$, where $\rho(\gamma_a)$ intersects one more curve in $\R'$ than $\gamma_a$ as in Figure~\ref{fig:g4}.
Then $\r'=x \nu_{r-1}\cdots\nu_{2} y$,
where
\[
    x=\begin{cases}\nu_r,&\gamma_a(1)\in\P\\  \nu'_{r+1}\nu'_r,&\gamma_a(1)\in\M\end{cases}\quad
        \text{and}\quad
    y=\begin{cases}\nu_r,&\gamma_a(0)\in\P\\  \nu'_{1}\nu'_0,&\gamma_a(0)\in\M\end{cases}
\]
where $\nu'_0$, $\nu'_{1}$, $\nu'_r$ and $\nu'_{r+1}$ are letters corresponding to arc segments of $\rho(\gamma_a)$ in $\R'$. By Lemma~\ref{lem:compprod}~(1) a letter in $(\r')^{\pm1}$ from $\R'$ is not smaller than any letter in $\m^{\pm1}$ or $\r^{\pm1}$ from $\R'$ (cf. also Figure~\ref{fig:g4}). In particular, $(\r')^{\pm1}$ and $\m^{\pm1}$/$\r^{\pm1}$ do not share letters from $\R'$.

\begin{figure}[htpb]\centering
	\begin{tikzpicture}[scale=1.5,rotate=30]
	\draw[\BLUE,thick](45:4)to[bend right=5](50:5.5);
	\draw[\BLUE,thick](45-7:4)to[bend left=10](50-1:5.5);
	\draw[ultra thick] (45-4*7:4) arc(45-4*7:45+4*7:4);
	\draw[\RED,thick](45-14:4).. controls +(60:.2) and +(30:1) .. (45:4);
	\draw[\RED,thick](45+14:4).. controls +(30:1) and +(60:.2) .. (45:4);
	\draw[\RED,thick](45-14:4).. controls +(45:1.5) and +(45:1.5) .. (45+14:4);
	\foreach \j in {-3,...,3} {\draw[\RED] (45+7*\j:4)node{$\bullet$};}
	\foreach \j in {-3,0,3} {\draw (45+7*\j:4)node{$\bullet$};}
	\draw(45:3.8)node{$P$};
	\draw(45-3:4.7)node[\BLUE]{$_{\rho(\gamma_a)}$};
	\draw(45+3:4.6)node[\BLUE]{$_{\gamma_a}$};
	\end{tikzpicture}
	\caption{The rotation of $\gamma_a$}
	\label{fig:g4}
\end{figure}
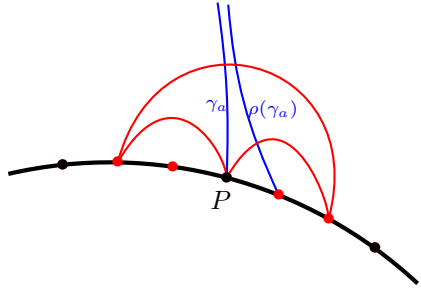

Without loss of generality, we assume that both of the orientations of $\gamma_a$ and $\gamma_b$ are from right to left in Figure~\ref{fig:g3}. Then the curve segments in the middle part correspond to the subwords $\m_{(i,j)}$ and $\r_{(h,l)}$ for some $0<i<j<m$, $0<h<l<r$ and the arc segments in the left (resp. right) part correspond to $\omega_j$ and $\nu_l$ (resp. $\omega_i$ and $\nu_h$). By Lemma~\ref{lem:compprod}~(2), we have $\omega_i^{-1}<\nu_h^{-1}$ and $\omega_j<\nu_l$.
Thus $J_I:=\left((i,j),(h,l)\right)$ is an int-pair in $H_0^{\m,\r}$. Moreover, it is clear that any arc segment of $\gamma_a$ which connects a marked point in $M$ does not appear in Figure~\ref{fig:g3}. Therefore $\nu_l\r_{(h,l)}\nu_j$ do not have letters from $\R'$ and thus it is also a subword of $\r'$. Hence $J_I$ is also an int-pair in $H_0^{\m,\r'}$ and we obtain a map
\begin{eqnarray*}
  J\colon\gamma_1\cap\gamma_2&\to&H_0^{\m_{\gamma_1},\m_{\rho(\gamma_2)}}\cup H_0^{\m_{\gamma_2},\m_{\rho(\gamma_1)}}\\
  I&\mapsto&J_I
\end{eqnarray*}
Clearly, different intersections correspond to different int-pairs. Hence the map $J$ is injective. So what is left to show is that the map $J$ is surjective.

Let $J_0=\left((i,j),(h,l)\right)$ be an int-pair in $H_0^{\m,\r'}$. Without loss of generality, we assume that $\m_{(i,j)}=\r'_{(h,l)}$ with $\omega_i^{-1}<{\nu'}_{h}^{-1}$ and $\omega_j<\nu'_l$, where $\nu'_a=\nu_a$ if $1<a<r$.
Since $\m$ and $\r'$ do not share letters from $\R'$ and $\m_{(i,j)}=\r'_{(h,l)}$ contains no punctured letters, we have $\r'_{(h,l)}=\r_{(h,l)}$ and the letters $\omega_i$, $\nu'_{h}$, $\omega_j$ and $\nu'_l$ exist.
Since $\omega_i^{-1}$ and ${\nu'}_{h}^{-1}$ are comparable, by Lemma~\ref{lem:compprod}~(1) their corresponding arc segments are in the same triangle. Hence if ${\nu'_{h}}^{-1}$ is from $\R'$, then so is $\omega_i^{-1}$. This is a contradiction because ${\nu'_{h}}^{-1}$ is a letter in $(\r')^{-1}$ and $\omega_i^{-1}$ is in $\m^{-1}$.
Hence neither $\nu'_h$ nor $\omega_i$ is from $\R'$. Similarly, neither $\nu'_l$ nor $\omega_j$ is from $\R'$. In particular, their corresponding arc segments do not connect to a marked point in $\M$
and we have $\nu'_h=\nu_h$ and $\nu'_l=\nu_l$.
Then the curve segments corresponding to $\omega_j\m_{(i,j)}\omega_i$ and $\nu_l\r_{(h,l)}\nu_h=\nu'_l\r'_{(h,l)}\nu'_h$ are of the form in Figure~\ref{fig:g3}
(note that the left/right parts are the cases in Figure~\ref{fig:orders}~(1), where both $\gamma_i$ do not connect to marked points in $\M$).
By construction, we see that $J_0$ is in the image of the map $J$ above, as required.
\end{proof}

For an int-pair $J$ without punctured letters, since $A_J\cong\k$, the unique 1-dimensional $A_J$-module is $\k$. Then $N_{(\gamma_i,\kappa_i)}\cong\k$ as $A_J$-modules (cf. Notation~\ref{notation}). So we have the following consequence:
\begin{equation}\label{eq:1}
\Int(\gamma_1,\gamma_2)=\sum\limits_{\{a,b\}=\{1,2\}}\sum\limits_{J\in H_0^{\m_{\gamma_a},\m_{\rho(\gamma_b)}}}\dim_\k\Hom_{A_J}(N_{(\gamma_a,\kappa_a)},N_{\rho(\gamma_b,\kappa_b)}).
\end{equation}

\subsection{Tagged intersections at the ends}
Recall that $\p=\p(\gamma_1,\gamma_2)=\{(t_1,t_2)\in\{0,1\}^2\mid\gamma_1(t_1)=\gamma_2(t_2)\in\P\}$ is the set of intersections between $\gamma_1$ and $\gamma_2$ at $\P$. Let
\[\p_1=\{(t_1,t_2)\in\p\mid\gamma_1|_{t_1\rightarrow(1-t_1)}\nsim\gamma_2|_{t_2\rightarrow(1-t_2)}\}\]
and
\[\t_1=\{(t_1,t_2)\in\p_1\mid\kappa_1(t_1)\neq\kappa_2(t_2)\}.\]
There is an analogue result of Lemma~\ref{lem:above} for $\p_1$.
\begin{lemma}\label{lem:above2}
There is a bijection between $\p_1$ and the disjoint union $H_1^{\m_{\gamma_1},\m_{\rho(\gamma_2)}}\cup H_1^{\m_{\gamma_2},\m_{\rho(\gamma_1)}}$.
\end{lemma}
\begin{proof}
Each intersection in $\p_1$ has the form in Figure~\ref{fig:pint} with four possible right parts and $\{a,b\}=\{1,2\}$.
Then the required bijection follows from a similar proof of Lemma~\ref{lem:above}.
\end{proof}

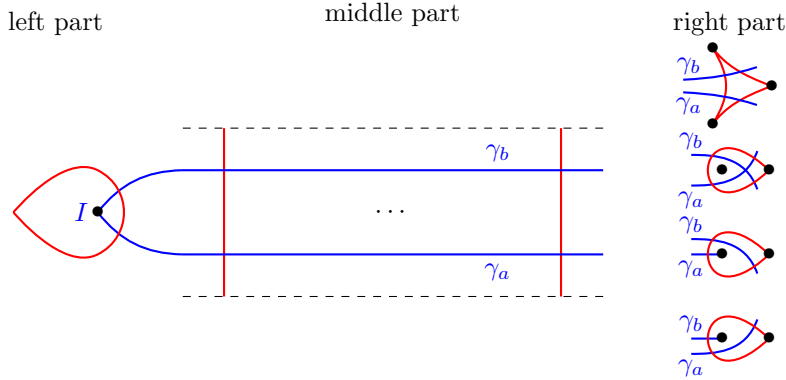
\begin{figure}[htpb]\centering
\begin{tikzpicture}[scale=0.56]
\draw (-8,4.5)node{left part};

\begin{scope}[xshift=9cm,yshift=3cm,xscale=.7,yscale=.3,rotate=180]
\draw[\RED,thick,bend right=15] (0,0)to (2,3) to (2,-3) to(0,0);
\draw (0,0)node{$\bullet$} (2,3)node{$\bullet$} (2,-3)node{$\bullet$};
\draw[\BLUE,thick,bend right=15](3,-.5)to(.5,-1.5);
\draw[\BLUE,thick,bend left=15](3,.5)to(.5,1.5);
\draw[\BLUE,thick](2.8,3)node[above]{$\gamma_a$};
\draw[\BLUE,thick](2.8,-3)node[below]{$\gamma_b$};
\end{scope}
\begin{scope}[xshift=6.7cm,yshift=1cm,xscale=.56,yscale=.91]
\draw(-1.3,0)node[white]{x};
\draw[thick,\BLUE](.7,.4).. controls +(0:1) and +(120:1) ..(3.5,-.5);
\draw[thick,\BLUE](.7,-.4).. controls +(0:1) and +(-120:1) ..(3.5,.5);
\draw[thick,\BLUE] (.7,.3)node[above]{$\gamma_b$} (.7,-.3)node[below]{$\gamma_a$};
\draw[\RED,thick](4,0).. controls +(150:4) and +(-150:4) .. (4,0);
\draw(4,0)node{$\bullet$}(2,0)node{$\bullet$};
\end{scope}
\begin{scope}[xshift=6.7cm,yshift=-1cm,xscale=.56,yscale=.91]
\draw[thick,\BLUE](.7,.4).. controls +(0:1) and +(120:1) ..(3.5,-.5);
\draw[thick,\BLUE](.7,0)to(2,0);
\draw[thick,\BLUE] (.7,.3)node[above]{$\gamma_b$} (.7,.1)node[below]{$\gamma_a$};
\draw[\RED,thick](4,0).. controls +(150:4) and +(-150:4) .. (4,0);
\draw(4,0)node{$\bullet$}(2,0)node{$\bullet$};
\end{scope}
\begin{scope}[xshift=6.7cm,yshift=-3cm,xscale=.56,yscale=.91]
\draw[thick,\BLUE](.7,-.4).. controls +(0:1) and +(-120:1) ..(3.5,.5);
\draw[thick,\BLUE](.7,0)to(2,0);
\draw[thick,\BLUE] (.7,-.1)node[above]{$\gamma_b$} (.7,-.3)node[below]{$\gamma_a$};
\draw[\RED,thick](4,0).. controls +(150:4) and +(-150:4) .. (4,0);
\draw(4,0)node{$\bullet$}(2,0)node{$\bullet$};
\end{scope}

\draw  (8,4.5)node{right part}
    (0,4.2)node[above]{middle part};

\draw[\BLUE,thick,bend right=25,>=stealth](-7,0)to(-5,-1);
\draw[\BLUE,thick,bend right=-25,>=stealth](-7,0)to(-5,1);
\draw(-6-1,0)node{$\bullet$} (-8,0)node{} (0,0)node{$\cdots$};
\draw[\BLUE](-7,0)node[left]{$I$};
\draw[\RED,thick] plot[smooth,tension=1.5] coordinates
{(-8-1,0) (-6-1,1) (-6-1,-1)  (-8-1,0)};
\draw(2.5,1)node[above,\BLUE]{$\gamma_{b}$}(2.5,-1)node[below,\BLUE]{$\gamma_{a}$};
\draw[dashed] (5,2)to(-5,2) (5,-2)to(-5,-2);
\draw[\RED, thick] (4,2)to(4,-2) (-4,-2)to(-4,2);
\draw[\BLUE, thick] (-5,1).. controls +(0:3.9) and +(180:3.9) .. (5,1);
\draw[\BLUE, thick] (-5,-1).. controls +(0:3.9) and +(180:3.9) .. (5,-1);
\end{tikzpicture}
\caption{A punctured intersection}
\label{fig:pint}
\end{figure}

For an intersection $I=(t_1,t_2)\in\p_1$, let $J_I$ be the corresponding int-pair in $H_1^{\m_{\gamma_1},\m_{\rho(\gamma_2)}}\cup H_1^{\m_{\gamma_2},\m_{\rho(\gamma_1)}}$.
Then we have the associated algebra $A_{J_I}=\k[x]/(x^2-x)$ and $N_{(\gamma_a,\kappa_a)}\cong\k_{\kappa_a(t_a)}$ and $N_{\rho(\gamma_b,\kappa_b)}\cong\k_{1-\kappa_b(t_{b})}$ as $A_{J_I}$-modules
(cf. Notation~\ref{notation}), for $\{a,b\}=\{1,2\}$.
Using the formula \eqref{eq:xy}, we have
\[
    \dim_\k\Hom_{{A_{J_I}}}(N_{(\gamma_a,\kappa_a)},N_{\rho(\gamma_b,\kappa_b)})=
    \begin{cases}1,&\text{for $(t_1,t_2)\in\t_1$},\\
        $0$,& \text{for $(t_1,t_2)\notin\t_1$.}
    \end{cases}
\]
Hence, we obtain a consequence of Lemma~\ref{lem:above2}:
\begin{equation}\label{eq:3}
|\t_1|=\sum\limits_{\{a,b\}=\{1,2\}}\sum\limits_{J\in H_1^{\m_{\gamma_a},\m_{\rho(\gamma_b)}}}\dim_\k\Hom_{A_J}(N_{(\gamma_a,\kappa_a)},N_{\rho(\gamma_b,\kappa_b)}).
\end{equation}

Let
\[\p_2=\{(t_1,t_2)\in\p\mid\gamma_1|_{t_1\rightarrow(1-t_1)}\sim\gamma_2|_{t_2\rightarrow(1-t_2)},
    \gamma_1(1-t_1)=\gamma_2(1-t_2)\in\P\}.\]
Observe that for each $(t_1,t_2)\in\p_2$, $(1-t_1,1-t_2)$ is also in $\p_2$.
We call them twin intersections (at punctures).
Clearly, there is at most one pair of twin intersections in $\p_2$.
Let
\[\t_2=\{(t_1,t_2)\in\p_2\mid\kappa_1(t_1)\neq\kappa_2(t_2),\ \kappa_1(1-t_1)\neq\kappa_2(1-t_2)\}.\]
Suppose that there is a (unique) pair of twin intersections $(t_1,t_2)$ and $(1-t_1,1-t_2)$ in $\p_2$.
Then both the endpoints of $\gamma_i$ are in $\P$ and thus $\gamma_i=\rho(\gamma_i)$.
Reversing one of $\gamma_i$ if necessary, assume that $\gamma_1\sim\gamma_2$.
So $\m_{\gamma_1}=\m_{\gamma_2}=\m_{\rho(\gamma_1)}=\m_{\rho(\gamma_2)}$. Then
this pair of twin intersections induces the int-pairs
$J_{1,2}=(\m_{\gamma_1},\m_{\rho(\gamma_2)})$ in $H_2^{\m_{\gamma_1},\m_{\rho(\gamma_2)}}$
and
$J_{2,1}=(\m_{\gamma_2},\m_{\rho(\gamma_1)})$ in $H_2^{\m_{\gamma_2},\m_{\rho(\gamma_1)}}$,
which are the only ones with two punctured letters. In this case, for any $\{a,b\}=\{1,2\}$ we have $A_{J_{a,b}}\cong\k\<x,y\>/(x^2-x,y^2-y)$ and $N_{(\gamma_a,\kappa_a)}\cong\k_{\kappa_a(t_a),\kappa_a(1-t_a)}$ and $N_{\rho(\gamma_b,\kappa_b)}\cong\k_{1-\kappa_b(t_b),1-\kappa_b(1-t_b)}$ as $A_{J_{a,b}}$-modules (cf. Notation~\ref{notation}).
Using the formula \eqref{eq:xy2}, we have
\[
    \dim_\k\Hom_{A_{J_{a,b}}}
    \left(N_{(\gamma_a,\kappa_a)},N_{\rho(\gamma_b,\kappa_b)}\right)=
    \begin{cases}1,&\text{for $(t_1,t_2)\in\t_2$},\\
        $0$,& \text{for $(t_1,t_2)\notin\t_2$.}
    \end{cases}
\]
and hence
\begin{equation}\label{eq:4}
|\t_2|=\sum\limits_{\{a,b\}=\{1,2\}}\sum\limits_{J\in H_2^{\m_{\gamma_a},\m_{\rho(\gamma_b)}}}\dim_\k\Hom_{A_J}(N_{(\gamma_a,\kappa_a)},N_{\rho(\gamma_b,\kappa_b)}).
\end{equation}

\subsection{Summary}

\renewcommand{\arraystretch}{2}

By definition, we have
\begin{equation}\label{eq:2}
|\t\left((\gamma_1,\kappa_1),(\gamma_2,\kappa_2)\right)|=|\t_1|+|\t_2|.
\end{equation}
Combining \eqref{eq:1}, \eqref{eq:3}, \eqref{eq:4} and \eqref{eq:2}, we have
\[\begin{array}{rcl}
\Int\left((\gamma_1,\kappa_1),(\gamma_2,\kappa_2)\right)&=&\sum\limits_{\{a,b\}=\{1,2\}}\sum\limits_{J\in H^{\m_{\gamma_a},\m_{\rho(\gamma_b)}}}\dim_\k\Hom_{A_J}(N_{(\gamma_a,\kappa_a)},N_{\rho(\gamma_b,\kappa_b)})\\
&=&\sum\limits_{\{a,b\}=\{1,2\}}\dim_\k\Hom_{\Lambda^{\T'}}(M^{\T'}_{(\gamma_a,\kappa_a)},M^{\T'}_{\rho(\gamma_b,\kappa_b)})\\
&=&\sum\limits_{\{a,b\}=\{1,2\}}\dim_\k\Hom_{\Lambda^{\T'}}(M^{\T'}_{(\gamma_a,\kappa_a)},\tau M^{\T'}_{(\gamma_b,\kappa_b)})\\
&=&\dim_\k\Ext^1_{\C(\T')}(X^{\T'}_{(\gamma_1,\kappa_1)}, X^{\T'}_{(\gamma_2,\kappa_2)})
\end{array}\]
Here, the second equality follows from Theorem~\ref{thm:G2},
the third one follows from the fact that $M'_{\rho(\gamma_b,\kappa_b)}=\tau M'_{\gamma_b,\kappa_b}$ (Theorem~\ref{thm:T-rotation})
and the last one follows from \cite[Lemma 3.3]{Pa}.
This finishes the proof.

\appendix
\section{Admissible triangulations}\label{app:ex}
In this section, we give some results about admissible triangulations which will be used in the paper.

\begin{lemma}\label{lem:ex}
There is an admissible triangulation,  i.e. every puncture is in a self-folded triangle,
of any marked surface with non-empty boundary.
\end{lemma}
\begin{proof}
If the surface $\surf$ does not admit any puncture, then any triangulation is admissible. Now let $P$ be a puncture in $\surf$. Use induction on the rank $n$,
staring from the trivial case when $n=1$ and $\surf$ is a once-punctured monogon.
Now suppose $n\geq2$.
Consider the curve $\alpha$ with $\Int(\alpha,\alpha)=0$ and $\alpha(0)=\alpha(1)=M$, which encloses a disk with the puncture $P$.
Cutting along $\alpha$
we obtain a surface $\surf/\alpha$ (cf. the second column of Figure~\ref{fig:5c}) whose rank is $n-1$.
By inductive assumption we deduce that $\surf/\alpha$, and hence $\surf$
admits an admissible triangulation.
\end{proof}

\begin{lemma}\label{lem:loz-conn}
Any two admissible triangulations are connected by a sequence of $\lozenge$-flips.
\end{lemma}
\begin{proof}
Let $\T_i$, $i=1,2$, be two admissible triangulations.
Use induction on the rank $n$ of the marked surface $\surf$,
starting with the trivial case when $n=1$.

Consider a puncture $P$, which is connected to exactly one marked point $M_i$ in $\T_i$.
If $PM_1\sim PM_2$, we can delete the self-folded triangles containing $P$ from $\T_i$
and reduce to the case with a smaller rank.

Now suppose that $PM_1\nsim PM_2$.
Frozen the self-folded triangle in $\T_1$ containing $PM_1$.
By inductive assumption for the remaining surface,
we can flip $\T_1$ to a triangulation $\T_1'$,
with local picture as in the left picture of Figure~\ref{fig:lozenge}
with $A=M_2$, $B=M_1$ and the curve $PA\sim PM_2$.
Then by one $\lozenge$-flip we can locally flip $\T_1'$ to another triangulation $\T_1''$
such that it contains the curve $PM_2$ in $\T_2$,
which becomes the $PM_1\sim PM_2$ case above.

\begin{figure}[ht]\centering
\begin{tikzpicture}[scale=.5]
\draw[thick] (0,0.2) circle (2.2)(0,1) to (0,-2);
\draw (0,1) node {$\bullet$}node[right]{$P$} (0,-2) node {$\bullet$} node[below]{$B$};
\draw (4,1) node {$\Longrightarrow$};
\draw[thick] (0,1) circle (3);
\draw (0,4)node{$\bullet$}node[above]{$A$};
\end{tikzpicture}\quad
\begin{tikzpicture}[scale=.5,rotate=180]
\draw[thick] (0,0.2) circle (2.2) (0,1)to (0,-2);
\draw (0,1) node {$\bullet$}node[right]{$P$} (0,-2) node {$\bullet$}node[above]{$A$};
\draw[thick] (0,1) circle (3);
\draw (0,4)node{$\bullet$}node[ below]{$B$};
\end{tikzpicture}
\caption{A $\lozenge$-flip}
\label{fig:lozenge}
\end{figure}
\end{proof}

\section{Explicit version of DWZ's mutations of decorated representations for biquivers with potential}\label{app:DWZ}
\begin{table}[htbp]
\caption{The first type of $\lozenge$-flips}\label{table1}
\centering
\begin{tabular}{c||c|c}
   &
    \begin{tikzpicture}[scale=.25]

    \draw[very thick]plot [smooth,tension=1] coordinates {(-120:6) (-95:4.5) (-60:6)};
    \draw[very thick]plot [smooth,tension=1] coordinates {(110:5) (85:4) (70:5)};
    \draw[very thick]plot [smooth,tension=1] coordinates {(30:5) (10:4) (-10:5)};
    \draw[very thick]plot [smooth,tension=1] coordinates {(150:5) (170:4) (185:5)};
    \draw[\RED,thick](-95:4.5)to[bend left=23](10:4)to[bend left](85:4)to(-95:4.5)
    to[bend right](170:4)to[bend right](85:4);
    \draw[<-,>=stealth,bend left=10,\BLUE](.9,2.65)to(0.25,2.4);
    \draw(0,1.1)node[above right]{$_{_{{\beta_2}}}$};
    \draw[<-,>=stealth,bend left=10,\BLUE](0.2,2.4)to(-0.5,2.5);
    \draw(-.27,1.05)node[above]{\tiny{$_{_{{\alpha_1}}}$}};
    \draw[<-,>=stealth,bend left=10,\BLUE](-1,-2)to(-.25,-2+.2);
    \draw(-.45,-.5)node[below]{$_{_{{\beta_1}^{}}}$};
    \draw[<-,>=stealth,bend left=10,\BLUE](-.1,-1.5-.3)to(0.75,-1.7-.3);
    \draw(.63,-.58)node[below]{$_{_{{\alpha_2}^{}}}$};
    \draw[->,>=stealth,bend left=-10,\BLUE](-2.55,-.1)to(-2.55,1);
    \draw(-1.75,.84)node[below]{$_{_{{\gamma_1}^{}}}$};
    \draw[<-,>=stealth,bend left=10,\BLUE](2.6,-0)to(2.6,1.1);
    \draw(1.99,1.2)node[below]{$_{_{{\gamma_2}^{}}}$};
    \draw[thick](-95:4.5)node{$\bullet$}(85:4)node{$\bullet$}
        (10:4)node{$\bullet$}(170:4)node{$\bullet$};
    \draw(-.27,-0.3)node[\RED,above]{$i$};
    \end{tikzpicture}
&
    \begin{tikzpicture}[scale=.25]

    \draw[very thick]plot [smooth,tension=1] coordinates {(-120:6) (-95:4.5) (-60:6)};
    \draw[very thick]plot [smooth,tension=1] coordinates {(110:5) (85:4) (70:5)};
    \draw[very thick]plot [smooth,tension=1] coordinates {(30:5) (10:4) (-10:5)};
    \draw[very thick]plot [smooth,tension=1] coordinates {(150:5) (170:4) (185:5)};
    \draw[\RED,thick](-95:4.5)to[bend left=23](10:4)to[bend left](85:4)(-95:4.5)
    to[bend right](170:4)to[bend right](85:4)   (170:4)to[bend right=10](10:4);

    \draw[->,>=stealth,bend left=-10,\BLUE](-2,-.56)to(-2,.35);
    \draw[->,>=stealth,bend left=-10,\BLUE](-2,.45)to(-2,1.35);
    \draw(-1.1,1.8)node[below]{$_{_{\alpha_1^{*}}}$}
        (-1.25,.44)node[below]{$_{_{\beta_1^{*}}}$};
    \draw[<-,>=stealth,bend left=10,\BLUE](2,-0.5)to(2,.35);
    \draw[<-,>=stealth,bend left=10,\BLUE](2,.45)to(2,1.45);
    \draw(1.1,1.8)node[below]{$_{_{\beta_2^{*}}}$}
        (1.25,.44)node[below]{$_{_{\alpha_2^{*}}}$};

    \draw[thick](-95:4.5)node{$\bullet$}(85:4)node{$\bullet$}
        (10:4)node{$\bullet$}(170:4)node{$\bullet$};
    \draw(2,-2)node[below]{\tiny{$_{_{_{{[\beta_1\alpha_2]}}}}$}};
    \draw[<-,>=stealth,bend left=10,\BLUE](-.8,-2.3)to(0.55,-2.3);
    \draw(-3.8,2.1)node[above right]{$_{_{{[\beta_2\alpha_1]}}}$};
    \draw[<-,>=stealth,bend left=10,\BLUE](.9,2.65)to(-0.5,2.5);
    \end{tikzpicture}
\\\hline
&$\xymatrix@R=.8pc@C=2.5pc{
\cdot\ar[dr]|{\alpha_1}&&\cdot\ar[dd]|{\gamma_2}\\
    & i\ar@{<-}[dr]|{\alpha_2}\ar[ur]|{\beta_2} \\
    \ar@{<-}[ur]|{\beta_1}\ar[uu]|{\gamma_1}\cdot&&\cdot}$
&$\xymatrix@R=.8pc@C=2.5pc{
\cdot\ar[rr]|{^{[\beta_2\alpha_1]}}\ar@{<-}[dr]|{\alpha_1^*}
    &&\cdot\\&i\ar[dr]|{\alpha_2^*}\ar@{<-}[ur]|{\beta_2^*} \\  \cdot\ar[ur]|{\beta_1^*}&&\cdot\ar[ll]|{^{[\beta_1\alpha_2]}}}$
\\\hline
    \begin{tikzpicture}[xscale=.1,yscale=.09]

    \draw[ ]plot [smooth,tension=1] coordinates {(-120:6) (-95:4.5) (-60:6)};
    \draw[ ]plot [smooth,tension=1] coordinates {(110:5) (85:4) (70:5)};
    \draw[ ]plot [smooth,tension=1] coordinates {(30:5) (10:4) (-10:5)};
    \draw[ ]plot [smooth,tension=1] coordinates {(150:5) (170:4) (185:5)};
    \draw[\RED](-95:4.5)to[bend left=23](10:4)to[bend left](85:4)(-95:4.5)
    to[bend right](170:4)to[bend right](85:4);
    \draw[](-95:4.5)node{$.$}(85:4)node{$.$}
        (10:4)node{$.$}(170:4)node{$.$};
    \draw[bend left=10,\BLUE](1,2.5)to(-0.5,2.5);
    \end{tikzpicture}
&
\begin{tikzpicture}[xscale=.1,yscale=.08]
    \draw(-7,4)node(a){$\k$}(7,4)node(b){$\k$}(0,0)node(c){$\k$}
        (-7,-4)node(d){0}(7,-4)node(e){0};
        \draw(-3.5,3.7)node{\tiny{1}}(3.1,3.7)node{\tiny{1}};
        \draw[->,>=stealth](d)to(a);\draw[->,>=stealth](c)to(b);\draw[->,>=stealth](a)to(c);
        \draw[->,>=stealth](e)to(c);\draw[->,>=stealth](c)to(d);\draw[->,>=stealth](b)to(e);
    \end{tikzpicture}
&
\begin{tikzpicture}[xscale=.1,yscale=.08]
    \draw(-7,4)node(a){$\k$}(7,4)node(b){$\k$}(0,0)node(c){0}
        (-7,-4)node(d){0}(7,-4)node(e){0};
        \draw(0,5)node{\tiny{1}};
        \draw[->,>=stealth](a)to(b);\draw[->,>=stealth](b)to(c);\draw[->,>=stealth](c)to(a);
        \draw[->,>=stealth](d)to(c);\draw[->,>=stealth](c)to(e);\draw[->,>=stealth](e)to(d);
    \end{tikzpicture}
\\\hline
    \begin{tikzpicture}[xscale=.1,yscale=.09]
    \draw[bend left,\BLUE](10:4)to(-0.5,2.5);
    \draw[ ]plot [smooth,tension=1] coordinates {(-120:6) (-95:4.5) (-60:6)};
    \draw[ ]plot [smooth,tension=1] coordinates {(110:5) (85:4) (70:5)};
    \draw[ ]plot [smooth,tension=1] coordinates {(30:5) (10:4) (-10:5)};
    \draw[ ]plot [smooth,tension=1] coordinates {(150:5) (170:4) (185:5)};
    \draw[\RED](-95:4.5)to[bend left=23](10:4)to[bend left](85:4)(-95:4.5)
    to[bend right](170:4)to[bend right](85:4);
    \draw[](-95:4.5)node{$.$}(85:4)node{$.$}
        (10:4)node{$.$}(170:4)node{$.$};
    \end{tikzpicture}
&
\begin{tikzpicture}[xscale=.1,yscale=.08]
    \draw(-7,4)node(a){$\k$}(7,4)node(b){0}(0,0)node(c){$\k$}
        (-7,-4)node(d){0}(7,-4)node(e){0};
        \draw(-3.5,3.7)node{\tiny{1}};
        \draw[->,>=stealth](d)to(a);\draw[->,>=stealth](c)to(b);\draw[->,>=stealth](a)to(c);
        \draw[->,>=stealth](e)to(c);\draw[->,>=stealth](c)to(d);\draw[->,>=stealth](b)to(e);
    \end{tikzpicture}
&
\begin{tikzpicture}[xscale=.1,yscale=.08]
    \draw(-7,4)node(a){$\k$}(7,4)node(b){0}(0,0)node(c){0}
        (-7,-4)node(d){0}(7,-4)node(e){0};
        \draw[->,>=stealth](a)to(b);\draw[->,>=stealth](b)to(c);\draw[->,>=stealth](c)to(a);
        \draw[->,>=stealth](d)to(c);\draw[->,>=stealth](c)to(e);\draw[->,>=stealth](e)to(d);
    \end{tikzpicture}
\\\hline
    \begin{tikzpicture}[xscale=.1,yscale=.09]
    \draw[bend left=-5,\BLUE](1.2,-1.5)to(-0.5,2.5);
    \draw[ ]plot [smooth,tension=1] coordinates {(-120:6) (-95:4.5) (-60:6)};
    \draw[ ]plot [smooth,tension=1] coordinates {(110:5) (85:4) (70:5)};
    \draw[ ]plot [smooth,tension=1] coordinates {(30:5) (10:4) (-10:5)};
    \draw[ ]plot [smooth,tension=1] coordinates {(150:5) (170:4) (185:5)};
    \draw[\RED](-95:4.5)to[bend left=23](10:4)to[bend left](85:4)(-95:4.5)
    to[bend right](170:4)to[bend right](85:4);
    \draw[](-95:4.5)node{$.$}(85:4)node{$.$}
        (10:4)node{$.$}(170:4)node{$.$};
    \end{tikzpicture}
&
\begin{tikzpicture}[xscale=.1,yscale=.08]
    \draw(-7,4)node(a){$\k$}(7,4)node(b){0}(0,0)node(c){$\k$}
        (-7,-4)node(d){0}(7,-4)node(e){$\k$};
        \draw(-3.5,3.7)node{\tiny{1}}(3.1,-3.7)node{\tiny{1}};
        \draw[->,>=stealth](d)to(a);\draw[->,>=stealth](c)to(b);\draw[->,>=stealth](a)to(c);
        \draw[->,>=stealth](e)to(c);\draw[->,>=stealth](c)to(d);\draw[->,>=stealth](b)to(e);
    \end{tikzpicture}
&
\begin{tikzpicture}[xscale=.1,yscale=.08]
    \draw(-7,4)node(a){$\k$}(7,4)node(b){0}(0,0)node(c){$\k$}
        (-7,-4)node(d){0}(7,-4)node(e){$\k$};
        \draw(-3.5,0.7)node{\tiny{1}}(3.1,-0.7)node{\tiny{1}};
        \draw[->,>=stealth](a)to(b);\draw[->,>=stealth](b)to(c);\draw[->,>=stealth](c)to(a);
        \draw[->,>=stealth](d)to(c);\draw[->,>=stealth](c)to(e);\draw[->,>=stealth](e)to(d);
    \end{tikzpicture}
\\\hline
    \begin{tikzpicture}[xscale=-.1,yscale=.09]
    \draw[bend left,\BLUE](10:4)to(-0.5,2.5);
    \draw[ ]plot [smooth,tension=1] coordinates {(-120:6) (-95:4.5) (-60:6)};
    \draw[ ]plot [smooth,tension=1] coordinates {(110:5) (85:4) (70:5)};
    \draw[ ]plot [smooth,tension=1] coordinates {(30:5) (10:4) (-10:5)};
    \draw[ ]plot [smooth,tension=1] coordinates {(150:5) (170:4) (185:5)};
    \draw[\RED](-95:4.5)to[bend left=23](10:4)to[bend left](85:4)(-95:4.5)
    to[bend right](170:4)to[bend right](85:4);
    \draw[](-95:4.5)node{$.$}(85:4)node{$.$}
        (10:4)node{$.$}(170:4)node{$.$};
    \end{tikzpicture}
&
\begin{tikzpicture}[xscale=.1,yscale=.08]
    \draw(-7,4)node(a){0}(7,4)node(b){$\k$}(0,0)node(c){$\k$}
        (-7,-4)node(d){0}(7,-4)node(e){0};
        \draw(3.1,3.7)node{\tiny{1}};
        \draw[->,>=stealth](d)to(a);\draw[->,>=stealth](c)to(b);\draw[->,>=stealth](a)to(c);
        \draw[->,>=stealth](e)to(c);\draw[->,>=stealth](c)to(d);\draw[->,>=stealth](b)to(e);
    \end{tikzpicture}
&
\begin{tikzpicture}[xscale=.1,yscale=.08]
    \draw(-7,4)node(a){0}(7,4)node(b){$\k$}(0,0)node(c){0}
        (-7,-4)node(d){0}(7,-4)node(e){0};
        \draw[->,>=stealth](a)to(b);\draw[->,>=stealth](b)to(c);\draw[->,>=stealth](c)to(a);
        \draw[->,>=stealth](d)to(c);\draw[->,>=stealth](c)to(e);\draw[->,>=stealth](e)to(d);
    \end{tikzpicture}
\\\hline
    \begin{tikzpicture}[xscale=-.1,yscale=.09]
    \draw[bend left=-5,\BLUE](1.2,-1.5)to(-0.5,2.5);
    \draw[ ]plot [smooth,tension=1] coordinates {(-120:6) (-95:4.5) (-60:6)};
    \draw[ ]plot [smooth,tension=1] coordinates {(110:5) (85:4) (70:5)};
    \draw[ ]plot [smooth,tension=1] coordinates {(30:5) (10:4) (-10:5)};
    \draw[ ]plot [smooth,tension=1] coordinates {(150:5) (170:4) (185:5)};
    \draw[\RED](-95:4.5)to[bend left=23](10:4)to[bend left](85:4)(-95:4.5)
    to[bend right](170:4)to[bend right](85:4);
    \draw[](-95:4.5)node{$.$}(85:4)node{$.$}
        (10:4)node{$.$}(170:4)node{$.$};
    \end{tikzpicture}
&
\begin{tikzpicture}[xscale=.1,yscale=.08]
    \draw(-7,4)node(a){0}(7,4)node(b){$\k$}(0,0)node(c){$\k$}
        (-7,-4)node(d){$\k$}(7,-4)node(e){0};
        \draw(-3.5,-3.7)node{\tiny{1}}(3.1,3.7)node{\tiny{1}};
        \draw[->,>=stealth](d)to(a);\draw[->,>=stealth](c)to(b);\draw[->,>=stealth](a)to(c);
        \draw[->,>=stealth](e)to(c);\draw[->,>=stealth](c)to(d);\draw[->,>=stealth](b)to(e);
    \end{tikzpicture}
&
\begin{tikzpicture}[xscale=.1,yscale=.08]
    \draw(-7,4)node(a){0}(7,4)node(b){$\k$}(0,0)node(c){$\k$}
        (-7,-4)node(d){$\k$}(7,-4)node(e){0};
        \draw(-3.5,-0.7)node{\tiny{1}}(3.5,0.7)node{\tiny{1}};
        \draw[->,>=stealth](a)to(b);\draw[->,>=stealth](b)to(c);\draw[->,>=stealth](c)to(a);
        \draw[->,>=stealth](d)to(c);\draw[->,>=stealth](c)to(e);\draw[->,>=stealth](e)to(d);
    \end{tikzpicture}
\\\hline
    \begin{tikzpicture}[xscale=.1,yscale=.09]
    \draw[bend left=-5,\BLUE](170:4)to(10:4);
    \draw[ ]plot [smooth,tension=1] coordinates {(-120:6) (-95:4.5) (-60:6)};
    \draw[ ]plot [smooth,tension=1] coordinates {(110:5) (85:4) (70:5)};
    \draw[ ]plot [smooth,tension=1] coordinates {(30:5) (10:4) (-10:5)};
    \draw[ ]plot [smooth,tension=1] coordinates {(150:5) (170:4) (185:5)};
    \draw[\RED](-95:4.5)to[bend left=23](10:4)to[bend left](85:4)(-95:4.5)
    to[bend right](170:4)to[bend right](85:4);
    \draw[](-95:4.5)node{$.$}(85:4)node{$.$}
        (10:4)node{$.$}(170:4)node{$.$};
    \end{tikzpicture}
&
\begin{tikzpicture}[xscale=.1,yscale=.08]
    \draw(-7,4)node(a){0}(7,4)node(b){$0$}(0,0)node(c){$\k$}
        (-7,-4)node(d){$0$}(7,-4)node(e){0};
        \draw[->,>=stealth](d)to(a);\draw[->,>=stealth](c)to(b);\draw[->,>=stealth](a)to(c);
        \draw[->,>=stealth](e)to(c);\draw[->,>=stealth](c)to(d);\draw[->,>=stealth](b)to(e);
    \end{tikzpicture}
&
\begin{tikzpicture}[xscale=.1,yscale=.12]
    \draw(-7,4)node(a){ }(7,4)node(b){ }(0,4)node(c){$M'=0$}
        (0,0)node(d){$V_j'=\delta_{ij}\k$};
    \end{tikzpicture}
\\\hline
\end{tabular}
\end{table}

\begin{table}[htbp]
\caption{The second type of $\lozenge$-flips}\label{table2}
\centering
\begin{tabular}{c||c|c}
&
    \begin{tikzpicture}[xscale=.35,yscale=.25,rotate=180]
    \draw[very thick]plot [smooth,tension=1] coordinates {(-110:5) (-95:4.5) (-70:5)};
    \draw[very thick]plot [smooth,tension=1] coordinates {(115:5) (85:4) (65:5)};
    \draw[\RED,thick](-95:4.5)to[bend left=60](85:4)
        (-95:4.5)to[bend right=60](85:4) (-95:4.5)to(0,-1);

    \draw[\RED, thick](-95:4.5)
    .. controls +(60:6) and +(100:7) .. (-95:4.5);
    \draw[<-,bend left=10,\BLUE,>=stealth]
        (-1.3,-3.5)to(-.6,-3.2);\draw(-1,-2.7)node{$_{\beta}$};
    \draw[<-,bend left=10,\BLUE,>=stealth]
        (.3,-3.2)to(.8,-3.5);\draw(1,-2.7)node{$_{\alpha}$};
    \draw[->,bend left=-10,\BLUE,>=stealth]
        (-.8,3.2)to(1.1,3.2);\draw(0.2,2.7)node{$_{\gamma}$};
    \draw[very thin](-95:4.5)node{$\bullet$} (85:4)node{$\bullet$}(0,-1)node{$\bullet$};
    \draw(-.27,-0.3)node[\RED,below]{$i$}(0.2,-3)node[\RED,below]{$\overline{i}$};
    \end{tikzpicture}
&
    \begin{tikzpicture}[xscale=.35,yscale=.25,rotate=0]
    \draw[very thick]plot [smooth,tension=1] coordinates {(-113:5.24) (-95:4.5) (-67:5.4)};
    \draw[very thick]plot [smooth,tension=1] coordinates {(110:5) (85:4) (70:5)};
    \draw[\RED,thick](-95:4.5)to[bend left=60](85:4)
        (-95:4.5)to[bend right=60](85:4) (-95:4.5)to(0,-1);

    \draw[\RED, thick](-95:4.5)
    .. controls +(60:6) and +(100:7) .. (-95:4.5);
    \draw[->,bend left=10,\BLUE,>=stealth]
        (-.6,-3.2)to(-1.3,-3.5);\draw(-1.25,-2.7)node{$\;_{_{\alpha^*}}$};
    \draw[->,bend left=10,\BLUE,>=stealth]
        (.8,-3.5)to(.3,-3.2);\draw(.85,-2.7)node{$\;_{_{\beta^*}}$};
    \draw[->,bend left=-10,\BLUE,>=stealth]
        (-.8,3.2)to(1.1,3.2);\draw(0.2,2.5)node{$_{[\beta\alpha]}$};
    \draw[very thin](-95:4.5)node{$\bullet$} (85:4)node{$\bullet$}(0,-1)node{$\bullet$};
    \end{tikzpicture}
\\\hline
&$\xymatrix@R=.8pc@C=1.5pc{
&i\ar[dr]|{\beta}\ar@{<-}[dl]|{\alpha}\ar@(ur,ul)[]|{\varepsilon_i}\\
\cdot\ar@{<-}[rr]|\gamma&&\cdot}$
&$\xymatrix@R=.8pc@C=1.5pc{
&i\ar@{<-}[dr]|{\beta^*}\ar[dl]|{\alpha^*}\ar@(ur,ul)[]|{\varepsilon_i}\\
\cdot\ar[rr]|{[\beta\alpha]}&&\cdot}$
\\\hline
\begin{tikzpicture}[xscale=.12,yscale=.1,rotate=180]
    \draw[\RED](-95:4.5)to[bend left=60](85:4)
        (-95:4.5)to[bend right=60](85:4);
    \draw[]plot [smooth,tension=1] coordinates {(-110:6) (-95:4.5) (-70:6)};
    \draw[]plot [smooth,tension=1] coordinates {(120:5) (85:4) (60:5)};
    \draw[very thin](-95:4.5)node{$.$} (85:4)node{$.$}(0,0)node{\Large{$.$}};
    \draw[bend left=-10,\BLUE](-1.1,2.7)to(1.5,2.7);
    \end{tikzpicture}
&
\begin{tikzpicture}[xscale=.1,yscale=.1]
    \draw(-7,-5)node(a){$\k$}(7,-5)node(b){$\k$}
        (0,0)node(c){0};
    \draw[<-,>=stealth](c)to(a);\draw[<-,>=stealth](a)to(b);\draw[<-,>=stealth](b)to(c);
    \draw[<-,>=stealth,dashed](120:2).. controls +(135:8) and +(60:8) .. (45:2);
    \draw(0,-5)node[below]{\tiny{1}};
    \end{tikzpicture}
&
\begin{tikzpicture}[xscale=.1,yscale=.1]
    \draw(-8,-5)node(a){$\k$}(8,-5)node(b){$\k$}
        (0,0)node(c){$\k^2$};
    \draw[<-,>=stealth](a)to(c);\draw[<-,>=stealth](b)to(a);\draw[<-,>=stealth](c)to(b);
    \draw[<-,>=stealth,dashed](120:2).. controls +(135:8) and +(60:8) .. (45:2);
    \draw(0,-5)node[below]{\tiny{0}};
    \draw(4,-1)node[right]{$\left(\begin{smallmatrix}1\\0\end{smallmatrix}\right)$};
    \draw(-4,-1)node[left]{$\left(\begin{smallmatrix}0&1\end{smallmatrix}\right)$};
    \draw(2,5)node[right]{$\left(\begin{smallmatrix}0&0\\1&1\end{smallmatrix}\right)$};
    \end{tikzpicture}
\\\hline
\begin{tikzpicture}[xscale=.12,yscale=.1,rotate=180]
    \draw[\RED](-95:4.5)to[bend left=60](85:4)
        (-95:4.5)to[bend right=60](85:4);
    \draw[]plot [smooth,tension=1] coordinates {(-110:6) (-95:4.5) (-70:6)};
    \draw[]plot [smooth,tension=1] coordinates {(120:5) (85:4) (60:5)};
    \draw[very thin](-95:4.5)node{$.$} (85:4)node{$.$}(0,0)node{\Large{$.$}};
    \draw[\BLUE]plot [smooth,tension=1] coordinates
        {(-1.1,2.7)(1,0)(-1.7,-2.7)};
    \end{tikzpicture}
&

\begin{tikzpicture}[xscale=.1,yscale=.1]
    \draw(-7,-5)node(a){$0$}(9,-5)node(b){$\k^2$}
        (0,0)node(c){$\k^2$};
    \draw[<-,>=stealth](c)to(a);\draw[<-,>=stealth](a)to(b);\draw[<-,>=stealth](b)to(c);
    \draw[<-,>=stealth,dashed](120:2).. controls +(135:8) and +(60:8) .. (45:2);
    \draw(4,-1)node[right]{$\left(\begin{smallmatrix}1&0\\0&1\end{smallmatrix}\right)$};
    \draw(2,5)node[right]{$\left(\begin{smallmatrix}0&0\\1&1\end{smallmatrix}\right)$};
    \end{tikzpicture}
&
\begin{tikzpicture}[xscale=.1,yscale=.1]
    \draw(-8,-5)node(a){$0$}(8,-5)node(b){$\k^2$}
        (0,0)node(c){$\k^2$};
    \draw[<-,>=stealth](a)to(c);\draw[<-,>=stealth](b)to(a);\draw[<-,>=stealth](c)to(b);
    \draw[<-,>=stealth,dashed](120:2).. controls +(135:8) and +(60:8) .. (45:2);
    \draw(4,-1)node[right]{$\left(\begin{smallmatrix}1&0\\0&1\end{smallmatrix}\right)$};
    \draw(2,5)node[right]{$\left(\begin{smallmatrix}0&0\\1&1\end{smallmatrix}\right)$};
    \end{tikzpicture}

\\\hline
\begin{tikzpicture}[xscale=.12,yscale=.1,rotate=180]
    \draw[\RED](-95:4.5)to[bend left=60](85:4)
        (-95:4.5)to[bend right=60](85:4);
    \draw[\BLUE]plot [smooth,tension=1] coordinates
        {(-1.1,2.7)(1,0)(-95:4.5)};
    \draw[very thin](-95:4.5)node{$.$} (85:4)node{$.$}(0,0)node{\Large{$.$}};
    \draw[]plot [smooth,tension=1] coordinates {(-110:6) (-95:4.5) (-70:6)};
    \draw[]plot [smooth,tension=1] coordinates {(120:5) (85:4) (60:5)};
    \end{tikzpicture}
&

\begin{tikzpicture}[xscale=.1,yscale=.1]
    \draw(-7,-5)node(a){$0$}(7,-5)node(b){$\k$}
        (0,0)node(c){0};
    \draw[<-,>=stealth](c)to(a);\draw[<-,>=stealth](a)to(b);\draw[<-,>=stealth](b)to(c);
    \draw[<-,>=stealth,dashed](120:2).. controls +(135:8) and +(60:8) .. (45:2);
    \end{tikzpicture}
&
\begin{tikzpicture}[xscale=.1,yscale=.1]
    \draw(-8,-5)node(a){$0$}(8,-5)node(b){$\k$}
        (0,0)node(c){$\k^2$};
    \draw[<-,>=stealth](a)to(c);\draw[<-,>=stealth](b)to(a);\draw[<-,>=stealth](c)to(b);
    \draw[<-,>=stealth,dashed](120:2).. controls +(135:8) and +(60:8) .. (45:2);
    \draw(4,-1)node[right]{$\left(\begin{smallmatrix}1\\0\end{smallmatrix}\right)$};
    \draw(2,5)node[right]{$\left(\begin{smallmatrix}0&0\\1&1\end{smallmatrix}\right)$};
    \end{tikzpicture}

\\\hline
\begin{tikzpicture}[xscale=.12,yscale=.1,rotate=180]
    \draw[\RED](-95:4.5)to[bend left=60](85:4)
        (-95:4.5)to[bend right=60](85:4);
    \draw[bend left=1,\BLUE](0,-.1)to(-2.2,-.1);
    \draw[]plot [smooth,tension=1] coordinates {(-110:6) (-95:4.5) (-70:6)};
    \draw[]plot [smooth,tension=1] coordinates {(120:5) (85:4) (60:5)};
    \draw[very thin](-95:4.5)node{$.$} (85:4)node{$.$}(0,0)node{\Large{$.$}}
        (-0.6,-.1)node[\BLUE,above]{\tiny{$\kappa$}};
    \end{tikzpicture}
&

\begin{tikzpicture}[xscale=.1,yscale=.1]
    \draw(-7,-5)node(a){$0$}(7,-5)node(b){$\k$}
        (0,0)node(c){$\k$};
    \draw(2,5)node[right]{$\kappa$};
    \draw[<-,>=stealth](c)to(a);\draw[<-,>=stealth](a)to(b);\draw[<-,>=stealth](b)to(c);
    \draw[<-,>=stealth,dashed](120:2).. controls +(135:8) and +(60:8) .. (45:2);
        \draw(4,-1)node[right]{1};
    \end{tikzpicture}
&
\begin{tikzpicture}[xscale=.1,yscale=.1]
    \draw(-8,-5)node(a){$0$}(8,-5)node(b){$\k$}
        (0,0)node(c){$\k$};
    \draw[<-,>=stealth](a)to(c);\draw[<-,>=stealth](b)to(a);\draw[<-,>=stealth](c)to(b);
    \draw[<-,>=stealth,dashed](120:2).. controls +(135:8) and +(60:8) .. (45:2);
    \draw(2,5)node[right]{$\kappa$};
        \draw(4,-1)node[right]{1};
    \end{tikzpicture}

\\\hline
\begin{tikzpicture}[xscale=.12,yscale=.1,rotate=0]
    \draw[\RED](-95:4.5)to[bend left=60](85:4)
        (-95:4.5)to[bend right=60](85:4)(-1,0)node[\BLUE]{\tiny{$\kappa$}};
    \draw[bend left=1,\BLUE](0,-.1)to(-95:4.5);
    \draw[]plot [smooth,tension=1] coordinates {(-110:6) (-95:4.5) (-70:6)};
    \draw[]plot [smooth,tension=1] coordinates {(120:5) (85:4) (60:5)};
    \draw[very thin](-95:4.5)node{$.$} (85:4)node{$.$}(0,0)node{\Large{$.$}};
    \end{tikzpicture}
&
\begin{tikzpicture}[xscale=.1,yscale=.1]
    \draw(-7,-5)node(a){$0$}(7,-5)node(b){$0$}
        (0,0)node(c){$\k$};
    \draw(2,5)node[right]{$\kappa$};
    \draw[<-,>=stealth](c)to(a);\draw[<-,>=stealth](a)to(b);\draw[<-,>=stealth](b)to(c);
    \draw[<-,>=stealth,dashed](120:2).. controls +(135:8) and +(60:8) .. (45:2);
    \end{tikzpicture}
&
\begin{tikzpicture}[xscale=.1,yscale=.12]
    \draw(-7,4)node(a){ }(7,4)node(b){ }(0,4)node(c){$M'=0$}
        (0,0)node(d){$V_j'=\delta_{ij}\k$}
        (0,-4)node(e){$V_{\varepsilon_i}'=1-\kappa$};
    \end{tikzpicture}
\\\hline
\end{tabular}
\end{table}
Let $(\gamma,\kappa)$ be a tagged curve in $\TC$ and $(M,V)=(M^\T_{(\gamma,\kappa)},V^\T_{(\gamma,\kappa)})$ be the corresponding decorated representation of $(Q^\T,W^\T)$, defined in Construct~\ref{cstr:dec}. For a vertex $i\in Q_0^\T$, construct $\mu_i(M,V)=(M',V')$ as follows, where we use $\hookrightarrow$ to denote the canonical inclusion and $\twoheadrightarrow$ the canonical projection.

If there is no dashed loop at $i$, the subquivers of $Q^\T$ and $\mu_i(Q^\T)$ consisting of all arrows adjacent to $i$ are shown in the second row in Table~\ref{table1}. Construct $\mu_i(M,V)=(M',V')$ as follows.
\begin{itemize}
  \item For any $j\neq i$, set $M'_j=M_j$ and $V'_j=V_j$.
  \item Define $$M'_i=\frac{\ker M_{\gamma_1}\oplus\ker M_{\gamma_2}}{\im \left(\begin{smallmatrix}M_{\beta_1}\\M_{\beta_2}\end{smallmatrix}\right)}\oplus\im M_{\gamma_1}\oplus\im M_{\gamma_2}\oplus\frac{\ker \left(\begin{smallmatrix}M_{\alpha_1}&M_{\alpha_2}\end{smallmatrix}\right)}{\im M_{\gamma_1}\oplus\im M_{\gamma_2}}\oplus V_i$$ and $$V'_i=\frac{\ker M_{\beta_1}\cap\ker M_{\beta_2}}{\ker M_{\beta_1}\cap\ker M_{\beta_2}\cap\left(\im M_{\alpha_1}+\im M_{\alpha_2}\right)}.$$
  \item For any arrow $a\in Q^\T_1$ not incident with $i$, set $M'_\alpha=M_\alpha$.
  \item For any arrow $\varepsilon\in Q^\T_2$, set $M'_\varepsilon=M_\varepsilon$ and $V'_\varepsilon=V_\varepsilon$.
  \item Define $M'_{[\beta_2\alpha_1]}=M_{\beta_2}M_{\alpha_1}$ and $M'_{[\beta_1\alpha_2]}=M_{\beta_1}M_{\alpha_2}$.
  \item The map $M'_{\alpha^\ast_x}:M'_i\to M'_{s(\alpha_x)}$, is given by the canonical inclusion $\im M_{\gamma_x}\hookrightarrow M_{s(\alpha_x)}$, and the composition
      \[\frac{\ker \left(\begin{smallmatrix}M_{\alpha_1}&M_{\alpha_2}\end{smallmatrix}\right)}{\im M_{\gamma_1}\oplus\im M_{\gamma_2}}\xrightarrow{a} \ker \left(\begin{smallmatrix}M_{\alpha_1}&M_{\alpha_2}\end{smallmatrix}\right)\hookrightarrow M_{s(\alpha_1)}\oplus M_{s(\alpha_2)} \twoheadrightarrow M_{s(\alpha_x)},\]
      where $a$ is a right inverse of the canonical projection $\ker \left(\begin{smallmatrix}M_{\alpha_1}&M_{\alpha_2}\end{smallmatrix}\right)\twoheadrightarrow\frac{\ker \left(\begin{smallmatrix}M_{\alpha_1}&M_{\alpha_2}\end{smallmatrix}\right)}{\im M_{\gamma_1}\oplus\im M_{\gamma_2}}$.
  \item The map $M'_{\beta^\ast_x}:M'_{t(\beta_x)}\to M'_i$, is given by the map $M_{t(\beta_x)}\xrightarrow{M_{\gamma_x}}\im M_{\gamma_x}$ and the composition
      \[M_{t(\beta_x)}\hookrightarrow M_{t(\beta_1)}\oplus M_{t(\beta_2)}\xrightarrow{b} \ker M_{\gamma_1}\oplus \ker M_{\gamma_2}\twoheadrightarrow\frac{\ker M_{\gamma_1}\oplus\ker M_{\gamma_2}}{\im \left(\begin{smallmatrix}M_{\beta_1}\\M_{\beta_2}\end{smallmatrix}\right)},\]
      where $b$ is a left inverse of the inclusion $\ker M_{\gamma_1}\oplus \ker M_{\gamma_2}\hookrightarrow M_{t(\beta_1)}\oplus M_{t(\beta_2)}$.
\end{itemize}
If there is a dashed loop $\varepsilon_i$ at $i$, the subquivers of $Q^\T$ and $\mu_i(Q^\T)$ consisting of all arrows adjacent to $i$ are shown in the second row in Table~\ref{table2}. Construct $\mu_i(M,V)=(M',V')$ as follows.
\begin{itemize}
  \item For any $j\neq i$, set $M'_j=M_j$ and $V'_j=V_j$.
  \item Define
  $$\begin{array}{c}
    M'_i=\frac{\ker M_{\gamma}}{\im M_{\beta}M_{\varepsilon_i}}\oplus\frac{\ker M_{\gamma}}{\im M_{\beta}(1-M_{\varepsilon_i})}\oplus\im (M_{\gamma})^{\oplus2}
    \oplus\frac{\ker M_{\varepsilon_i}M_{\alpha}\oplus \ker (1-M_{\varepsilon_i})M_{\alpha}}
        {\im M_{\gamma}}
      \oplus V_i\quad
    \end{array}
  $$
  and
  \[\begin{array}{c}V'_i=\frac{\ker M_{\beta}M_{\varepsilon_i}}{\ker M_{\beta}M_{\varepsilon_i}\cap\im M_{\varepsilon_i}M_\alpha}\oplus\frac{\ker M_{\beta}(1-M_{\varepsilon_i})}{\ker M_{\beta}(1-M_{\varepsilon_i})\cap\im (1-M_{\varepsilon_i})M_\alpha}.\end{array}\]
  \item The map $V'_{\varepsilon_i}$ is given by the identity on the first summand.
  \item For any arrow $a\in Q^\T_1$ not incident with $i$, set $M'_\alpha=M_\alpha$.
  \item For any arrow $\varepsilon\in Q^\T_2$ not incident with $i$, set $M'_\varepsilon=M_\varepsilon$ and $V'_\varepsilon=V_\varepsilon$.
  \item Define $M'_{[\beta\alpha]}=M_{\beta}M_{\varepsilon_i}M_{\alpha}$.
  \item The map $M'_{\alpha^\ast}:M'_i\to M'_{s(\alpha)}$ is given by the map $(\im M_{\gamma})^{\oplus2}\xrightarrow{\left(\begin{smallmatrix}\iota&\iota\end{smallmatrix}\right)} M_{s(\alpha_x)}$ and the composition
      $$\begin{array}{c}
        \frac{\ker M_{\varepsilon_i}M_{\alpha}\oplus \ker (1-M_{\varepsilon_i})M_{\alpha}}
        {\im M_{\gamma}}
        \xrightarrow{a} \ker M_{\varepsilon_i}M_{\alpha}\oplus\ker (1-M_{\varepsilon_i})M_{\alpha}\hookrightarrow M_{s(\alpha)}\oplus  M_{s(\alpha)},
      \end{array}$$
      where $\iota$ is the inclusion and $a$ is a right inverse of the projections
      $\ker M_{\varepsilon_i}M_{\alpha}\oplus\ker (1-M_{\varepsilon_i})M_{\alpha}\twoheadrightarrow \frac{\ker M_{\varepsilon_i}M_{\alpha}\oplus \ker (1-M_{\varepsilon_i})M_{\alpha}}
        {\im M_{\gamma}}$.
  \item The map $M'_{\beta^\ast}:M'_{t(\beta)}\to M'_i$ is given by the map $M_{t(\beta)}\xrightarrow{\left(\begin{smallmatrix}M_{\gamma}&M_{\gamma}\end{smallmatrix}\right)^t}(\im M_{\gamma})^{\oplus2}$ and the composition $$M_{(t(\beta))}\xrightarrow{b}\ker M_{\gamma}\xrightarrow{\left(\begin{smallmatrix}\pi,\pi'\end{smallmatrix}\right)}\frac{\ker M_{\gamma}}{\im M_{\beta}M_{\varepsilon_i}}\oplus\frac{\ker M_{\gamma}}{\im M_{\beta}(1-M_{\varepsilon_i})},$$ where $b$ is a left inverse of the inclusion $\ker M_{\gamma}\hookrightarrow M_{(t(\beta))}$ and $\pi, \pi'$ are the projections.
  \item The map $M'_{\varepsilon_i}$ is given by the identity on $\frac{\ker M_{\gamma}}{\im M_{\beta}M_{\varepsilon_i}}$, the identity on $\frac{\ker M_{\varepsilon_i}M_{\alpha}}{\im M_{\gamma}}$, the map $\left(\begin{smallmatrix}1&0\\0&0\end{smallmatrix}\right)\colon\im (M_{\gamma})^{\oplus2}\to\im (M_{\gamma})^{\oplus2}$ and $V_{\varepsilon_i}:V_i\to V_i$.
\end{itemize}
By \cite[Corollary~10.12]{DWZ}, $\mu_i(M,V)$ is a decorated representation of $\mu_i(Q^\T,W^\T)$ in the both cases.

\begin{remark}
The first mutation formula above for $(M,V)$ is an explicit version of DWZ's mutation
of decorated representations \cite{DWZ}.
The second mutation formula above is the composition of two DWZ's mutations
of decorated representations.
\end{remark}

\section*{acknowledgements}
	The authors would like to thank the algebra group of NTNU, especially Aslak Bakke Buan and Idun Reiten for their enthusiastic help.
	Further thanks to Daniel Labardini-Fragoso, Bangming Deng, Jie Xiao and Bin Zhu for their helpful discussions and comments. We would like to thank the anonymous referees for
	many comments and suggestions which improved the presentation of the paper.

\end{document}